\documentclass[reqno]{amsart}
\usepackage{amsmath,amsthm,amssymb,amsfonts}
\usepackage{mathrsfs}
\numberwithin{equation}{section}
\usepackage{color}
\usepackage{enumerate}
\usepackage{multirow} 
\newtheoremstyle{mystyle}
{}
{}
{\normalfont}
{ }
{\bfseries}
{}
{10pt}
{ }
\theoremstyle{mystyle}
\newtheorem{theorem}{Theorem}

\newtheorem{lemma}{Lemma}
\newtheorem{remark}{Remark}

\usepackage{mathtools}
\usepackage[pdftex]{graphicx}

\def\tr{\mathop{\rm tr}\nolimits}
\def\vec{\mathop{\rm vec}\nolimits}
\def\vech{\mathop{\rm vech}\nolimits}
\def\det{\mathop{\rm det}\nolimits}
\def\rank{\mathop{\rm rank}\nolimits}

\newcommand{\dd}{\mathrm d}
\newcommand{\E}{\mathbb E}
\newcommand{\PP}{\mathbb P}
\allowdisplaybreaks[4]
\usepackage[foot]{amsaddr}
\usepackage[top=3cm, bottom=3cm, left=2.5cm,right=2.5cm]{geometry}
\usepackage{comment}
\large
\title[SEM for diffusion processes based on high-frequency data]{Structural equation modeling with latent variables for diffusion processes based on high-frequency data}
\author[S. Kusano]{Shogo Kusano $^{1}$}
\author[M. Uchida]{Masayuki Uchida $^{1,2}$}
\address{$^{1}$Graduate School of Engineering Science, Osaka University}
\address{$^{2}$Center for Mathematical Modeling and Data Science (MMDS), Osaka University and JST CREST}
\begin{document}
\begin{abstract}
\fontsize{8pt}{10pt}\selectfont
We consider structural equation modeling (SEM) with latent variables for diffusion processes based on high-frequency data. We derive the quasi-likelihood estimators for parameters in the SEM. The goodness-of-fit test based on the quasi-likelihood ratio is proposed. Furthermore, the asymptotic properties of our proposed estimators are examined.
\end{abstract}
\maketitle
\section{Introduction}
\fontsize{10pt}{16pt}\selectfont
We consider structural equation modeling  (SEM) with latent variables for diffusion processes. 
The stochastic process $X_{1,t}$ is defined as the following factor model:
\begin{align}
	X_{1,t}=\Lambda_{x_1,m}\xi_{m,t}+\delta_{m,t}\label{X},
\end{align}
where $m \in \mathbb{N}$ is a model number, 
$\{X_{1,t}\}_{t\geq 0}$ is a $p_1$-dimensional observable vector process, $\{\xi_{m,t}\}_{t\geq 0}$ is a $k_{1,m}$-dimensional latent common factor vector process, $\{\delta_{m,t}\}_{t\geq 0}$ is a $p_{1}$-dimensional latent unique factor vector process, $\Lambda_{x_1,m}\in\mathbb{R}^{p_1\times k_{1,m}}$ is a constant loading matrix, $p_1$ is not zero, $p_1$ and $k_{1,m}$ are fixed, and $k_{1,m}\leq p_1$. The stochastic process $X_{2,t}$ is defined by the factor model as follows:
\begin{align}
	X_{2,t}=\Lambda_{x_2,m}\eta_{m,t}+\varepsilon_{m,t}\label{Y},
\end{align}
where $\{X_{2,t}\}_{t\geq 0}$ is a $p_2$-dimensional observable vector process, $\{\eta_{m,t}\}_{t\geq 0}$ is a $k_{2,m}$-dimensional latent common factor vector process, $\{\varepsilon_{m,t}\}_{t\geq 0}$ is a $p_{2}$-dimensional latent unique factor vector process, $\Lambda_{x_2,m}\in\mathbb{R}^{p_2\times k_{2,m}}$ is a constant loading matrix, $p_2$ is not zero, $p_2$ and $k_{2,m}$ are fixed, and $k_{2,m}\leq p_2$.
Furthermore, the relationship between $\eta_{m,t}$ and  $\xi_{m,t}$ is expressed as follows:
\begin{align}
	\eta_{m,t}=B_{0,m}\eta_{m,t}+\Gamma_{m}\xi_{m,t}+\zeta_{m,t}\label{eta},
\end{align}
where $\{\zeta_{m,t}\}_{t\geq 0}$ is a $k_{2,m}$-dimensional latent unique factor vector process, $B_{0,m}\in\mathbb{R}^{k_{2,m}\times k_{2,m}}$ is a constant loading matrix, whose diagonal elements are zero, and $\Gamma_{m}\in\mathbb{R}^{k_{2,m}\times k_{1,m}}$ is a constant loading matrix. Assume that $\{\xi_{m,t}\}_{t\geq 0}$ satisfies the following stochastic differential equation:
\begin{align*}
	\begin{cases}
	\dd \xi_{m,t}=B_{1,m}(\xi_{m,t})\dd t+S_{1,m}\dd W_{1,t}\quad(t\in [0,T]),\\
	\xi_{m,0}=c_{1,m},
	\end{cases}
\end{align*}
where $B_{1,m}:\mathbb{R}^{k_{1,m}}\rightarrow\mathbb{R}^{k_{1,m}}$, $S_{1,m}\in\mathbb{R}^{k_{1,m}\times r_1}$, $c_{1,m}\in\mathbb{R}^{k_{1,m}}$ and $W_{1,t}$ is an $r_{1}$-dimensional standard Wiener process, $\{\delta_{m,t}\}_{t\geq 0}$ 
is defined as the following stochastic differential equation:
\begin{align*}
	\begin{cases}
	\dd \delta_{m,t}=B_{2,m}(\delta_{m,t})\dd t+S_{2,m}\dd W_{2,t}\quad (t\in [0,T]),\\
	\delta_{m,0}=c_{2,m},
	\end{cases}
\end{align*}
where $B_{2,m}:\mathbb{R}^{p_1}\rightarrow\mathbb{R}^{p_1}$, $S_{2,m}\in\mathbb{R}^{p_1\times r_2}$, $c_{2,m}\in\mathbb{R}^{p_1}$ and $W_{2,t}$ is an $r_{2}$-dimensional standard Wiener process, $\{\varepsilon_{m,t}\}_{t\geq 0}$ 
satisfies   the following stochastic differential equation:
\begin{align*}
	\begin{cases}
	\dd \varepsilon_{m,t}=B_{3,m}(\varepsilon_{m,t})\dd t+S_{3,m}\dd W_{3,t}\quad (t\in [0,T]),\\
	\varepsilon_{m,0}=c_{3,m},
	\end{cases}
\end{align*}
where $B_{3,m}:\mathbb{R}^{p_2}\rightarrow\mathbb{R}^{p_2}$, $S_{3,m}\in\mathbb{R}^{p_2\times r_3}$, $c_{3,m}\in\mathbb{R}^{p_2}$ and $W_{3,t}$ is an $r_{3}$-dimensional standard Wiener process, and $\{\zeta_{m,t}\}_{t\geq 0}$ 
is defined by the stochastic differential equation as follows:
\begin{align*}
	\begin{cases}
	\dd \zeta_{m,t}=B_{4,m}(\zeta_{m,t})\dd t+S_{4,m}\dd W_{4,t}\quad (t\in [0,T]),\\
	\zeta_{m,0}=c_{4,m},
	\end{cases}
\end{align*}
where $B_{4,m}:\mathbb{R}^{k_{2,m}}\rightarrow\mathbb{R}^{k_{2,m}}$, $S_{4,m}\in\mathbb{R}^{k_{2,m}\times r_4}$, $c_{4,m}\in\mathbb{R}^{k_{2,m}}$ and $W_{4,t}$ is an $r_{4}$-dimensional standard Wiener process. We assume that $W_{1,t}$, $W_{2,t}$, $W_{3,t}$ and $W_{4,t}$ are independent. Set $X_t=(X_{1,t}^{\top},X_{2,t}^{\top})^{\top}$. $\{X_{t_{i}^n}\}_{i=1}^n$ are discrete observations, where $t_{i}^n= ih_n$ and $T=nh_n$, and $p_1$, $p_2$, $k_{1,m}$ and $k_{2,m}$ are independent of $n$.

SEM is a method that describes the relationships between latent variables that cannot
be observed. SEM has been used in various fields, e.g., behavioral science, economics, engineering, and medical science. For example, in psychology, SEM is used to investigate the relationships between intelligence and motivation. Note that intelligence and motivation are latent variables. J{\"o}reskog \cite{Joreskog(1970)} proposed this method by combining path analysis and confirmatory factor analysis. For path analysis and confirmatory factor analysis, see, e.g., Mueller \cite{Mueller(1999)}. Several models have been proposed to formulate SEM. In this paper, we consider the model defined by
(\ref{X}), (\ref{Y}) and (\ref{eta}), which is called the LInear Structural RELations (LISREL) model 
( J{\"o}reskog \cite{Joreskog(1972)}). 
The LISREL model is one of the most well-known models in SEM and can be expressed complex relationships between latent variables. For more information on the LISREL model, see, e.g., Everitt \cite{Everitt(1984)}. 
Note that SEM is a confirmatory analysis method rather than an exploratory analysis method. 
SEM is used to specify the model from a theoretical viewpoint of each research field before conducting the analysis. This is the difference between confirmatory analysis methods and exploratory analysis methods such as exploratory factor analysis. 
In behavioral science, factor analysis for time series data has been  actively studied; see, e.g., Molenaar \cite{Molenaar(1985)} and Pena and box \cite{Pena(1987)}. 
Moreover, Czi{\'a}ky \cite{Cziraky(2004)} proposed
SEM for time series data called dynamic structural equation model with latent variables (DSEM). Asparouhov et.al. \cite{Asparouhov} studied the more general DSEM model.

Recently, we can easily obtain high-frequency data such as stock price data and life-log data (blood pressure and EEG, etc.) thanks to the development of measuring devices,
and statistical inference for stochastic differential equations based on high-frequency data
has been developed. For parametric estimation of diffusion processes based on high-frequency data,  
see for example, 
Yoshida \cite{Yoshida(1992)}, 
Genon-Catalot and Jacod \cite{Genon(1993)},
Kessler \cite{kessler(1997)}, 
Uchida and Yoshida \cite{Uchi-Yoshi(2012)} 
and references therein.
In financial econometrics, the factor model for high-frequency data 
has been extensively researched. In this field, parameters and the number of factors are estimated by using principal component analysis for high-frequency data (A{\"i}t-Sahalia and Xiu \cite{Ait(2019)}) when the factor is latent; see, e.g., A{\"i}t-Sahalia and Xiu \cite{Ait(2017)}. However, these studies are based on high dimensionality. 
For a low-dimensional model, the estimator does not have consistency; see Bai \cite{Bai(2003)}. On the other hand, Kusano and Uchida \cite{Kusano(2022)} proposed classical factor analysis for diffusion processes. Their method works well for a low-dimensional model. 
However, to the best of our knowledge, 
there have been few studies of SEM for high-frequency data.
Oud and Jansen \cite{Oud(2000)} and Driver et.al. \cite{Driver(2017)} considered SEM
for stochastic differential equations. 
Note that their model differs from the model in this paper.
In the field of causal inference, Hansen and Sokol \cite{Hansen(2014)} studied SEM for stochastic differential equations. However, their model is the path analysis model, so that their method cannot describe the relationships between latent variables. Note that these studies do not assume that the data is sampled at high-frequency. 
On the other hand, we propose SEM for diffusion processes based on high-frequency data.

In this paper, we assume that the volatilities for diffusion processes and loading matrices are not time-variant but constant to simplify the discussion.
We leave  for future work the discussion on the model 
where the volatilities for diffusion processes and loading matrices are time-varying.
Furthermore, we do not discuss a high-dimensional case. Bai \cite{bai(2012)} studied the asymptotic properties of factor analysis based on the maximum likelihood estimation for a high-dimension model. We expect that our quasi-likelihood method will also work well for a high-dimension model. The investigation is future work.

The paper is organized as follows. In Section 2, notation and assumptions are introduced. 
In Section 3, we study SEM for diffusion processes in the ergodic and non-ergodic cases. 
First, the asymptotic properties of the realized covariance are examined. 
Next, we derive the quasi-likelihood estimators for parameters in the SEM.
It is shown that the estimators have good asymptotic properties.
Furthermore, we propose the goodness-of-fit test based on the quasi-likelihood ratio and investigate the asymptotic properties. 
In Section 4, we give examples and simulation studies 
to investigate the asymptotic performance of the results described in Section 3.
Section 5 is devoted to the proofs of theorems given in Section 3.

\section{Notation and assumptions}
For any vector $v$, $|v|=\sqrt{\tr{vv^\top}}$ and $v^{(i)}$ is the $i$-th element of $v$,
where $\top$ denotes the transpose.
For any matrix A, $\|A\|=\sqrt{\tr{AA^\top}}$, and $A_{ij}$ is the $(i,j)$-th element of $A$. $\mathbb{I}_{p}$ denotes the identity matrix of size $p$. Define $O_{p\times q}$ as the $p\times q$ zero matrix. For any symmetric matrix $A\in\mathbb{R}^{p\times p}$, $\vec A$,  $\vech A$ and $\mathbb{D}_{p}$ 
denote the vectorization of $A$, the half-vectorization of $A$ and the $p^2\times\bar{p}$ duplication matrix,
respectively. Here, $\vec{A}=\mathbb{D}_{p}\vech{A}$ and $\bar{p}=p(p+1)/2$; see, e.g., Harville \cite{Harville(1998)}.  
For any matrix $A$, the Moore-Penrose inverse of $A$ is denoted by $A^{+}$. 
If $A$ is a positive definite matrix, we write $A>0$.
For any positive sequence $u_{n}$, 
$R:[0,\infty)\times \mathbb{R}^d\rightarrow \mathbb{R}$
is defined as $|R({u_{n}},x)|\leq u_{n}C(1+|x|)^C$ for some $C>0$. Let $\mathscr{F}^{n}_{i}=\sigma(W_{1,s},W_{2,s},W_{3,s},W_{4,s},s\leq t_{i}^n)$ for $i=1,\cdots n$. Let $C^{k}_{\uparrow}(\mathbb R^d)$ be 
the space of all functions $f$ satisfying the following conditions:
\begin{itemize}
    \item[(i)] $f$ is continuously differentiable with respect to $x\in \mathbb{R}^d$ up to order $k$. 
    \item[(ii)] $f$ and all its derivatives are of polynomial growth in $x\in \mathbb{R}^d$, i.e., 
    $g$ is
of polynomial growth in $x\in \mathbb{R}^d$ if $\displaystyle g(x)=R(1,x)$. 
\end{itemize}
$N_{p}(\mu,\Sigma)$ represents the $p$-dimensional normal random variable with mean $\mu\in\mathbb{R}^p$ and covariance matrix $\Sigma\in\mathbb{R}^{p\times p}$. 
Let $\chi^2_{r}$ be  the random variable which has 
the chi-squared distribution with $r$ degrees of freedom. $\chi^2_{r}(\alpha)$ denotes an upper $\alpha$ point of the chi-squared distribution with $r$ degrees of freedom, where $0\leq\alpha\leq 1$. The symbols $\stackrel{P}{\longrightarrow}$ and $\stackrel{d}{\longrightarrow}$ express convergence in probability and convergence in distribution, respectively.
Let 
    $\Sigma_{\xi\xi,m}=S_{1,m}S_{1,m}^{\top}$,
    $\Sigma_{\delta\delta,m}=S_{2,m}S_{2,m}^{\top}$,
    $\Sigma_{\varepsilon\varepsilon,m}=S_{3,m}S_{3,m}^{\top}$,
    $\Sigma_{\zeta\zeta,m}=S_{4,m}S_{4,m}^{\top}$ and $\Psi_{m}=\mathbb{I}_{k_{2,m}}-B_{0,m}$.
Furthermore, we make the following assumptions.
\begin{enumerate}
	\renewcommand{\labelenumi}{{\textbf{[A1]}}}
	\item 
	\begin{enumerate}
	\item
	There exists a constant $C>0$ such that for any $x,y\in\mathbb R^{k_{1,m}}$, 
	\begin{align*}
	|B_{1,m}(x)-B_{1,m}(y)|\leq C|x-y|.
	\end{align*}
	\item For all $\ell\geq 0$, 
	$\displaystyle\sup_t\E\bigl[|\xi_{m,t}|^{\ell}\bigr]<\infty$.
	\item $B_{1,m}\in C^{4}_{\uparrow}(\mathbb R^{k_{1,m}})$.
	\end{enumerate}
\end{enumerate}
\begin{enumerate}
	\renewcommand{\labelenumi}{{\textbf{[A2]}}}
	\item The diffusion process  $\xi_{m,t}$ is ergodic with its invariant measure $\pi_{\xi_m}$: For any $\pi_{\xi_m}$-integrable function $g$, it holds that
	\begin{align*}
	\frac{1}{T}\int_{0}^{T}{g(\xi_{m,t})dt}\overset{P}{\longrightarrow}\int g(x)\pi_{\xi_{m}}(dx)
	\end{align*}
	as $T\longrightarrow\infty$. 
\end{enumerate}
\begin{enumerate}
	\renewcommand{\labelenumi}{{\textbf{[B1]}}}
	\item 
	\begin{enumerate}
	\item There exists a constant $C>0$ such that for any $x,y\in\mathbb R^{p_1}$, 
	\begin{align*}
	|B_{2,m}(x)-B_{2,m}(y)|\leq C|x-y|.
	\end{align*}
	\item For all $\ell\geq 0$, 
	$\displaystyle\sup_t\E\bigl[|\delta_{m,t}|^{\ell}\bigr]<\infty$.
	\item $B_{2,m}\in C^{4}_{\uparrow}(\mathbb R^{p_1})$.
	\end{enumerate}
\end{enumerate}
\begin{enumerate}
	\renewcommand{\labelenumi}{{\textbf{[B2]}}}
	\item $\Sigma_{\delta\delta,m}>0$.
\end{enumerate}
\begin{enumerate}
	\renewcommand{\labelenumi}{{\textbf{[B3]}}}
	\item The diffusion process  $\delta_{m,t}$ is ergodic with its invariant measure $\pi_{\delta_m}$: For any $\pi_{\delta_m}$-integrable function $g$, it holds that
	\begin{align*}
		\frac{1}{T}\int_{0}^{T}{g(\delta_{m,t})dt}\overset{P}{\longrightarrow}\int g(x)\pi_{\delta_{m}}(dx)
	\end{align*}
	as $T\longrightarrow\infty$. 
\end{enumerate}
\begin{enumerate}
	\renewcommand{\labelenumi}{{\textbf{[C1]}}}
	\item \begin{enumerate}
	\item There exists a constant $C>0$ such that for any $x,y\in\mathbb R^{p_2}$, 
	\begin{align*}
	|B_{3,m}(x)-B_{3,m}(y)|\leq C|x-y|.
	\end{align*}
	\item For all $\ell\geq 0$, 
	$\displaystyle\sup_t\E\bigl[|\varepsilon_{m,t}|^{\ell}\bigr]<\infty$.
	\item $B_{3,m}\in C^{4}_{\uparrow}(\mathbb R^{p_2})$.
	\end{enumerate}
\end{enumerate}
\begin{enumerate}
	\renewcommand{\labelenumi}{{\textbf{[C2]}}}
	\item $\Sigma_{\varepsilon\varepsilon,m}>0$.
\end{enumerate}
\begin{enumerate}
	\renewcommand{\labelenumi}{{\textbf{[C3]}}}
	\item The diffusion process  $\varepsilon_{m,t}$ is ergodic with its invariant measure $\pi_{\varepsilon_{m}}$: For any $\pi_{\varepsilon_{m}}$-integrable function $g$, it holds that
	\begin{align*}
		\frac{1}{T}\int_{0}^{T}{g(\varepsilon_{m,t})dt}\overset{P}{\longrightarrow}\int g(x)\pi_{\varepsilon_{m}}(dx)
	\end{align*}
	as $T\longrightarrow\infty$. 
\end{enumerate}
\begin{enumerate}
	\renewcommand{\labelenumi}{{\textbf{[D1]}}}
	\item 
	\begin{enumerate}
	\item There exists a constant $C>0$ such that for any $x,y\in\mathbb R^{k_{2,m}}$, 
	\begin{align*}
	|B_{4,m}(x)-B_{4,m}(y)|\leq C|x-y|.
	\end{align*}
	\item For all $\ell\geq 0$, $\displaystyle\sup_t\E\bigl[|\zeta_{m,t}|^{\ell}\bigr]<\infty$.
	\item $B_{4,m}\in C^{4}_{\uparrow}(\mathbb R^{k_{2,m}})$.
\end{enumerate}
\end{enumerate}
\begin{enumerate}
	\renewcommand{\labelenumi}{{\textbf{[D2]}}}
	\item The diffusion process  $\zeta_{m,t}$ is ergodic with its invariant measure $\pi_{\zeta_m}$: For any $\pi_{\zeta_m}$-integrable function $g$, it holds that
	\begin{align*}
	\frac{1}{T}\int_{0}^{T}{g(\zeta_{m,t})dt}\overset{P}{\longrightarrow}\int g(x)\pi_{\zeta_m}(dx)
	\end{align*}
	as $T\longrightarrow\infty$. 
\end{enumerate}
\begin{enumerate}
	\renewcommand{\labelenumi}{{\textbf{[E]}}}
	\item $\Psi_{m}$ is non-singular.
\end{enumerate}
\begin{enumerate}
	\renewcommand{\labelenumi}{{\textbf{[F]}}}
	\item $\rank{\Lambda_{x_1,m}}=k_{1,m}$.
\end{enumerate}
\begin{remark}
$\bf{[A1]}$, $\bf{[B1]}$, $\bf{[C1]}$ and $\bf{[D1]}$ are the standard assumptions for ergodic diffusion processes. For example, see Kessler \cite{kessler(1997)}. $\bf{[B2]}$, $\bf{[C2]}$, $\bf{[E]}$ and $\bf{[F]}$ imply that $\Sigma_{m}(\theta_{m})$ is non-singular. For details, see Lemma \ref{Sigmaposlemma}.
\end{remark}
\section{Main theorems}
\subsection{Ergodic case}
In the LISREL model, we will estimate $\Lambda_{x_1,m}$, $\Lambda_{x_2,m}$, $\Gamma_{m}$, $\Psi_{m}$, $\Sigma_{\xi\xi,m}$, $\Sigma_{\delta\delta,m}$, $\Sigma_{\varepsilon\varepsilon,m}$ and $\Sigma_{\zeta\zeta,m}$. Note that some of these elements are assumed to be known in order to satisfy an identifiability condition for parameter estimation.
See Remark \ref{identification} for constraints on the parameter and the identifiability condition.
Set the parameter as $\theta_{m}\in\Theta_{m}$, where $\Theta_{m}\subset\mathbb{R}^{q_{m}}$ is a convex compact space. $\theta_{m}$ includes only unknown and non-duplicated elements of $\Lambda_{x_1,m}$, $\Lambda_{x_2,m}$, $\Gamma_{m}$, $\Psi_{m}$, $\Sigma_{\xi\xi,m}$, $\Sigma_{\delta\delta,m}$, $\Sigma_{\varepsilon\varepsilon,m}$ and $\Sigma_{\zeta\zeta,m}$. Define the covariance structure as 
\begin{align}
	\Sigma_{m}(\theta_{m})=
	\begin{pmatrix}
		\Sigma_{X_1X_1,m}(\theta_{m}) & \Sigma_{X_1X_2,m}(\theta_{m})\\
		\Sigma_{X_1X_2,m}(\theta_{m})^{\top} & \Sigma_{X_2X_2,m}(\theta_{m})
	\end{pmatrix},\label{sigmam}
\end{align}
where 
\begin{align*}
    \qquad\qquad\qquad\Sigma_{X_1X_1,m}(\theta_{m})&=\Lambda_{x_1,m}\Sigma_{\xi\xi,m}\Lambda_{x_1,m}^{\top}+\Sigma_{\delta\delta,m},\\
    \Sigma_{X_1X_2,m}(\theta_{m})&=\Lambda_{x_1,m}\Sigma_{\xi\xi,m}\Gamma_{m}^{\top}\Psi^{-1\top}_{m}\Lambda_{x_2,m}^{\top},\\
    \Sigma_{X_2X_2,m}(\theta_{m})&=\Lambda_{x_2,m}\Psi^{-1}_{m}(\Gamma_m\Sigma_{\xi\xi,m}\Gamma^{\top}_m+\Sigma_{\zeta\zeta,m})\Psi^{-1\top}_{m}\Lambda_{x_2,m}^{\top}+\Sigma_{\varepsilon\varepsilon,m}.
\end{align*}
To estimate (\ref{sigmam}), we use the realized covariance as follows:
\begin{align*}
	Q_{XX}=\frac{1}{T}\sum_{i=1}^{n}(X_{t_{i}^n}-X_{t_{i-1}^n})(X_{t_{i}^n}-X_{t_{i-1}^n})^{\top}.
\end{align*}
Let
\begin{align*}
	W_{m}(\theta_{m})=2\mathbb{D}_{p}^{+\top}(\Sigma_{m}(\theta_{m})\otimes\Sigma_{m}(\theta_{m}))\mathbb{D}_{p}^{+}.
\end{align*}
For the realized covariance, the following theorem holds.
\begin{theorem}\label{Qztheorem}
Under 
$\bf{[A1]}$-$\bf{[A2]}$, $\bf{[B1]}$-$\bf{[B3]}$, $\bf{[C1]}$-$\bf{[C3]}$, $\bf{[D1]}$-$\bf{[D2]}$, $\bf{[E]}$ and $\bf{[F]}$, 
as $h_n\longrightarrow0$ and $nh_n\longrightarrow\infty$, 
\begin{align*}
	Q_{XX}\stackrel{P_{\theta_{m}\ } }{\longrightarrow} \Sigma_{m}(\theta_{m}).
\end{align*}
In addition, as $nh_n^2\longrightarrow0$, 
\begin{align*}
\sqrt{n}(\vech{Q_{XX}}-\vech{\Sigma_{m}(\theta_{m})})\stackrel{d}{\longrightarrow} N_{\bar{p}}(0,W_{m}(\theta_{m})).
\end{align*}
\begin{remark}
This result is similar to the asymptotic result of the sample variance matrix for the i.i.d model, 
see, e.g., Browne \cite{brown(1974)}.
\end{remark}
\end{theorem}
Next, we consider the parameter estimation. Set the following quasi-log-likelihood function:
\begin{align}
	\ell_{m,n}(\theta_{m})=-\frac{pn}{2}\log(2\pi)-\frac{pn}{2}\log h_n-\frac{n}{2}\log\det{\Sigma_{m}(\theta_{m})}-\frac{n}{2}\tr{\Bigl\{\Sigma_{m}(\theta_{m})^{-1}Q_{XX}\Bigr\}} \label{quasilog}
\end{align}
as $Q_{XX}>0$.
Let
\begin{align*}
	\ell_{n}(\Sigma)=-\frac{pn}{2}\log(2\pi)-\frac{pn}{2}\log h_n-\frac{n}{2}\log\det\Sigma-\frac{n}{2}\tr\bigl\{\Sigma^{-1}Q_{XX}\bigr\},
\end{align*}
where $\Sigma\in\mathbb{R}^{p\times p}$ is a positive definite matrix. Note that $\ell_{n}(\Sigma)$ has a maximum value
\begin{align*}
	-\frac{pn}{2}\log(2\pi)-\frac{pn}{2}\log h_n-\frac{n}{2}\log\det Q_{XX}-\frac{np}{2}
\end{align*}
at $\Sigma=Q_{XX}$ as $Q_{XX}>0$. Define the following function:
\begin{align}  
	\begin{split}
	F(Q_{XX},\Sigma_{m}(\theta_{m}))&=-\frac{2}{n}\ell_{m,n}(\theta_{m})+\frac{2}{n}\left\{-\frac{pn}{2}\log(2\pi)-\frac{pn}{2}\log h_n-\frac{n}{2}\log\det Q_{XX}-\frac{np}{2}\right\}\\
	&=\log\det\Sigma_{m}(\theta_{m})-\log\det Q_{XX}+\tr{\Bigl\{\Sigma_{m}(\theta_{m})^{-1}Q_{XX}\Bigr\}}-p.\label{F1}
	\end{split}
\end{align}
From Theorem 1 in Shapiro \cite{Shapiro(1985)}, (\ref{F1}) is rewritten as
\begin{align*}
	F(Q_{XX},\Sigma_{m}(\theta_{m}))=(\vech{Q_{XX}}-\vech{\Sigma_{m}(\theta_{m})})^{\top}V(Q_{XX},\Sigma_{m}(\theta_{m}))(\vech{Q_{XX}}-\vech{\Sigma_{m}(\theta_{m}}))
\end{align*}
as $Q_{XX}>0$, where 
\begin{align*}
	V(Q_{XX},\Sigma_{m}(\theta_{m}))
	&=\mathbb{D}_{p}^{+\top}\int_{0}^{1}\int_{0}^{1}\lambda_{2}(\Sigma_{m}(\theta_{m})+\lambda_{1}\lambda_{2}(Q_{XX}-\Sigma_{m}(\theta_m)))^{-1}\\
	&\qquad\qquad\qquad\qquad\otimes(\Sigma_{m}(\theta_{m})+\lambda_{1}\lambda_{2}(Q_{XX}-\Sigma_{m}(\theta_m)))^{-1}d\lambda_{1}d\lambda_{2}\mathbb{D}_{p}^{+}
\end{align*}
as $Q_{XX}>0$. Furthermore, set the following function:
\begin{align*}
	\tilde{F}(Q_{XX},\Sigma_{m}(\theta_{m}))=(\vech{Q_{XX}}-\vech{\Sigma_{m}(\theta_{m})})^{\top}\tilde{V}(Q_{XX},\Sigma_{m}(\theta_{m}))(\vech{Q_{XX}}-\vech{\Sigma_{m}(\theta_{m}})),
\end{align*}
where 
\begin{align*}
	\tilde{V}(Q_{XX},\Sigma_{m}(\theta_{m}))=\left\{
	\begin{array}{ll}
	V(Q_{XX},\Sigma_{m}(\theta_{m})),
	& (Q_{XX} \mbox{ is non-singular}), \\
	\mathbb{I}_{\bar{p}},  & (Q_{XX}\mbox{ is singular}).
	\end{array}\right. 
\end{align*}
The contrast function is given by
\begin{align*}
	\mathbb{F}_{m,n}(\theta_{m})=\tilde{F}(Q_{XX},\Sigma_{m}(\theta_{m})).
\end{align*}
The minimum contrast estimator $\hat{\theta}_{m,n}$ is defined as
\begin{align*}
	\mathbb{F}_{m,n}(\hat{\theta}_{m,n})=\inf_{\theta_{m}\in\Theta_{m}}\mathbb{F}_{m,n}(\theta_{m}).
\end{align*}
\begin{remark}
We derive the quasi-log-likelihood function (\ref{quasilog}). 
Let $\Xi_{m,t}$ be the Euler-Maruyama approximation of $\xi_{m,t}$.
One has 
\begin{align*}
	\Xi_{m,t_{i}^n}-\Xi_{m,t_{i-1}^n}&=B_{1,m}(\Xi_{m,t_{i-1}^n})h_n+S_{1,m}(W_{1,t_{i}^n}-W_{1,t_{i-1}^n}).
\end{align*}
In the same way, set $\Delta_{m,t}$, $E_{m,t}$ and $Z_{m,t}$ as the Euler-Maruyama approximation of $\delta_{m,t}$, $\varepsilon_{m,t}$ and $\zeta_{m,t}$, respectively. We get
\begin{align*}
	\Delta_{m,t_{i}^n}-\Delta_{m,t_{i-1}^n}&=B_{2,m}(\Delta_{m,t_{i-1}^n})h_n+S_{2,m}(W_{2,t_{i}^n}-W_{2,t_{i-1}^n}),\\
	E_{m,t_{i}^n}-E_{m,t_{i-1}^n}&=B_{3,m}(E_{m,t_{i-1}^n})h_n+S_{3,m}(W_{3,t_{i}^n}-W_{3,t_{i-1}^n}),\\
	Z_{m,t_{i}^n}-Z_{m,t_{i-1}^n}&=B_{4,m}(Z_{m,t_{i-1}^n})h_n+S_{4,m}(W_{4,t_{i}^n}-W_{4,t_{i-1}^n}).
\end{align*}
Note that it holds that
\begin{align*}
	\begin{split}
	&\Lambda_{x_1,m}(\Xi_{m,t_{i}^n}-\Xi_{m,t_{i-1}^n})+\Delta_{m,t_{i}^n}-\Delta_{m,t_{i-1}^n}\\
	&\qquad=\Lambda_{x_1,m}S_{1,m}(W_{1,t_{i}^n}-W_{1,t_{i-1}^n})+S_{2,m}(W_{2,t_{i}^n}-W_{2,t_{i-1}^n})+R(\Xi_{m,t_{i-1}^n},h_n)+R(\Delta_{m,t_{i-1}^n},h_n)
	\end{split}
\end{align*}
from $\bf{[A1]}$(c) and $\bf{[B1]}$(c). If we set $\bar{X}_{1,t}$ as an approximation to $X_{1,t}$, we obtain
\begin{align}
	\bar{X}_{1,t_{i}^n}-\bar{X}_{1,t_{i-1}^n}=\Lambda_{x_1,m}S_{1,m}(W_{1,t_{i}^n}-W_{1,t_{i-1}^n})+S_{2,m}(W_{2,t_{i}^n}-W_{2,t_{i-1}^n}) \label{X1app}
\end{align}
from (\ref{X}).
In the same way, since it follows from \textbf{[A1]} (c), \textbf{[C1]} (c) and \textbf{[D1]} (c) that
\begin{align*}
	\begin{split}
	&\Lambda_{x_2,m}\Psi_{m}^{-1}\Gamma_{m}(\Xi_{m,t_{i}^n}-\Xi_{m,t_{i-1}^n})+\Lambda_{x_2,m}\Psi_{m}^{-1}(Z_{m,t_{i}^n}-Z_{m,t_{i-1}^n})+E_{m,t_{i}^n}-E_{m,t_{i-1}^n}\\
	&\qquad\qquad\qquad=\Lambda_{x_2,m}\Psi_{m}^{-1}\Gamma_{m}S_{1,m}(W_{1,t_{i}^n}-W_{1,t_{i-1}^n})+\Lambda_{x_2,m}\Psi_{m}^{-1}S_{4,m}(W_{4,t_{i}^n}-W_{4,t_{i-1}^n})\\
	&\qquad\qquad\qquad\qquad\quad+S_{3,m}(W_{3,t_{i}^n}-W_{3,t_{i-1}^n})+R(\Xi_{m,t_{i-1}^n},h_n)+R(Z_{m,t_{i-1}^n},h_n)+R(E_{m,t_{i-1}^n},h_n), \\
	\end{split}
\end{align*}
(\ref{Y}) and (\ref{eta}) imply that
\begin{align}
	\begin{split}
	\bar{X}_{2,t_{i}^n}-\bar{X}_{2,t_{i-1}^n}&=\Lambda_{x_2,m}\Psi_{m}^{-1}\Gamma_{m}S_{1,m}(W_{1,t_{i}^n}-W_{1,t_{i-1}^n})\\
	&\qquad\quad+\Lambda_{x_{2},m}\Psi_{m}^{-1}S_{4,m}(W_{4,t_{i}^n}-W_{4,t_{i-1}^n})+S_{3,m}(W_{3,t_{i}^n}-W_{3,t_{i-1}^n}),\label{X2app}
	\end{split}
\end{align}
where $\bar{X}_{2,t}$ denotes an approximation to $X_{2,t}$. 
Set $\bar{X}_{t}=(\bar{X}_{1,t}^{\top},\bar{X}_{2,t}^{\top})^{\top}$. 
We see from (\ref{X1app}) and (\ref{X2app}) that
\begin{align*}
	\begin{split}
	\bar{X}_{t_{i}^n}-\bar{X}_{t_{i-1}^n}
	&=\begin{pmatrix}
		\Lambda_{x_1,m}S_{1,m} & O_{p_1\times r_1}\\
		O_{p_2\times r_1} &\Lambda_{x_2,m}\Psi_{m}^{-1}\Gamma_{m}S_{1,m}
	\end{pmatrix}
	\begin{pmatrix}
	    W_{1,t_{i}^n}-W_{1,t_{i-1}^n}\\
		W_{1,t_{i}^n}-W_{1,t_{i-1}^n}
	\end{pmatrix}+\begin{pmatrix}
		S_{2,m} \\
		O_{p_2\times r_2} 
	\end{pmatrix}(
	W_{2,t_{i}^n}-W_{2,t_{i-1}^n})\\
	&\qquad\quad+\begin{pmatrix}
	O_{p_1\times r_3} \\
	S_{3,m}
	\end{pmatrix}
	(W_{3,t_{i}^n}-W_{3,t_{i-1}^n})
	+\begin{pmatrix}
	O_{p_1\times r_4}\\
    \Lambda_{x_2,m}\Psi_{m}^{-1}S_{4,m}
	\end{pmatrix}(
	W_{4,t_{i}^n}-W_{4,t_{i-1}^n})
	\end{split}
\end{align*}
as an approximation to $X_{t}$. The property of the Brownian motion implies that
\begin{align}
	\begin{pmatrix}
	W_{1,t_{i}^n}-W_{1,t_{i-1}^n}\\
	W_{1,t_{i}^n}-W_{1,t_{i-1}^n}
	\end{pmatrix}\sim N_{2r_1}\left(0,h_n\begin{pmatrix}
	\mathbb{I}_{r_{1}} & \mathbb{I}_{r_{1}}\\
	\mathbb{I}_{r_{1}} & \mathbb{I}_{r_{1}}
    \end{pmatrix}\right).\label{W1-1}
\end{align}
A standard computation implies that
\begin{align*}
	&\begin{pmatrix}
	\Lambda_{x_1,m}S_{1,m} & O_{p_1\times r_1}\\
	O_{p_2\times r_1} &\Lambda_{x_2,m}\Psi_{m}^{-1}\Gamma_{m}S_{1,m}
	\end{pmatrix}
	\begin{pmatrix}
	\mathbb{I}_{r_{1}} & \mathbb{I}_{r_{1}}\\
	\mathbb{I}_{r_{1}} & \mathbb{I}_{r_{1}}
	\end{pmatrix}
	\begin{pmatrix}
	\Lambda_{x_1,m}S_{1,m} & O_{p_1\times r_1}\\
	O_{p_2\times r_1} &\Lambda_{x_2,m}\Psi_{m}^{-1}\Gamma_{m}S_{1,m}
	\end{pmatrix}^{\top}\\
	&\qquad\qquad\qquad\qquad\qquad=\begin{pmatrix}
	\Lambda_{x_1,m}\Sigma_{\xi\xi,m}\Lambda_{x_1,m}^{\top} & \Lambda_{x_1,m}\Sigma_{\xi\xi,m}\Gamma_{m}^{\top}\Psi_{m}^{-1\top}\Lambda_{x_2,m}^{\top}\\
	\Lambda_{x_2,m}\Psi_{m}^{-1}\Gamma_{m}\Sigma_{\xi\xi,m}\Lambda_{x_1,m}^{\top}
	& \Lambda_{x_2,m}\Psi_{m}^{-1}\Gamma_{m}\Sigma_{\xi\xi,m}\Gamma_{m}^{\top}\Psi_{m}^{-1\top}\Lambda_{x_2,m}^{\top}
    \end{pmatrix},
\end{align*}
so that one gets
\begin{align*}
	\begin{split}
	&\begin{pmatrix}
	\Lambda_{x_1,m}S_{1,m} & O_{p_1\times r_1}\\
	O_{p_2\times r_1} &\Lambda_{x_2,m}\Psi_{m}^{-1}\Gamma_{m}S_{1,m}
	\end{pmatrix}
	\begin{pmatrix}
	W_{1,t_{i}^n}-W_{1,t_{i-1}^n}\\
	W_{1,t_{i}^n}-W_{1,t_{i-1}^n}
	\end{pmatrix}\\
	&\qquad\qquad\qquad\sim N_{p}\left(0,h_n\begin{pmatrix}
	\Lambda_{x_1,m}\Sigma_{\xi\xi,m}\Lambda_{x_1,m}^{\top}
	& \Lambda_{x_1,m}\Sigma_{\xi\xi,m}\Gamma_{m}^{\top}\Psi_{m}^{-1\top}\Lambda_{x_2,m}^{\top}\\
	\Lambda_{x_2,m}\Psi_{m}^{-1}\Gamma_{m}\Sigma_{\xi\xi,m}\Lambda_{x_1,m}^{\top}
	&\Lambda_{x_2,m}\Psi_{m}^{-1}\Gamma_{m}\Sigma_{\xi\xi,m}\Gamma_{m}^{\top}\Psi_{m}^{-1\top}\Lambda_{x_2,m}^{\top}
	\end{pmatrix}\right)
	\end{split}
\end{align*}
from (\ref{W1-1}). By an analogous manner, we have
\begin{align*}
	\begin{pmatrix}
	S_{2,m} \\
	O_{p_2\times r_2} 
	\end{pmatrix}(
	W_{2,t_{i}^n}-W_{2,t_{i-1}^n})&\sim N_{p}\left(0,h_n\begin{pmatrix}
	\Sigma_{\delta\delta,m} & O_{p_1\times p_2}\\
	O_{p_2\times p_1} & O_{p_2\times p_2}
	\end{pmatrix}\right),\\
	\begin{pmatrix}
	O_{p_1\times r_3} \\
	S_{3,m}
	\end{pmatrix}
	(W_{3,t_{i}^n}-W_{3,t_{i-1}^n})
	&\sim N_{p}\left(0,h_n\begin{pmatrix}
	O_{p_1\times p_1} & O_{p_1\times p_2}\\
	O_{p_2\times p_1} & \Sigma_{\varepsilon\varepsilon,m}
	\end{pmatrix}\right),\\
	\begin{split}
	\qquad\qquad\begin{pmatrix}
	O_{p_1\times r_4} \\
	\Lambda_{x_2,m}\Psi_{m}^{-1}S_{4,m}
	\end{pmatrix}(
	W_{4,t_{i}^n}-W_{4,t_{i-1}^n})&\sim N_{p}\left(0,h_n\begin{pmatrix}
	O_{p_1\times p_1} & O_{p_1\times p_2}\\
	O_{p_2\times p_1}&\Lambda_{x_2,m}\Psi_{m}^{-1}\Sigma_{\zeta\zeta,m}\Psi_{m}^{-1\top}\Lambda_{x_2,m}^{\top} 
	\end{pmatrix}\right).
	\end{split}
\end{align*}
Therefore, since $W_{1,t}$, $W_{2,t}$, $W_{3,t}$ and $W_{4,t}$ are independent, it follows that 
\begin{align*}
	\bar{X}_{t_{i}^n}-\bar{X}_{t_{i-1}^n}\sim N_{p}\bigl(0,h_n\Sigma_{m}(\theta_{m})\bigr).
\end{align*}
Hence, one has the following joint probability density function of $(\bar{X}_{t_{i}^n})_{0\leq i\leq n}$: 
\begin{align*}
	\prod_{i=1}^{n}\frac{1}{(2\pi)^{\frac{p}{2}}\det{\bigl\{h_n\Sigma_{m}(\theta_{m})}\bigr\}^{\frac{1}{2}}}\exp{\left\{-\frac{1}{2h_n}(\bar{x}_{t_{i}^n}-\bar{x}_{t_{i-1}^n})^{\top}\Sigma_{m}(\theta_{m})^{-1}(\bar{x}_{t_{i}^n}-\bar{x}_{t_{i-1}^n})\right\}}.
\end{align*}
Set the quasi-likelihood as follows:
\begin{align*}
	L_{m,n}(\theta_{m})=\prod_{i=1}^{n}\frac{1}{(2\pi )^{\frac{p}{2}}\det{\bigl\{h_n\Sigma_{m}(\theta_{m})}\bigr\}^{\frac{1}{2}}}\exp{\left\{-\frac{1}{2h_n}(X_{t_{i}^n}-X_{t_{i-1}^n})^{\top}\Sigma_{m}(\theta_{m})^{-1}(X_{t_{i}^n}-X_{t_{i-1}^n})\right\}}.
\end{align*}
Since
\begin{align*}
	&\quad\log L_{m,n}(\theta_{m})\\
	&=\sum_{i=1}^n\left\{-\frac{p}{2}\log(2\pi)-\frac{p}{2}\log h_n
	-\frac{1}{2}\log\det{\Sigma_{m}(\theta_{m})}-\frac{1}{2h_n}(X_{t_{i}^n}-X_{t_{i-1}^n})^{\top}\Sigma_{m}(\theta_{m})^{-1}(X_{t_{i}^n}-X_{t_{i-1}^n})\right\}\\
	&=-\frac{pn}{2}\log(2\pi)-\frac{pn}{2}\log h_n-\frac{n}{2}\log\det{\Sigma_{m}(\theta_{m})}-\frac{1}{2h_n}\sum_{i=1}^n\tr{\left\{\Sigma_{m}(\theta_{m})^{-1}(X_{t_{i}^n}-X_{t_{i-1}^n})(X_{t_{i}^n}-X_{t_{i-1}^n})^{\top}\right\}}\\
	&=-\frac{pn}{2}\log(2\pi)-\frac{pn}{2}\log h_n-\frac{n}{2}\log\det{\Sigma_{m}(\theta_{m})}-\frac{n}{2}\tr{\Bigl\{\Sigma_{m}(\theta_{m})^{-1}Q_{XX}\Bigr\}},
\end{align*}
we obtain the quasi-log likelihood function (\ref{quasilog}).
	\end{remark}
Let $\theta_{m,0}$ be the true parameter and
\begin{align*}
    \Delta_{m}=\frac{\partial^2}{\partial\theta_{m}^2}\vech{\Sigma_{m}(\theta_{m,0})}.
\end{align*}
Furthermore, we make the following assumptions.
\begin{enumerate}
	\renewcommand{\labelenumi}{{\textbf{[G]}}}
		\item $\Sigma_{m}(\theta_{m})=\Sigma_{m}(\tilde{\theta}_{m})\Longrightarrow \theta_{m}=\tilde{\theta}_{m}$.
\end{enumerate}
\begin{enumerate}
	\renewcommand{\labelenumi}{{\textbf{[H]}}}\item $\rank{\Delta_{m}}=q_m$.
\end{enumerate}
\begin{remark}\label{identification}
Assumption $\bf{[G]}$ is an identifiability condition for parameter estimation 
and implies the consistency of the minimum contrast estimator $\hat{\theta}_{m,n}$. 
Like the factor model, 
the LISREL model does not have the identifiability condition for parameter estimation
when the parameters  are unconstrained.
To satisfy $\bf{[G]}$, some parameters may be fixed to 0 or 1, 
or some parameters are assumed to be the same value as other parameters. 
These constraints are determined from the theoretical viewpoint of each research field,
see Section 4 for an example of a model that satisfies $\bf{[G]}$. Unfortunately, in the LISREL model, 
simple sufficient conditions for $\bf{[G]}$ are not known. For the identification problem, e.g., see Everitt \cite{Everitt(1984)}. Assumption $\bf{[H]}$ implies that $\Delta_{m}^{\top}W_{m}(\theta_{m,0})^{-1}\Delta_{m}$ is non-singular, see Lemma \ref{Sigmaposlemma}.
\end{remark}

For the minimum contrast estimator, we obtain the following theorem.
\begin{theorem}\label{thetatheorem}
Under 
$\bf{[A1]}$-$\bf{[A2]}$, $\bf{[B1]}$-$\bf{[B3]}$, $\bf{[C1]}$-$\bf{[C3]}$, $\bf{[D1]}$-$\bf{[D2]}$, $\bf{[E]}$, $\bf{[F]}$, 
$\bf{[G]}$ and $\bf{[H]}$,
as $h_n\longrightarrow0$ and $nh_n\longrightarrow\infty$, 
\begin{align*}
	\hat{\theta}_{m,n}\stackrel{P_{\theta_{m,0}\ } }{\longrightarrow}\theta_{m,0}.
\end{align*}
In addition, as $nh_n^2\longrightarrow0$,
\begin{align*}
	\sqrt{n}(\hat{\theta}_{m,n}-\theta_{m,0})\stackrel{d}{\longrightarrow}N_{q_m}\bigl(0,(\Delta_{m}^{\top}W_{m}(\theta_{m,0})^{-1}\Delta_{m})^{-1}\bigr).
\end{align*}
\end{theorem}
Next, we consider the goodness-of-fit test. 
The statistical hypothesis test is as follows:
\begin{align*}
	\left\{
	\begin{array}{ll}
	H_0: \Sigma_{m}(\theta_{m})=\Sigma_{m^{*}}(\theta_{m^*}),\\
	H_1: \Sigma_{m}(\theta_{m})\neq\Sigma_{m^{*}}(\theta_{m^*}),
	\end{array}
	\right.
\end{align*}
where $m^{*}\in\mathbb{N}$ is a model number. Set
\begin{align*}
	L_{n}(\Sigma)=\prod_{i=1}^{n}\frac{1}{(2\pi)^{\frac{p}{2}}\det{(h_n\Sigma)}^{\frac{1}{2}}}\exp{\left\{-\frac{1}{2h_n}(X_{t_{i}^n}-X_{t_{i-1}^n})^{\top}\Sigma^{-1}(X_{t_{i}^n}-X_{t_{i-1}^n})\right\}},
\end{align*}
where $\Sigma\in\mathbb{R}^{p\times p}$ is a positive definite matrix. 
Since
\begin{align*}
    L_{n}(\Sigma_{m^*}(\theta_{m^{*}}))=L_{m^*,n}(\theta_{m^{*}}),
\end{align*}
the quasi-likelihood ratio $\lambda_{m^{*},n}$ is defined as
\begin{align*}
	\lambda_{m^{*},n}=\frac{\max_{\theta_{m^{*}}\in\Theta_{m^{*}}}L_{m^*,n}(\theta_{m^{*}})}{\max_{\Sigma>0}L_{n}(\Sigma)},
\end{align*}
where $\Sigma\in\mathbb{R}^{p\times p}$ is a positive definite matrix. It follows that
\begin{align*}
	-2\log\lambda_{m^{*},n}&=-2\max_{\theta_{m^{*}}\in\Theta_{m^{*}}}\ell_{m^*,n}(\theta_{m^{*}})+2\max_{\Sigma>0}\ell_{n}(\Sigma)\\
	&=-2\left\{-\frac{pn}{2}\log(2\pi)-\frac{pn}{2}\log h_n-\frac{n}{2}\log\det{\Sigma_{m^{*}}(\hat{\theta}_{m^{*},n})}
	-\frac{n}{2}\tr{\Bigl\{\Sigma_{m^{*}}(\hat{\theta}_{m^{*},n})^{-1}Q_{XX}\Bigr\}}\right\}\\
	&\quad +2\left\{-\frac{pn}{2}\log(2\pi)-\frac{pn}{2}\log h_n-\frac{n}{2}\log\det{Q_{XX}}-\frac{np}{2}\right\}\\
	&=n\left\{\log\det{\Sigma_{m^{*}}(\hat{\theta}_{m^{*},n})}-\log\det{Q_{XX}+\tr{\Bigl\{\Sigma_{m^{*}}(\hat{\theta}_{m^{*},n})^{-1}Q_{XX}\Bigr\}}-p}\right\}\\
	&=n\mathbb{F}_{m^*,n}(\hat{\theta}_{m^{*},n})
\end{align*}
as $Q_{XX}>0$. The quasi-likelihood ratio test statistic is given by
\begin{align*}
	\mathbb{T}_{m^*,n}=n\mathbb{F}_{m^*,n}(\hat{\theta}_{m^{*},n}).
\end{align*}
The asymptotic result of the test statistic $\mathbb{T}_{m^*,n}$ is as follows.
\begin{theorem}
Under 
$\bf{[A1]}$-$\bf{[A2]}$, $\bf{[B1]}$-$\bf{[B3]}$, $\bf{[C1]}$-$\bf{[C3]}$, $\bf{[D1]}$-$\bf{[D2]}$, 
$\bf{[E]}$, $\bf{[F]}$, $\bf{[G]}$ and $\bf{[H]}$,
as $h_n\longrightarrow0$ , $nh_n\longrightarrow\infty$ and $nh_n^2\longrightarrow 0$,
\begin{align*}
	\mathbb{T}_{m^*,n}\stackrel{d}{\longrightarrow}\chi^2_{\bar{p}-q_{m^*}}
\end{align*}
under $H_0$.
\end{theorem}
From Theorem 3, we can construct the test of asymptotic significance level $\alpha\in(0,1)$. Set the rejection region as
\begin{align*}
	\bigl\{t_{m^*,n}>\chi^2_{\bar{p}-q_{m^*}}(\alpha)\bigr\},
\end{align*}
where $t_{m^*,n}$ is the observed value of the test statistic $\mathbb{T}_{m^*,n}$. 

Finally, we investigate the consistency of the test.
Let
\begin{align*}
	\mathbb{U}_{m^*}(\theta_{m^*})=F(\Sigma_{m}(\theta_{m,0}),\Sigma_{m^*}(\theta_{m^*})).
\end{align*}
$\bar{\theta}_{m^*}$ is defined as
\begin{align*}
    \mathbb{U}_{m^*}(\bar{\theta}_{m^*})=\inf_{\theta_{m^*}\in\Theta_{m^*}} \mathbb{U}_{m^*}(\theta_{m^*}).
\end{align*}
In addition, we make the following assumption:
\begin{enumerate}
	\renewcommand{\labelenumi}{{\textbf{[I]}}}
	\item $\mathbb{U}_{m^*}(\theta_{m^*})=\mathbb{U}_{m^*}(\tilde{\theta}_{m^*})\Longrightarrow \theta_{m^*}=\tilde{\theta}_{m^*}$.
\end{enumerate}
\begin{remark}
Assumption $\textbf{[I]}$ implies that $\hat{\theta}_{m^*,n}\stackrel{P}{\longrightarrow}\bar{\theta}_{m^*}$ under $H_1$, see Lemma \ref{starconslemma}. 
\end{remark}

We have the following theorem.
\begin{theorem}
Under 
$\bf{[A1]}$-$\bf{[A2]}$, $\bf{[B1]}$-$\bf{[B3]}$, $\bf{[C1]}$-$\bf{[C3]}$, $\bf{[D1]}$-$\bf{[D2]}$, 
$\bf{[E]}$, $\bf{[F]}$ and $\bf{[I]}$,
as $h_n\longrightarrow0$ and $nh_n\longrightarrow\infty$, 
\begin{align*}
	\PP\left(\mathbb{T}_{m^*,n}>\chi^2_{\bar{p}-q_{m^*}}(\alpha)\right)\stackrel{}{\longrightarrow}1
\end{align*}
under $H_1$.
\end{theorem}
\begin{remark}
The goodness-of-fit test has several problems. For example, if the tests with the significance level $\alpha$ are used repeatedly, the overall significance level is not $\alpha$. See, e.g., Bentler and Bonett \cite{Bentler(1980)} for problems with 
the goodness-of-fit test. However, the goodness-of-fit test is one of 
the most popular methods for model evaluation in SEM; 
see, e.g., Mcdonald \cite{McDonald(2002)}. Thus, we consider only the goodness-of-fit test as a model evaluation method in this paper and leave the other methods for future work.
\end{remark}
\subsection{Non-ergodic case}
We investigate the non-ergodic case, where $\bf{[A2]}$, $\bf{[B3]}$, $\bf{[C3]}$ and $\bf{[D2]}$ are not assumed and $T$ is fix. In the non-ergodic case, the following results similar to the ergodic case hold.
\begin{theorem}\label{Qtheoremnon}
Under 
$\bf{[A1]}$, $\bf{[B1]}$-$\bf{[B2]}$, $\bf{[C1]}$-$\bf{[C2]}$, $\bf{[D1]}$, $\bf{[E]}$ and $\bf{[F]}$,
as $h_n\longrightarrow0$,
\begin{align*}
	Q_{XX}\stackrel{P_{\theta_{m}}\ }{\longrightarrow} \Sigma_{m}(\theta_{m})
\end{align*}
and
\begin{align*}
    \sqrt{n}(\vech{Q_{XX}}-\vech{\Sigma_{m}(\theta_{m})})\stackrel{d}{\longrightarrow} 
    N_{\bar{p}}(0,W_{m}(\theta_{m})). 
\end{align*}
\end{theorem}
\begin{theorem}\label{thetatheoremnon}
Under 
$\bf{[A1]}$, $\bf{[B1]}$-$\bf{[B2]}$, $\bf{[C1]}$-$\bf{[C2]}$, $\bf{[D1]}$, $\bf{[E]}$, $\bf{[F]}$, $\bf{[G]}$ and $\bf{[H]}$,
as $h_n\longrightarrow0$,
\begin{align*}
	\hat{\theta}_{m,n}\stackrel{P_{\theta_{m,0}}\ }{\longrightarrow} \theta_{m,0}
\end{align*}
and
\begin{align*}
	\sqrt{n}(\hat{\theta}_{m,n}-\theta_{m,0})\stackrel{d}{\longrightarrow}
	N_{q_m}\bigl(0,(\Delta_{m}^{\top}W_{m}(\theta_{m,0})^{-1}\Delta_{m})^{-1}\bigr).
\end{align*}
\end{theorem}
\begin{theorem}\label{chitheoremnon}
Under 
$\bf{[A1]}$, $\bf{[B1]}$-$\bf{[B2]}$, $\bf{[C1]}$-$\bf{[C2]}$, $\bf{[D1]}$, $\bf{[E]}$, $\bf{[F]}$, $\bf{[G]}$ and $\bf{[H]}$,
as $h_n\longrightarrow0$,
\begin{align*}
	\mathbb{T}_{m^*,n}\stackrel{d}{\longrightarrow}\chi^2_{\bar{p}-q_{m^*}}
\end{align*}
under $H_0$.
\end{theorem}
\begin{theorem}\label{testtheoremnon}
Under 
$\bf{[A1]}$, $\bf{[B1]}$-$\bf{[B2]}$, $\bf{[C1]}$-$\bf{[C2]}$, $\bf{[D1]}$, $\bf{[E]}$, $\bf{[F]}$ and $\bf{[I]}$,
as $h_n\longrightarrow0$, 
\begin{align*}
	\PP\left(\mathbb{T}_{m^*,n}>\chi^2_{\bar{p}-q_{m^*}}(\alpha)\right)\stackrel{}{\longrightarrow}1
\end{align*}
under $H_1$.
\end{theorem}

\section{Examples and simulation results}
\subsection{True model}
Set $k_{1,m}=2$ and $k_{2,m}=1$. The stochastic process $X_{1,t}$ is defined as
the following factor model:
\begin{align*}
	X_{1,t}=\Lambda_{x_1,m}\xi_{m,t}+\delta_{m,t},
\end{align*}
where $\{X_{1,t}\}_{t\geq 0}$ is a four-dimensional observable vector process, 
$\{\xi_{m,t}\}_{t\geq 0}$ is a two-dimensional latent common factor vector process, 
$\{\delta_{m,t}\}_{t\geq 0}$ is a four-dimensional latent unique
factor vector process and
\begin{align*}
	\Lambda_{x_1,m}=
	\begin{pmatrix}
	1 & (\Lambda_{x_1,m})_{21} & 0 & 0\\
	0 & 0 & 1 &(\Lambda_{x_1,m})_{42}
	\end{pmatrix}^{\top}\in\mathbb{R}^{4\times 2},
\end{align*}
where $(\Lambda_{x_1,m})_{21}$ and $(\Lambda_{x_1,m})_{42}$ are not zero. 
The stochastic process $X_{2,t}$ is defined by the factor model as follows:
\begin{align*}
	X_{2,t}=\Lambda_{x_2,m}\eta_{m,t}+\varepsilon_{m,t},
\end{align*}
where $\{X_{2,t}\}_{t\geq 0}$ is a two-dimensional observable vector process, $\{\eta_{m,t}\}_{t\geq 0}$ is a one-dimensional latent common factor vector process, $\{\varepsilon_{m,t}\}_{t\geq 0}$ is a two-dimensional latent unique factor vector process and 
\begin{align*}
	\Lambda_{x_2,m}=\bigl(
		1,\ (\Lambda_{x_{2},m})_{21}\bigr)
	^{\top}\in\mathbb{R}^{2\times 1},
\end{align*}
where $(\Lambda_{x_2,m})_{21}$ is not zero.
Furthermore, the relationship between $\eta_{m,t}$ and $\xi_{m,t}$ is expressed as follows:
\begin{align*}
	\eta_{m,t}=\Gamma_{m}\xi_{m,t}+\zeta_{m,t}
\end{align*}
where $\{\zeta_{m,t}\}_{t\geq 0}$ is a one-dimensional latent unique factor vector process and $\Gamma_{m}\in\mathbb{R}^{1\times2}$ is not a zero matrix. $\{\xi_{m,t}\}_{t\geq 0}$ satisfies the following two-dimensional OU process:
\begin{align*}
	\begin{cases}
	\dd \xi_{m,t}=-(A_{1,m}\xi_{m,t}-\mu_{1,m})\dd t+S_{1,m}
	\dd W_{1,t}\quad (t\in [0,T]),\\
	\xi_{m,0}=c_{1,m},
	\end{cases}
\end{align*}
where $A_{1,m}\in\mathbb{R}^{2\times 2}$, $\mu_{1,m}\in\mathbb{R}^{2}$, $S_{1,m}\in\mathbb{R}^{2\times 2}$ s.t. $\Sigma_{\xi\xi,m}=S_{1,m}S_{1,m}^{\top}$ is a positive definite matrix, $c_{1,m}\in\mathbb{R}^{2}$ and $W_{1,t}$ is a two-dimensional standard Wiener process. $\{\delta_{m,t}\}_{t\geq 0}$ 
is defined as the following four-dimensional-OU process:
\begin{align*}
	\begin{cases}
	\dd \delta_{m,t}=-(A_{2,m}\delta_{m,t}-\mu_{2,m})\dd t+S_{2,m}
	\dd W_{2,t}\quad (t\in [0,T]),\\
	\delta_{m,0}=c_{2,m},
	\end{cases}
\end{align*}
where $A_{2,m}\in\mathbb{R}^{4\times 4}$, $\mu_{2,m}\in\mathbb{R}^{4}$,  $S_{2,m}\in\mathbb{R}^{4\times 4}$ s.t. 
$\Sigma_{\delta\delta,m}=S_{2,m}S_{2,m}^{\top}$ is a positive definite diagonal matrix, 
$c_{2,m}\in\mathbb{R}^{4}$ and $W_{2,t}$ is a four-dimensional standard Wiener process. 
$\{\varepsilon_{m,t}\}_{t\geq 0}$ is defined by the two-dimensional OU process as follows:
\begin{align*}
	\begin{cases}
	\dd \varepsilon_{m,t}=-(A_{3,m}\varepsilon_{m,t}-\mu_{3,m})\dd t+S_{3,m}\dd W_{3,t}\quad (t\in [0,T]),\\
	\varepsilon_{m,0}=c_{3,m},
	\end{cases}
\end{align*}
where $A_{3,m}\in\mathbb{R}^{2\times 2}$, $\mu_{3,m}\in\mathbb{R}^{2}$, $S_{3,m}\in\mathbb{R}^{2\times 2}$ s.t. $\Sigma_{\varepsilon\varepsilon,m}=S_{3,m}S_{3,m}^{\top}$ is a positive definite diagonal matrix, $c_{3,m}\in\mathbb{R}^{2}$ and $W_{3,t}$ is a two-dimensional standard Wiener process. 
$\{\zeta_{m,t}\}_{t\geq 0}$ satisfies the following one-dimensional OU process:
\begin{align*}
	\begin{cases}
	\dd \zeta_{m,t}=-(A_{4,m}\zeta_{m,t}-\mu_{4,m})\dd t+S_{4,m}
	\dd W_{4,t}\quad (t\in [0,T]),\\
	\zeta_{m,0}=c_{4,m},
	\end{cases}
\end{align*}
where $A_{4,m}\in\mathbb{R}$, $\mu_{4,m}\in\mathbb{R}$, $S_{4,m}>0$, $c_{4,m}\in\mathbb{R}$ and $W_{4,t}$ is the one-dimensional standard Wiener process. We assume that $W_{1,t}$, $W_{2,t}$, $W_{3,t}$ and $W_{4,t}$ are independent. The parameter is expressed as 
\begin{align*}
	\theta_{m}&=\bigl((\Lambda_{x_1,m})_{21},(\Lambda_{x_1,m})_{42},(\Lambda_{x_2,m})_{21},(\Gamma_{m})_{11},(\Gamma_{m})_{12},(\Sigma_{\xi\xi,m})_{11},(\Sigma_{\xi\xi,m})_{12},
	(\Sigma_{\xi\xi,m})_{22},\\
	&\qquad\qquad(\Sigma_{\delta\delta,m})_{11},(\Sigma_{\delta\delta,m})_{22},(\Sigma_{\delta\delta,m})_{33},(\Sigma_{\delta\delta,m})_{44},(\Sigma_{\varepsilon\varepsilon,m})_{11},(\Sigma_{\varepsilon\varepsilon,m})_{22},\Sigma_{\zeta\zeta,m}\bigr)^{\top}\in\Theta_{m},
\end{align*}
where $\Theta_m=\bigl\{[-100,-0.1]\cup[0.1,100]\bigr\}^5\times[0.1,100]\times\bigl\{[-100,-0.1]\cup[0.1,100]\bigr\}\times[0.1,100]^8$. The covariance structure is defined as
\begin{align*}
	\Sigma_{m}(\theta_{m})=
	\begin{pmatrix}
		\Sigma_{X_1X_1,m}(\theta_{m}) & \Sigma_{X_1X_2,m}(\theta_{m})\\
		\Sigma_{X_1X_2,m}(\theta_{m})^{\top} & \Sigma_{X_2X_2,m}(\theta_{m})
	\end{pmatrix},
\end{align*}
where 
\begin{align*}
    \qquad\qquad\quad\Sigma_{X_1X_1,m}(\theta_{m})&=\Lambda_{x_1,m}\Sigma_{\xi\xi,m}\Lambda_{x_1,m}^{\top}+\Sigma_{\delta\delta,m},\\
    \Sigma_{X_1X_2,m}(\theta_{m})&=\Lambda_{x_1,m}\Sigma_{\xi\xi,m}\Gamma_{m}^{\top}\Lambda_{x_2,m}^{\top},\\
    \Sigma_{X_2X_2,m}(\theta_{m})&=\Lambda_{x_2,m}(\Gamma_m\Sigma_{\xi\xi,m}\Gamma^{\top}_m+\Sigma_{\zeta\zeta,m})\Lambda_{x_2,m}^{\top}+\Sigma_{\varepsilon\varepsilon,m}.
\end{align*}
The path diagram of the true model is shown in Figure \ref{M0path}. Furthermore, we set $(\Lambda_{x_1,m,0})_{21}=2$,\ $(\Lambda_{x_1,m,0})_{42}=3$,\ $(\Lambda_{x_2,m,0})_{21}=3$,\ $(\Gamma_{m})_{11}=1$ and $(\Gamma_{m})_{12}=2$,
where $\Lambda_{x_1,m,0}$, $\Lambda_{x_2,m,0}$ and $\Gamma_{m,0}$ are the true values of $\Lambda_{x_1,m}$, $\Lambda_{x_2,m}$ and $\Gamma_{m}$. Let
\begin{align*}
A_{1,m,0}&=\begin{pmatrix}
	0.5 & 0.3\\
	0.2 & 0.4
\end{pmatrix},\ \
\mu_{1,m,0}=\begin{pmatrix}
	2\\
	4
\end{pmatrix},\ \
S_{1,m,0}=\begin{pmatrix}
	1 & 1\\
	0 & 2
\end{pmatrix},\ \
c_{1,m}=\begin{pmatrix}
	3\\
	5
\end{pmatrix},
\end{align*}
where $A_{1,m,0}$, $\mu_{1,m,0}$ and $S_{1,m,0}$ are the true values of $A_{1,m}$, $\mu_{1,m}$ and $S_{1,m}$. Define
\begin{align*}
A_{2,m,0}&=\begin{pmatrix}
	3 & 0 & 0 & 0\\
	0 & 2 & 0 & 0\\
	0 & 0 & 3 & 0\\
	0 & 0 & 0 & 2
\end{pmatrix},\ \
\mu_{2,m,0}=
\begin{pmatrix}
	0 \\
	0 \\
	0 \\
	0
\end{pmatrix},\ \
S_{2,m,0}=\begin{pmatrix}
	1 & 0 & 0 & 0\\
	0 & 2 & 0 & 0\\
	0 & 0 & 2 & 0\\
	0 & 0 & 0 & 1
\end{pmatrix},\ \
c_{2,m}=\begin{pmatrix}
	0\\
	0\\
	0\\
	0
\end{pmatrix},
\end{align*}
where $A_{2,m,0}$, $\mu_{2,m,0}$ and $S_{2,m,0}$ are the true values of $A_{2,m}$, $\mu_{2,m}$ and $S_{2,m}$. Set
\begin{align*}
A_{3,m,0}&=\begin{pmatrix}
	2 & 0 \\
	0 & 3
\end{pmatrix},\ \
\mu_{3,m,0}=\begin{pmatrix}
	0\\
	0
\end{pmatrix},\ \
S_{3,m,0}=\begin{pmatrix}
	1 & 0 \\
	0 & 3
\end{pmatrix},\ \
c_{3,m}=\begin{pmatrix}
	0\\
	0
\end{pmatrix},
\end{align*}
where $A_{3,m,0}$, $\mu_{3,m,0}$ and $S_{3,m,0}$ are the true values of $A_{3,m}$, $\mu_{3,m}$ and $S_{3,m}$. Denote $A_{4,m,0}=1$,\ $\mu_{4,m,0}=0$,\ $S_{4,m,0}=2$ and $c_{4,m}=0$, where $A_{4,m,0}$, $\mu_{4,m,0}$ and $S_{4,m,0}$ are the true values of $A_{4,m}$, $\mu_{4,m}$ and $S_{4,m}$. Thus, the true parameter is expressed as 
\begin{align*}
	\theta_{m,0}=\bigl(
	2, 3, 3, 1, 2, 2, 2, 4, 1, 4, 4, 1, 1, 9, 4
	\bigr)^{\top}\in\Theta_{m}
\end{align*}
and we have 
\begin{align*}
	\Sigma_{m}(\theta_{m,0})=\begin{pmatrix}
		3 & 4 & 2 & 6 & 6 & 18\\
		4 & 12 & 4 & 12 & 12 & 36\\
		2 & 4 & 8 & 12 & 10 & 30\\
		6 & 12 & 12 & 37 & 30 & 90\\
		6 & 12 & 10 & 30 &31 & 90\\
		18 & 36 & 30 & 90 & 90 & 279
	\end{pmatrix}.
\end{align*}
\begin{figure}[h]
	\includegraphics[width=1\columnwidth]{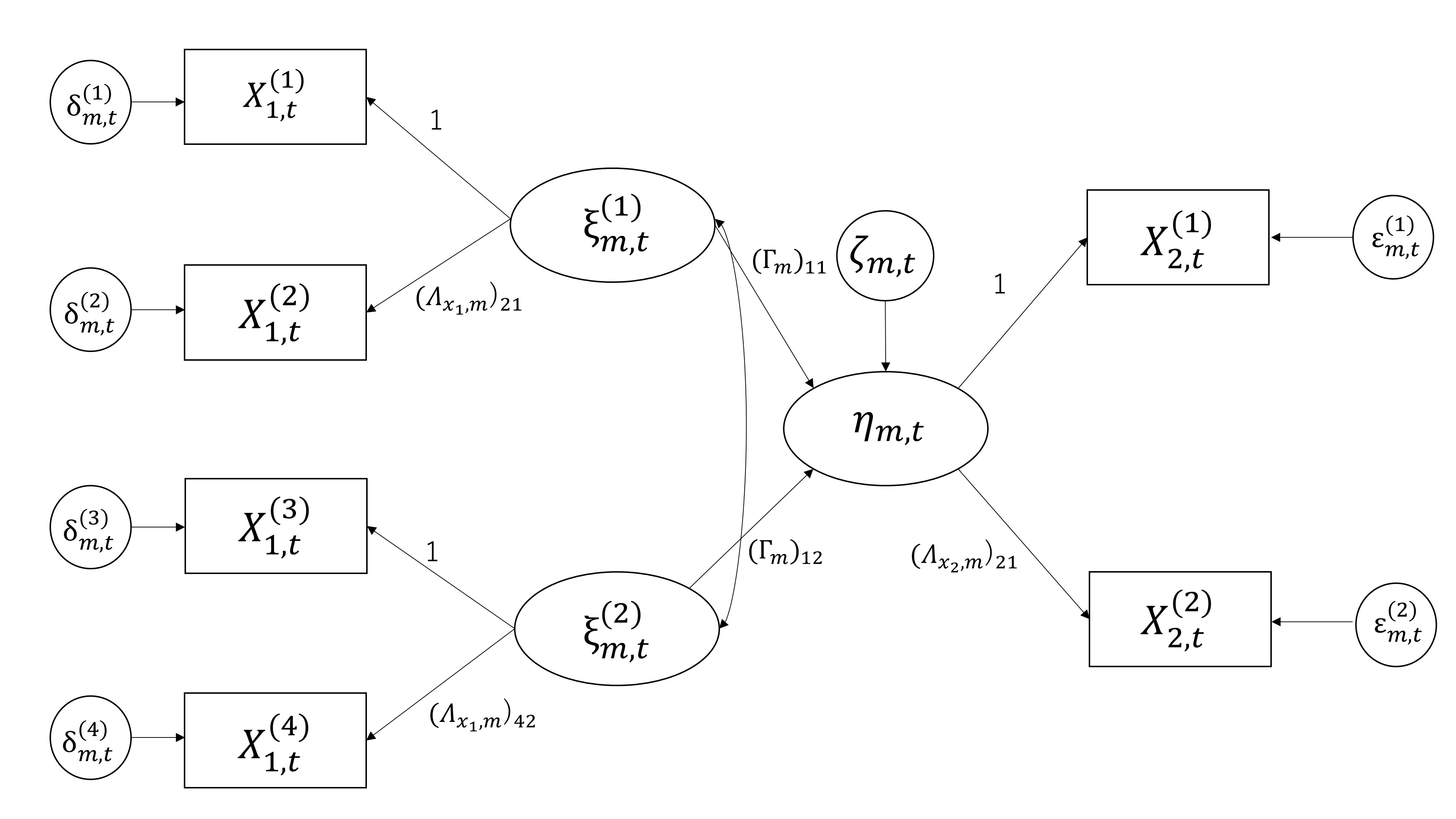}
	\caption{Path diagram of the true model.\qquad\qquad\qquad}
	\label{M0path}
\end{figure}
\begin{remark}
We check that the true model satisfies $\bf{[G]}$:
\begin{align*}
	\Sigma_{m}(\theta_{m})=\Sigma_{m}(\tilde{\theta}_{m})\Longrightarrow\theta_{m}=\tilde{\theta}_{m}.
\end{align*}
Assume that
\begin{align}
	\Sigma_{m}(\theta_{m})=\Sigma_{m}(\tilde{\theta}_{m}).\label{eq}
\end{align}
From the (1,3)-th element of $(\ref{eq})$, we obtain
\begin{align}
	(\Sigma_{\xi\xi,m})_{12}=(\tilde{\Sigma}_{\xi\xi,m})_{12}.\label{sigmaxi12}
\end{align}
Since it holds from the (2,3)-th  and (1,4)-th elements of $(\ref{eq})$ that
\begin{align*}
    (\Lambda_{x_1,m})_{21}(\Sigma_{\xi\xi,m})_{12}=(\tilde{\Lambda}_{x_1,m})_{21}(\tilde{\Sigma}_{\xi\xi,m})_{12},\ (\Lambda_{x_1,m})_{42}(\Sigma_{\xi\xi,m})_{12}=(\tilde{\Lambda}_{x_1,m})_{42}(\tilde{\Sigma}_{\xi\xi,m})_{12}
\end{align*}
and $(\Sigma_{\xi\xi,m})_{12}$ is not zero, 
we see from  (\ref{sigmaxi12}) that
\begin{align}
	(\Lambda_{x_1,m})_{21}=(\tilde{\Lambda}_{x_1,m})_{21},\  (\Lambda_{x_1,m})_{42}=(\tilde{\Lambda}_{x_1,m})_{42}.\label{lambdax}
\end{align}
As it follows from the (1,2)-th and (3,4)-th elements of (\ref{eq}) that
\begin{align*}
    (\Lambda_{x_1,m})_{21}(\Sigma_{\xi\xi,m})_{11}=(\tilde{\Lambda}_{x_1,m})_{21}(\tilde{\Sigma}_{\xi\xi,m})_{11},\  (\Lambda_{x_1,m})_{42}(\Sigma_{\xi\xi,m})_{22}=(\tilde{\Lambda}_{x_1,m})_{42}(\tilde{\Sigma}_{\xi\xi,m})_{22}
\end{align*}
and $(\Lambda_{x_1,m})_{21}$ and $(\Lambda_{x_1,m})_{42}$ are not zero, 
we obtain 
\begin{align}
	(\Sigma_{\xi\xi,m})_{11}=
	(\tilde{\Sigma}_{\xi\xi,m})_{11},\ (\Sigma_{\xi\xi,m})_{22}=(\tilde{\Sigma}_{\xi\xi,m})_{22}
	\label{sigmaxi}
\end{align}
from (\ref{lambdax}).
Since we get
\begin{align*}
    &\bigl(\Sigma_{X_1X_1,m}(\theta_{m})=\bigr)\ \Lambda_{x_1,m}\Sigma_{\xi\xi,m}\Lambda_{x_1,m}^{\top}+\Sigma_{\delta\delta,m}\\
    &\qquad\qquad\qquad\qquad\qquad\qquad\qquad=\tilde{\Lambda}_{x_1,m}\tilde{\Sigma}_{\xi\xi,m}\tilde{\Lambda}_{x_1,m}^{\top}+\tilde{\Sigma}_{\delta\delta,m}\ \bigl(=\Sigma_{X_1X_1,m}(\tilde{\theta}_{m})\bigr),
\end{align*}
from (\ref{eq}), we see
\begin{align}
    \Sigma_{\delta\delta,m}=\tilde{\Sigma}_{\delta\delta,m}\label{delta2}
\end{align}
from (\ref{sigmaxi12}), (\ref{lambdax}) and (\ref{sigmaxi}). 
Furthermore, it holds from the (1,5) and (3,5)-th elements of (\ref{eq}) that
\begin{align*}
    \Sigma_{\xi\xi,m}\Gamma_{m}^{\top}=\tilde{\Sigma}_{\xi\xi,m}\tilde{\Gamma}_{m}^{\top}
\end{align*}
and $\Sigma_{\xi\xi,m}$ is a positive definite matrix, which yields
\begin{align}
	\Gamma_{m}=\tilde{\Gamma}_{m} \label{gamma}
\end{align}
from (\ref{sigmaxi12}) and (\ref{sigmaxi}). 
Note that
\begin{align*}
    \Sigma_{\xi\xi,m}\Gamma_{m}^{\top}\neq 0
\end{align*}
since $\Gamma_{m}^{\top}$ is not a zero vector and $\Sigma_{\xi\xi,m}$ is a positive definite matrix. Recalling that
\begin{align*}
    &(\Lambda_{x_2,m})_{21}(\Sigma_{\xi\xi,m})_{11}(\Gamma_{m})_{11}+(\Lambda_{x_2,m})_{21}(\Sigma_{\xi\xi,m})_{12}(\Gamma_{m})_{12}\\
    &\qquad\qquad\qquad\qquad\qquad=(\tilde{\Lambda}_{x_2,m})_{21}(\tilde{\Sigma}_{\xi\xi,m})_{11}(\tilde{\Gamma}_{m})_{11}+(\tilde{\Lambda}_{x_2,m})_{21}(\tilde{\Sigma}_{\xi\xi,m})_{12}(\tilde{\Gamma}_{m})_{12}
\end{align*}
from the (1,6)-th element of (\ref{eq}), we have
\begin{align}
	\Lambda_{x_2,m}=\tilde{\Lambda}_{x_2,m}.\label{lambday}
\end{align}
from (\ref{sigmaxi12}), (\ref{sigmaxi}) and (\ref{gamma}). 
Since it holds from the (5,6)-th element of (\ref{eq}) that
\begin{align*}
	&(\Lambda_{x_2,m})_{21}(\Gamma_{m})_{11}^{2}(\Sigma_{\xi\xi,m})_{11}+2(\Lambda_{x_2,m})_{21}(\Gamma_{m})_{11}(\Gamma_{m})_{12}(\Sigma_{\xi\xi,m})_{12}\\
	&\qquad\qquad\qquad+(\Lambda_{x_2,m})_{21}(\Gamma_{m})_{12}^2(\Sigma_{\xi\xi,m})_{22}+(\Lambda_{x_2,m})_{21}\Sigma_{\zeta\zeta,m}\\
	&\qquad\qquad\qquad\qquad\qquad=(\tilde{\Lambda}_{x_2,m})_{21}(\tilde{\Gamma}_{m})_{11}^{2}(\tilde{\Sigma}_{\xi\xi,m})_{11}+2(\tilde{\Lambda}_{x_2,m})_{21}(\tilde{\Gamma}_{m})_{11}(\tilde{\Gamma}_{m})_{12}(\tilde{\Sigma}_{\xi\xi,m})_{12}\\
	&\qquad\qquad\qquad\qquad\qquad\qquad\qquad\qquad\qquad\qquad+(\tilde{\Lambda}_{x_2,m})_{21}(\tilde{\Gamma}_{m})_{12}^{2}(\tilde{\Sigma}_{\xi\xi,m})_{22}+(\tilde{\Lambda}_{x_2,m})_{21}\tilde{\Sigma}_{\zeta\zeta,m}
\end{align*}
and $(\Lambda_{x_2,m})_{21}$ is not zero, one has
\begin{align}
	\Sigma_{\zeta\zeta,m}=\tilde{\Sigma}_{\zeta\zeta,m} \label{sigmazeta}
\end{align}
from (\ref{sigmaxi12}), (\ref{lambdax}), (\ref{sigmaxi}), (\ref{gamma}) and (\ref{lambday}). 
Furthermore, we see from (\ref{eq}) that
\begin{align*}
    \bigl(\Sigma_{X_2X_2,m}(\theta_{m})=\bigr)\ &\Lambda_{x_2,m}(\Gamma_{m}\Sigma_{\xi\xi,m}\Gamma_{m}^{\top}+\Sigma_{\zeta\zeta,m})\Lambda_{x_2,m}^{\top}+\Sigma_{\varepsilon\varepsilon,m}\\
    &\qquad\qquad=\tilde{\Lambda}_{x_2,m}(\tilde{\Gamma}_{m}\tilde{\Sigma}_{\xi\xi,m}\tilde{\Gamma}_{m}^{\top}+\tilde{\Sigma}_{\zeta\zeta,m})\tilde{\Lambda}_{x_2,m}^{\top}+\tilde{\Sigma}_{\varepsilon\varepsilon,m}\ \bigl(=\tilde{\Sigma}_{X_2X_2,m}(\theta_{m})\bigr)
\end{align*}
and it follows from (\ref{sigmaxi12}), (\ref{lambdax}), (\ref{sigmaxi}), (\ref{gamma}), 
(\ref{lambday}) and (\ref{sigmazeta}) that
\begin{align}
    \Sigma_{\varepsilon\varepsilon,m}=\tilde{\Sigma}_{\varepsilon\varepsilon,m}\label{sigmavarepsilon}.
\end{align}
Therefore, from $(\ref{sigmaxi12})$-$(\ref{sigmavarepsilon})$, 
we obtain  $\theta_{m}=\tilde{\theta}_{m}$, which implies that the true model satisfies $\bf{[G]}$.
\end{remark}
\subsection{Correctly specified parametric model}
Let $k_{1,M_0}=2$ and $k_{2,M_0}=1$. Define
\begin{align*}
	\Lambda_{x_1,M_0}=
	\begin{pmatrix}
		1 & (\Lambda_{x_1,M_0})_{21} & 0 & 0\\
		0 & 0 & 1 &(\Lambda_{x_1,M_0})_{42}
	\end{pmatrix}^{\top}\in\mathbb{R}^{4\times 2},
\end{align*}
where $(\Lambda_{x_1,M_0})_{21}$ and $(\Lambda_{x_1,M_0})_{42}$ are not zero. Set
\begin{align*}
	\Lambda_{x_{2},M_0}=\bigl(
	1,\ (\Lambda_{x_{2},M_0})_{21}
	\bigr)^{\top}\in\mathbb{R}^{2\times 1}, 
\end{align*}
where $(\Lambda_{x_{2},M_0})_{21}$ is not zero. 
Let $\Gamma_{M_0}\in\mathbb{R}^{1\times 2}$, 
where $\Gamma_{M_0}$ is not a zero matrix. 
Furthermore, we assume that $\Sigma_{\xi\xi,m}\in\mathbb{R}^{2\times 2}$ is a positive definite matrix, $\Sigma_{\delta\delta,M_0}\in\mathbb{R}^{4\times 4}$ is a positive definite diagonal matrix, $\Sigma_{\varepsilon\varepsilon,M_0}\in\mathbb{R}^{2\times 2}$ is  a positive definite diagonal matrix, and 
$\Sigma_{\zeta\zeta,M_0}>0$. 
The parameter is expressed as
\begin{align*}
	\theta_{M_0}&=\bigl((\Lambda_{x_1,M_0})_{21},(\Lambda_{x_1,M_0})_{42},(\Lambda_{x_2,M_0})_{21},(\Gamma_{M_0})_{11},(\Gamma_{M_0})_{12},(\Sigma_{\xi\xi,M_0})_{11},(\Sigma_{\xi\xi,M_0})_{12},
	(\Sigma_{\xi\xi,M_0})_{22},\\
	&\qquad\quad(\Sigma_{\delta\delta,M_0})_{11},(\Sigma_{\delta\delta,M_0})_{22},(\Sigma_{\delta\delta,M_0})_{33},(\Sigma_{\delta\delta,M_0})_{44},(\Sigma_{\varepsilon\varepsilon,M_0})_{11},(\Sigma_{\varepsilon\varepsilon,M_0})_{22},\Sigma_{\zeta\zeta,M_0}\bigr)^{\top}\in\Theta_{M_0},
\end{align*}
where $\Theta_{M_0}=\bigl\{[-100,-0.1]\cup[0.1,100]\bigr\}^5\times[0.1,100]\times\bigl\{[-100,-0.1]\cup[0.1,100]\bigr\}\times[0.1,100]^8$. Therefore, we define the covariance structure as 
\begin{align*}
	\Sigma_{M_0}(\theta_{M_0})=
	\begin{pmatrix}
	\Sigma_{X_1X_1,M_0}(\theta_{M_0}) & \Sigma_{X_1X_2,M_0}(\theta_{M_0})\\
	\Sigma_{X_1X_2,M_0}(\theta_{M_0})^{\top} & \Sigma_{X_2X_2,M_0}(\theta_{M_0})
	\end{pmatrix},
\end{align*}
where 
\begin{align*}
	\qquad\qquad\quad\Sigma_{X_1X_1,M_0}(\theta_{M_0})&=\Lambda_{x_1,M_0}\Sigma_{\xi\xi,M_0}\Lambda_{x_1,M_0}^{\top}+\Sigma_{\delta\delta,M_0},\\
	\Sigma_{X_1X_2,M_0}(\theta_{M_0})&=\Lambda_{x_1,M_0}\Sigma_{\xi\xi,M_0}\Gamma_{M_0}^{\top}\Lambda_{x_2,M_0}^{\top},\\
	\Sigma_{X_2X_2,M_0}(\theta_{M_0})&=\Lambda_{x_2,M_0}(\Gamma_{M_0}\Sigma_{\xi\xi,M_0}\Gamma_{M_0}^{\top}+\Sigma_{\zeta\zeta,M_0})\Lambda_{x_2,M_0}^{\top}+\Sigma_{\varepsilon\varepsilon,M_0}.
\end{align*}
\subsection{Missspecified parametric model}
\subsubsection{Model $M_1$}
Set $k_{1,M_1}=1$ and $k_{2,M_1}=1$. Let
\begin{align*}
    \Lambda_{x_1,M_1}=\bigl(1,\ (\Lambda_{x_1,M_1})_{21},\ (\Lambda_{x_1,M_1})_{31},\ (\Lambda_{x_1,M_1})_{41}\bigr)^{\top}\in\mathbb{R}^{4\times 1},
\end{align*}
where $(\Lambda_{x_1,M_1})_{21}$, $(\Lambda_{x_1,M_1})_{31}$ and $(\Lambda_{x_1,M_1})_{41}$ are not zero. Set
\begin{align*}
	\Lambda_{x_2,M_1}=\bigl(
		1,\  (\Lambda_{x_2,M_{1}})_{21}
	\bigr)^{\top}\in\mathbb{R}^{2\times 1},
\end{align*}
where $(\Lambda_{x_2,M_{1}})_{21}$ is not zero. 
Let $\Gamma_{M_1}\in\mathbb{R}$, where $\Gamma_{M_1}$ is not zero. 
We assume that $\Sigma_{\xi\xi,M_1}>0$, $\Sigma_{\zeta\zeta,M_1}>0$,
$\Sigma_{\delta\delta,M_1}\in\mathbb{R}^{4\times 4}$ 
and $\Sigma_{\varepsilon\varepsilon,M_1}\in\mathbb{R}^{2\times 2}$ are  positive definite diagonal matrices.
The parameter is expressed as follows:
\begin{align*}
	\theta_{M_1}&=\bigl((\Lambda_{x_1,M_1})_{21}, (\Lambda_{x_1,M_1})_{31}, (\Lambda_{x_1,M_1})_{41}, (\Lambda_{x_2,M_{1}})_{21}, \Gamma_{M_1}, \Sigma_{\xi\xi,M_1},(\Sigma_{\delta\delta,M_1})_{11},\\
	&\qquad\qquad(\Sigma_{\delta\delta,M_1})_{22},(\Sigma_{\delta\delta,M_1})_{33},(\Sigma_{\delta\delta,M_1})_{44},(\Sigma_{\varepsilon\varepsilon,M_1})_{11},(\Sigma_{\varepsilon\varepsilon,M_1})_{22},\Sigma_{\zeta\zeta,M_1}\bigr)^{\top}\in\Theta_{M_1},
\end{align*}
where $\Theta_{M_1}=\{[-100,-0.1]\cup[0.1,100]\}^5\times[0.1,100]^8$. Therefore, we set the covariance structure as
\begin{align*}
    \Sigma_{M_1}(\theta_{M_1})=
	\begin{pmatrix}
		\Sigma_{X_1X_1,M_1}(\theta_{M_1}) & \Sigma_{X_1X_2,M_1}(\theta_{M_1})\\
		\Sigma_{X_1X_2,M_1}(\theta_{M_1})^{\top} & \Sigma_{X_2X_2,M_1}(\theta_{M_1})
	\end{pmatrix},
\end{align*}
where 
\begin{align*}
    \qquad\qquad\Sigma_{X_1X_1,M_1}(\theta_{M_1})&=\Lambda_{x_1,M_1}\Sigma_{\xi\xi,M_1}\Lambda_{x_1,M_1}^{\top}+\Sigma_{\delta\delta,M_1},\\
    \Sigma_{X_1X_2,M_1}(\theta_{M_1})&=\Lambda_{x_1,M_1}\Sigma_{\xi\xi,M_1}\Gamma_{M_1}\Lambda_{x_2,M_1}^{\top},\\
    \Sigma_{X_2X_2,M_1}(\theta_{M_1})&=\Lambda_{x_2,M_1}(\Gamma_{M_1}^2\Sigma_{\xi\xi,M_1}+\Sigma_{\zeta\zeta,M_1})\Lambda_{x_2,M_1}^{\top}+\Sigma_{\varepsilon\varepsilon,M_1}.
\end{align*}
Figure \ref{M1path} shows the path diagram of Model $M_1$.
\begin{figure}[h]
	\includegraphics[width=1\columnwidth]{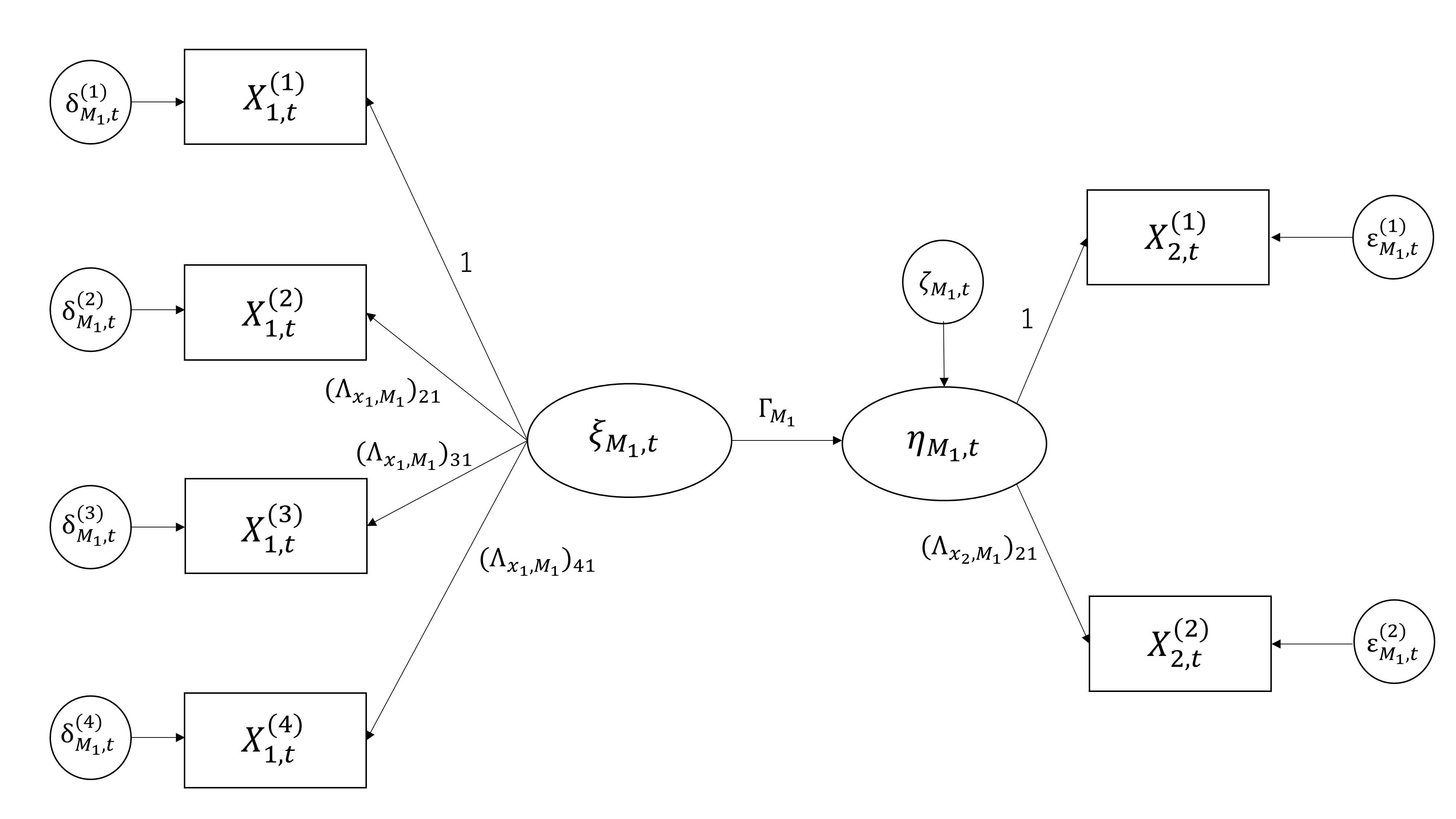}
	\caption{Path diagram of Model $M_1$.\qquad\qquad\qquad}
	\label{M1path}
\end{figure}
\subsubsection{Model $M_2$}
Let $k_{1,M_2}=2$ and $k_{2,M_2}=1$. Set
\begin{align*}
	\Lambda_{x_1,M_2}=
	\begin{pmatrix}
		1 & (\Lambda_{x_1,M_2})_{21} & 0 & 0\\
		0 & 0 & 1 &(\Lambda_{x_1,M_2})_{42}
	\end{pmatrix}^{\top}\in\mathbb{R}^{4\times 2},
\end{align*}
where $(\Lambda_{x_{1},M_2})_{21}$ and $(\Lambda_{x_1,M_2})_{42}$ are not zero. 
Let $\Gamma_{M_2}\in\mathbb{R}^{1\times 2}$, 
where $\Gamma_{M_2}$ is not a zero matrix. 
Assume that $\Sigma_{\xi\xi,m}\in\mathbb{R}^{2\times 2}$,  
$\Sigma_{\delta\delta,M_2}\in\mathbb{R}^{4\times 4}$
and
$\Sigma_{\varepsilon\varepsilon,M_2}\in\mathbb{R}^{2\times 2}$ 
are positive definite diagonal matrices, 
and $\Sigma_{\zeta\zeta,M_2}>0$. The parameter is expressed as 
\begin{align*}
	\theta_{M_2}&=\bigl((\Lambda_{x_1,M_2})_{21},(\Lambda_{x_1,M_2})_{42},(\Lambda_{x_2,M_2})_{21},(\Gamma_{M_2})_{11},(\Gamma_{M_2})_{12},(\Sigma_{\xi\xi,M_2})_{11},(\Sigma_{\xi\xi,M_2})_{22},\\
	&\qquad\quad(\Sigma_{\delta\delta,M_2})_{11},(\Sigma_{\delta\delta,M_2})_{22},(\Sigma_{\delta\delta,M_2})_{33},(\Sigma_{\delta\delta,M_2})_{44},(\Sigma_{\varepsilon\varepsilon,M_2})_{11},(\Sigma_{\varepsilon\varepsilon,M_2})_{22},\Sigma_{\zeta\zeta,M_2}\bigr)^{\top}\in\Theta_{M_2},
\end{align*}
where $\Theta_{M_{2}}=\bigl\{[-100,-0.1]\cup[0.1,100]\bigr\}^5\times[0.1,100]^9$. Therefore, we define 
\begin{align*}
	\Sigma_{M_2}(\theta_{M_2})=
	\begin{pmatrix}
	\Sigma_{X_1X_1,M_2}(\theta_{M_2}) & \Sigma_{X_1X_2,M_2}(\theta_{M_2})\\
	\Sigma_{X_1X_2,M_2}(\theta_{M_2})^{\top} & \Sigma_{X_2X_2,M_2}(\theta_{M_2})
	\end{pmatrix},
\end{align*}
where 
\begin{align*}
	\qquad\qquad\quad\Sigma_{X_1X_1,M_2}(\theta_{M_2})&=\Lambda_{x_1,M_2}\Sigma_{\xi\xi,M_2}\Lambda_{x_1,M_2}^{\top}+\Sigma_{\delta\delta,M_2},\\
	\Sigma_{X_1X_2,M_2}(\theta_{M_2})&=\Lambda_{x_1,M_2}\Sigma_{\xi\xi,M_2}\Gamma_{M_2}^{\top}\Lambda_{x_2,M_2}^{\top},\\
	\Sigma_{X_2X_2,M_2}(\theta_{M_2})&=\Lambda_{x_2,M_2}(\Gamma_{M_2}\Sigma_{\xi\xi,M_2}\Gamma_{M_2}^{\top}+\Sigma_{\zeta\zeta,M_2})\Lambda_{x_2,M_2}^{\top}+\Sigma_{\varepsilon\varepsilon,M_2}.
\end{align*}
The path diagram of Model $M_2$ is shown in Figure \ref{M2path}.
\begin{figure}[h]
	\includegraphics[width=1\columnwidth]{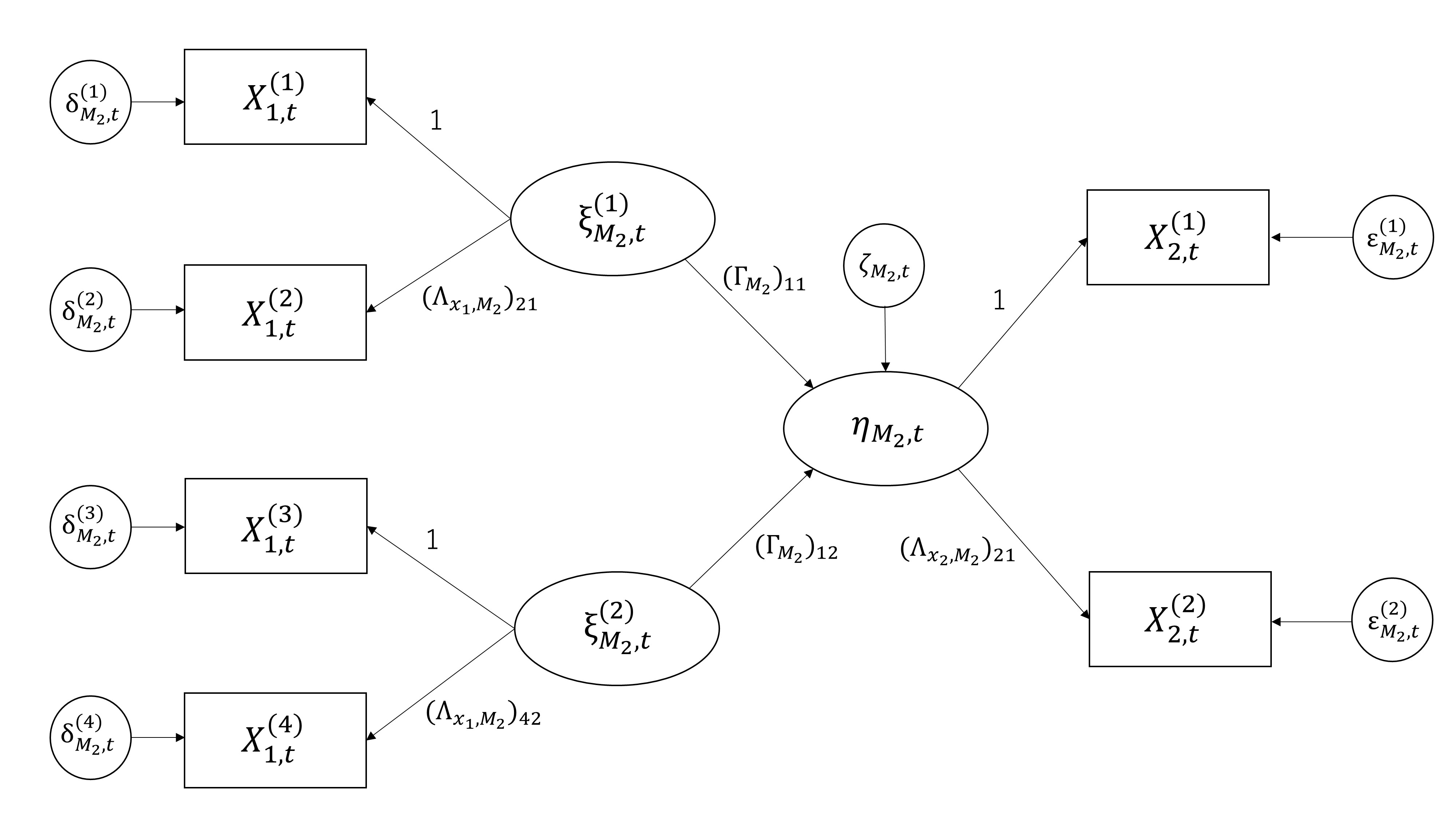}
	\caption{Path diagram of Model $M_2$.\qquad\qquad\qquad}
	\label{M2path}
\end{figure}
\subsection{Simulation results}
We set $(n,h_n,T)=(10^4,10^{-3},10^{1})$ and generated 10,000 independent sample paths from the true model. 
	
\subsubsection{Parameter estimation}
First, we check the asymptotic performance of $Q_{XX}$. Table \ref{QQtable} shows the sample mean and the sample standard deviation (SD) of $Q_{XX}$. 
From this table, we deduce that $Q_{XX}$ has consistency. Figure \ref{QQfigure} shows the histogram, the Q-Q plot and the empirical distribution of $\sqrt{n}((Q_{XX})_{11}-(\Sigma_{m}(\theta_{m,0}))_{11})$, which implies that $Q_{XX}$ has asymptotic normality. Thus, we see that Theorem \ref{Qtheoremnon} holds for this example. 
Next, we investigate the asymptotic performance of $\hat{\theta}_{m,n}$. 
To optimize $\mathbb{F}_{m,n}(\theta_{m})$, we use optim() with the BFGS method in R language. 
Set the initial value of the optimization as $\theta_{m}=\theta_{m,0}$. 
Table \ref{thetatable} shows the sample mean and the sample SD of $\hat{\theta}_{m,n}$
and we deduce that $\hat{\theta}_{m,n}$ has consistency. Figure \ref{thetafigure} shows the histogram, the Q-Q plot and the empirical distribution of $\sqrt{n}(\hat{\theta}_{m,n}^{(1)}-\theta_{m,0}^{(1)})$. This figure implies that $\hat{\theta}_{m,n}$ has asymptotic normality. Therefore, these results yield that Theorem \ref{thetatheoremnon} is correct for this example. See Appendix \ref{simulation} for details of simulation results. 
\subsection{Goodness-of-fit-test}
First, consider the following statistical hypothesis test:
\begin{align*}
	\left\{
	\begin{array}{ll}
	H_0: \Sigma_{m}(\theta_{m})=\Sigma_{M_0}(\theta_{M_0}),\\
	H_1: \Sigma_{m}(\theta_{m})\neq\Sigma_{M_0}(\theta_{M_0}).
    \end{array}
	\right.
\end{align*}
Note that the null hypothesis is true since Model $M_0$ is a correctly specified parametric model. Recall that the test statistic $\mathbb{T}_{M_0,n}=n\mathbb{F}_{M_0,n}(\hat{\theta}_{M_0,n})$, and the rejection region is 
\begin{align*}
    \Bigl\{t_{M_0,n}>\chi^2_{6}(0.05)=12.59\Bigr\}.
\end{align*}
Table \ref{testtable1} shows the sample mean 
and the sample SD of the test statistic $\mathbb{T}_{M_0,n}$. Figure \ref{testfigure1} shows the histogram, the Q-Q plot and the empirical distribution of the test statistic $\mathbb{T}_{M_0,n}$. Table \ref{testtable1} and  Figure \ref{testfigure1} imply that the test statistic $\mathbb{T}_{M_0,n}$ converges in distribution to a chi-squared distribution with $6$ degree of freedom under the null hypothesis. These simulation results support Theorem \ref{chitheoremnon}.

Next, we study the following statistical hypothesis test:
\begin{align}
		\left\{
		\begin{array}{ll}
			H_0: \Sigma_{m}(\theta_{m})=\Sigma_{M_1}(\theta_{M_1}),\\
			H_1: \Sigma_{m}(\theta_{m})\neq\Sigma_{M_1}(\theta_{M_1}).
		\end{array}
		\right.\label{testM1}
\end{align}
Note that the test statistic $\mathbb{T}_{M_1,n}=\mathbb{F}_{M_1,n}(\hat{\theta}_{M_1,n})$, and the rejection region is
\begin{align*}
    \Bigl\{t_{M_1,n}>\chi^2_{8}(0.05)=15.51\Bigr\}.
\end{align*}
Furthermore, 
we consider the following statistical hypothesis test:
	\begin{align}
		\left\{
		\begin{array}{ll}
			H_0: \Sigma_{m}(\theta_{m})=\Sigma_{M_2}(\theta_{M_2}),\\
			H_1: \Sigma_{m}(\theta_{m})\neq\Sigma_{M_2}(\theta_{M_2}).
		\end{array}
		\right.\label{testM2}
	\end{align}
The test statistic $\mathbb{T}_{M_2,n}=\mathbb{F}_{M_2,n}(\hat{\theta}_{M_2,n})$, and the rejection region is 
\begin{align*}
    \Bigl\{t_{M_2,n}>\chi^2_{7}(0.05)=14.07\Bigr\}.
\end{align*}
Note that the alternative hypothesis is true in both (\ref{testM1}) and (\ref{testM2}) tests since Model $M_1$ and $M_2$ are missspecified parametric models. 
To optimize $\mathbb{F}_{M_1,n}(\theta_{M_1})$ and $\mathbb{F}_{M_2,n}(\theta_{M_2})$, 
we perform the following procedure.
\begin{description}
\item[Step1] $\theta_{M_1,u}^{Initial}$ and $\theta_{M_2,u}^{Initial}$ for $u=1,\cdots,50$ are generated from the continuous uniform random numbers on the interval $[-100,100]^5\times[0.1,100]^8$ and $[-100,100]^5\times[0.1,100]^9$ respectively.
\item[Step2]
Let  $\theta_{M_1,u}^{Initial}$ be the initial value of the optimization,
and we obtain  $\tilde{\theta}_{M_{1},n,u}$ defined as
\begin{align}
	\mathbb{F}_{M_1,n}(\tilde{\theta}_{M_{1},n,u})=\inf_{\theta_{M_{1}}\in\Theta_{M_1}}\mathbb{F}_{M_1,n}(\theta_{M_1})\label{thetaM1}
\end{align}
for $u=1,\cdots 50$. Similarly, set the initial value of the optimization as 
$\theta_{M_2,u}^{Initial}$, 
and we get $\tilde{\theta}_{M_{2},n,u}$ given by
\begin{align}
	\mathbb{F}_{M_2,n}(\tilde{\theta}_{M_{2},n,u})=\inf_{\theta_{M_{2}}\in\Theta_{M_2}}\mathbb{F}_{M_2,n}(\theta_{M_2})\label{thetaM2}
\end{align}
for $u=1,\cdots 50$. Note that optim() is used with the L-BFGS-B method in R language to compute (\ref{thetaM1}) and (\ref{thetaM2}).
\item[Step3]
Let 
\begin{align*}
    \mathbb{F}_{M_1,n}(\hat{\theta}_{M_1,n})=\min{\Bigl\{\mathbb{F}_{M_1,n}(\tilde{\theta}_{M_{1},n,1}),\mathbb{F}_{M_1,n}(\tilde{\theta}_{M_{1},n,2}),\cdots,\mathbb{F}_{M_1,n}(\tilde{\theta}_{M_{1},n,50})\Bigr\}}
\end{align*}
and
\begin{align*}
    \mathbb{F}_{M_2,n}(\hat{\theta}_{M_2,n})=\min{\Bigl\{\mathbb{F}_{M_2,n}(\tilde{\theta}_{M_{2},n,1}),\mathbb{F}_{M_2,n}(\tilde{\theta}_{M_{2},n,2}),\cdots,\mathbb{F}_{M_2,n}(\tilde{\theta}_{M_{2},n,50})\Bigr\}}.
\end{align*}
\item[Step4] Set $\mathbb{T}_{M_1,n}=n\mathbb{F}_{M_1,n}(\hat{\theta}_{M_1,n})$ and $\mathbb{T}_{M_2,n}=n\mathbb{F}_{M_2,n}(\hat{\theta}_{M_2,n})$.
\end{description}
This procedure is repeated 10,000 times. Table \ref{testtable2} shows 
the quartiles of  the test statistics $\mathbb{T}_{M_1,n}$ 
and $\mathbb{T}_{M_2,n}$, which implies that the null hypothesis is rejected in both (\ref{testM1}) and (\ref{testM2}) tests all 10000 times. Figure \ref{testfigure2} shows the box plots of the test statistics $\mathbb{T}_{M_1,n}$ and $\mathbb{T}_{M_2,n}$. 
From this figure, we deduce that Model $M_2$ is closer to the true model than Model $M_1$.
\clearpage
\ \\
\ \\
\fontsize{11pt}{16pt}\selectfont
\begin{table}[h]
	\centering
	\begin{tabular}{ccccc}
		& $(Q_{XX})_{11}$ & $(Q_{XX})_{12}$ & $(Q_{XX})_{13}$ & $(Q_{XX})_{14}$ \\\hline
		Mean (True value) & 3.002 (3.000) & 4.002 (4.000) & 2.000 (2.000)& 6.001 (6.000) \\
		SD (Theoretical value) & 0.042 (0.042) & 0.072 (0.072)& 0.052 (0.053) & 0.121 (0.121)\\
		& $(Q_{XX})_{15}$ & $(Q_{XX})_{16}$ & $(Q_{XX})_{22}$ & $(Q_{XX})_{23}$\\\hline
		Mean (True value) & 6.001 (6.000) & 18.003 (18.000) & 12.007 (12.000) & 4.001 (4.000)\\
		SD (Theoretical value) & 0.113 (0.114) & 0.337 (0.341)& 0.170 (0.170) & 0.106 (0.106)  \\
		& $(Q_{XX})_{24}$ & $(Q_{XX})_{25}$ & $(Q_{XX})_{26}$ & $(Q_{XX})_{33}$\\\hline
		Mean (True value) & 12.002 (12.000) & 12.002 (12.000) & 36.007 (36.000) & 8.006 (8.000)\\
		SD (Theoretical value) & 0.245 (0.242) & 0.229 (0.227) & 0.686 (0.681) & 0.112 (0.113) \\
		& $(Q_{XX})_{34}$ & $(Q_{XX})_{35}$ & $(Q_{XX})_{36}$ & $(Q_{XX})_{44}$\\\hline
		Mean (True value) & 12.003 (12.000) & 10.003 (10.000) & 30.009 (30.000) & 37.010 (37.000)\\
		SD (Theoretical value) & 0.210 (0.210) & 0.187 (0.187)& 0.561 (0.560) & 0.530 (0.523) \\
		& $(Q_{XX})_{45}$ & $(Q_{XX})_{46}$ & $(Q_{XX})_{55}$ & $(Q_{XX})_{56}$\\\hline
		Mean (True value) & 30.006 (30.000) & 90.021 (90.000) & 31.008 (31.000) & 90.023 (90.000)\\
		SD (Theoretical value) & 0.457 (0.452) & 1.369 (1.357) & 0.441 (0.438) & 1.302 (1.294) \\
		& $(Q_{XX})_{66}$ &&& \\\hline
		Mean (True value) & 279.081 (279.000) & & &\\
		SD (Theoretical value) & 3.961 (3.946) & &  \\ \\
\end{tabular}
\caption{Sample mean and sample standard deviation (SD) of $Q_{XX}$.\qquad\qquad\qquad\qquad}
\label{QQtable}
\end{table}
\ \\
\begin{figure}[h]
	\centering
	\includegraphics[width=0.32\columnwidth]{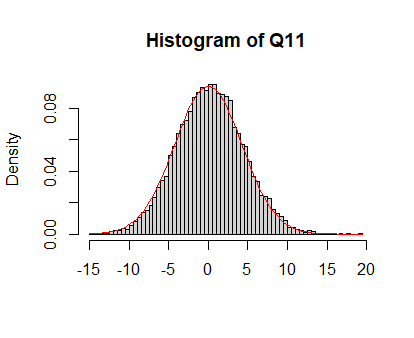}
	\includegraphics[width=0.32\columnwidth]{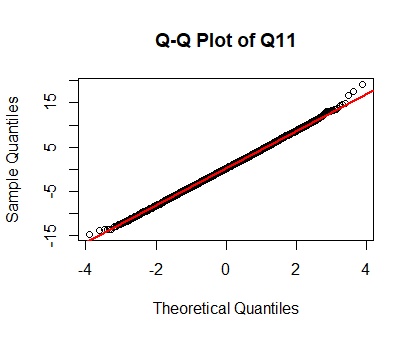}
	\includegraphics[width=0.32\columnwidth]{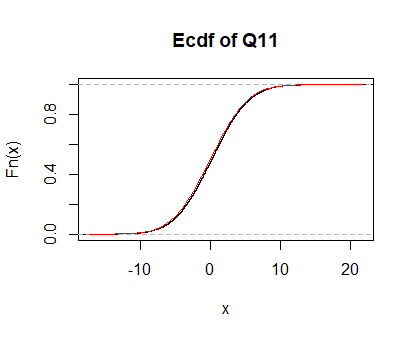}
	\caption{Histogram (left), Q-Q plot (middle) and empirical distribution (right) of $\sqrt{n}((Q_{XX})_{11}$ - $(\Sigma_{m}(\theta_{m,0}))_{11})$.}
	\label{QQfigure}
\end{figure}
\clearpage
\ \\
\ \\
\ \\
\ \\
\fontsize{12pt}{16pt}\selectfont
\begin{table}[h]
\centering
\begin{tabular}{ccccc}
	& $\hat{\theta}_{m,n}^{(1)}$ & $\hat{\theta}_{m,n}^{(2)}$ & $\hat{\theta}_{m,n}^{(3)}$ & $\hat{\theta}_{m,n}^{(4)}$ \\\hline
	Mean (True value) & 2.000 (2.000) & 3.000 (3.000) & 3.000 (3.000)& 0.999 (1.000) \\
	SD (Theoretical value) & 0.026 (0.026) & 0.336 (0.336) & 0.009 (0.008) & 0.036 (0.036) \\
	& $\hat{\theta}_{m,n}^{(5)}$ & $\hat{\theta}_{m,n}^{(6)}$ & $\hat{\theta}_{m,n}^{(7)}$ & $\hat{\theta}_{m,n}^{(8)}$\\\hline
	Mean (True value) & 2.001 (2.000) & 2.001 (2.000) & 2.000 (2.000) & 4.002 (4.000)\\
	SD (Theoretical value) & 0.030 (0.030) & 0.044 (0.044)& 0.045 (0.046) & 0.100 (0.100) \\
	& $\hat{\theta}_{m,n}^{(9)}$ & $\hat{\theta}_{m,n}^{(10)}$ & $\hat{\theta}_{m,n}^{(11)}$ & $\hat{\theta}_{m,n}^{(12)}$\\\hline
	Mean (True value) & 1.001 (1.000)  &  4.003 (4.000) & 4.004 (4.000)& 1.004 (1.000)\\
	SD (Theoretical value) & 0.024 (0.024) & 0.096 (0.096) & 0.059 (0.060) & 0.183 (0.182) \\
	& $\hat{\theta}_{m,n}^{(13)}$ & $\hat{\theta}_{m,n}^{(14)}$ & $\hat{\theta}_{m,n}^{(15)}$ & \\\hline
	Mean (True value) & 1.001 (1.000) & 9.007 (9.000) & 3.999 (4.000)&\\
	SD (Theoretical value) & 0.038 (0.038) & 0.341 (0.343) & 0.110 (0.109)\\ \\
	\end{tabular}
\caption{Sample mean and sample standard deviation (SD) of $\hat{\theta}_{m,n}$.\qquad\qquad\qquad\qquad}
\label{thetatable}
\end{table}
\ \\
\fontsize{14pt}{22pt}\selectfont
\begin{figure}[h]
	\centering
	\includegraphics[width=0.32\columnwidth]{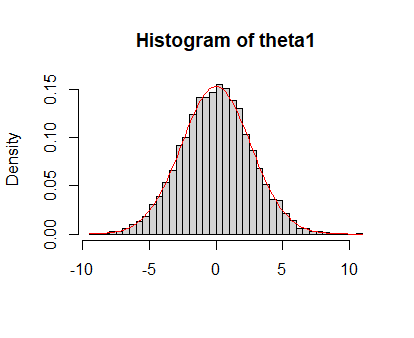}
	\includegraphics[width=0.32\columnwidth]{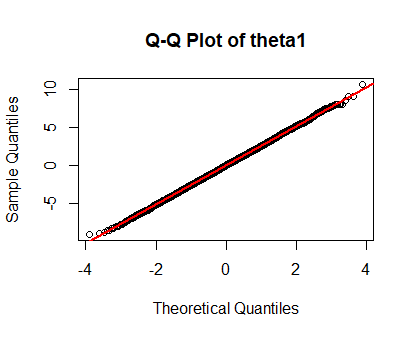}
	\includegraphics[width=0.32\columnwidth]{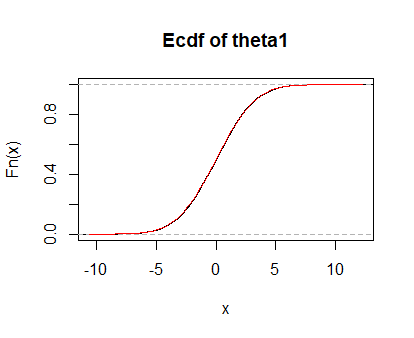}
	\caption{Histogram (left), Q-Q plot (middle) and empirical distribution (right) of $\sqrt{n}(\hat{\theta}_{m,n}^{(1)}-\theta_{m,0}^{(1)})$.}
	\label{thetafigure}
\end{figure}
\clearpage
\ \\
\ \\
\ \\
\begin{table}[h]
	\centering
	\begin{tabular}{cc}	
	& $\mathbb{T}_{M_{0},n}$ \\
	\hline
	Mean\ \  (True value) &\quad 5.980\ \ (6.000)\\
	SD\ \ (Theoretical value) &\quad 3.400\ \ (3.464)\\ \\
	\end{tabular}
\caption{Sample mean and sample standard deviation (SD) of the test statistic $\mathbb{T}_{M_{0},n}$.\qquad\qquad\qquad\qquad}
\label{testtable1}
\end{table}
\ \\
\ \\
\begin{figure}[h]
	\centering
	\includegraphics[width=0.32\columnwidth]{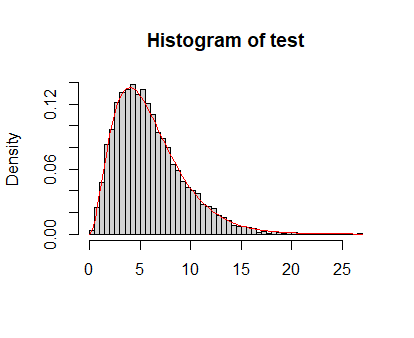}
	\includegraphics[width=0.32\columnwidth]{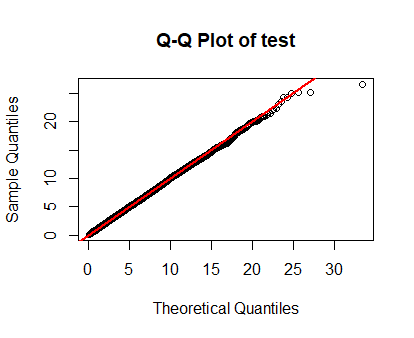}
	\includegraphics[width=0.32\columnwidth]{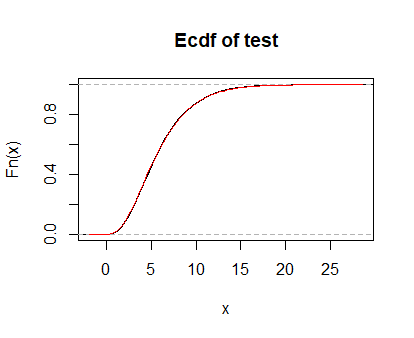}
\caption{Histogram (left), Q-Q plot (middle) and empirical distribution (right) of the test statistic $\mathbb{T}_{M_0,n}$.}
\label{testfigure1}
\end{figure}
\ \\
\ \\
\begin{table}[h]
	\centering
	\begin{tabular}{cccccc}
	&\quad Min &\  $Q1$ &\  Median &\  $Q3$ &\  Max\\\hline
	Model $M_1$ &\quad 5376 &\ 5829 &\ 5930 &\ 6035 &\ 6495  \\
	Model $M_2$ &\quad 4472 &\ 4851 &\ 4937 &\ 5021 &\ 5433 \\ \\
    \end{tabular}
\caption{Quartile of the test statistic $\mathbb{T}_{M_1,n}$ and $\mathbb{T}_{M_2,n}$.\qquad\qquad\qquad\qquad\qquad}
\label{testtable2}
\end{table}
\clearpage
\ \\
\ \\
\ \\
\begin{figure}[h]
	\includegraphics[width=0.9\columnwidth]{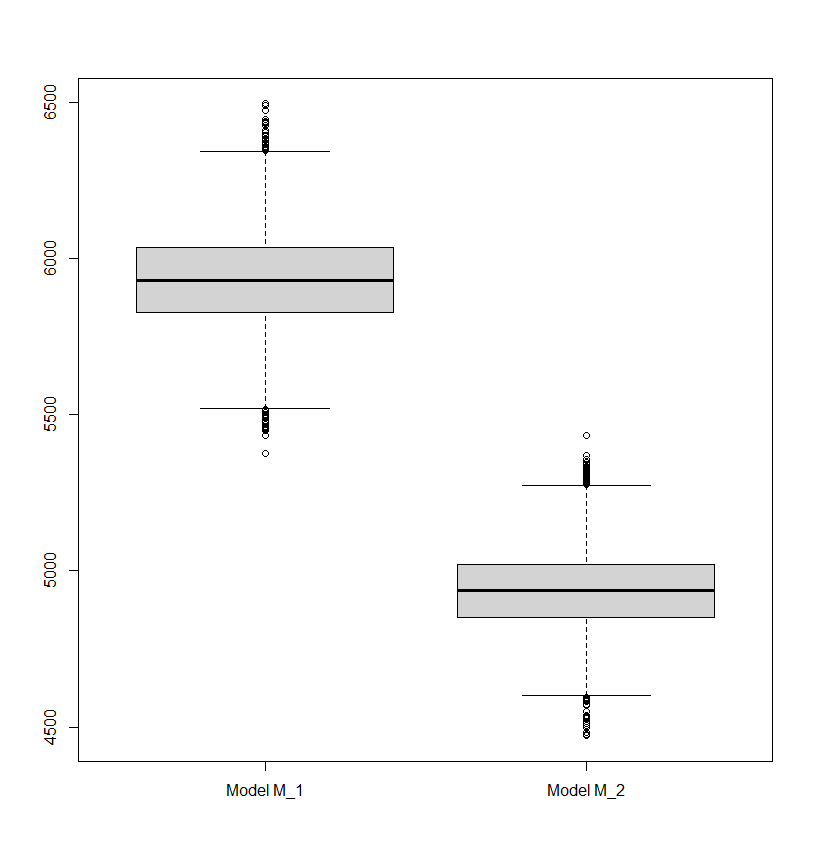}
	\centering
	\caption{Box plots of the test statistic $\mathbb{T}_{M_1,n}$ and $\mathbb{T}_{M_2,n}$.\qquad\qquad\qquad\qquad}
	\label{testfigure2}
\end{figure}
\clearpage

\fontsize{10pt}{16pt}\selectfont
\section{Proofs}
For the proof, we define the following notation. Let $\partial_{\theta_{m}}=\partial/\partial\theta_{m}$ and $\partial_{\theta_{m}}^2=\partial_{\theta_{m}}\partial_{\theta_{m}}^{\top}$. Set
\begin{align*}
	Q_{\xi\xi,m}&=\frac{1}{T}\sum_{i=1}^n(\Delta\xi_{m,i})(\Delta\xi_{m,i})^\top,\ \ Q_{\delta\delta,m}=\frac{1}{T}\sum_{i=1}^n(\Delta\delta_{m,i})(\Delta\delta_{m,i})^\top,\ \  Q_{\varepsilon\varepsilon,m}=\frac{1}{T}\sum_{i=1}^n(\Delta\varepsilon_{m,i})(\Delta\varepsilon_{m,i})^\top,\\
	Q_{\zeta\zeta,m}&=\frac{1}{T}\sum_{i=1}^n(\Delta\zeta_{m,i})(\Delta\zeta_{m,i})^\top,\ \ 
	Q_{\xi\delta,m}=\frac{1}{T}\sum_{i=1}^n(\Delta\xi_{m,i})(\Delta\delta_{m,i})^\top,\ \ Q_{\xi\varepsilon,m}=\frac{1}{T}\sum_{i=1}^n(\Delta\xi_{m,i})(\Delta\varepsilon_{m,i})^\top,\\
	Q_{\xi\zeta,m}&=\frac{1}{T}\sum_{i=1}^n(\Delta\xi_{m,i})(\Delta\zeta_{m,i})^\top,\ \ Q_{\delta\varepsilon,m}=\frac{1}{T}\sum_{i=1}^n(\Delta\delta_{m,i})(\Delta\varepsilon_{m,i})^\top,\ \ Q_{\delta\zeta,m}=\frac{1}{T}\sum_{i=1}^n(\Delta\delta_{m,i})(\Delta\zeta_{m,i})^\top,\\
	Q_{\varepsilon\zeta,m}&=\frac{1}{T}\sum_{i=1}^n(\Delta\varepsilon_{m,i})(\Delta\zeta_{m,i})^{\top},
\end{align*}
where  
\begin{align*}
    \Delta\xi_{m,i}=\xi_{m,t_{i}^n}-\xi_{m,t_{i-1}^n},\ 
    \Delta\delta_{m,i}=\delta_{m,t_{i}^n}-\delta_{m,t_{i-1}^n},\ 
    \Delta\varepsilon_{m,i}=\varepsilon_{m,t_{i}^n}-\varepsilon_{m,t_{i-1}^n},\ 
    \Delta\zeta_{m,i}=\zeta_{m,t_{i}^n}-\zeta_{m,t_{i-1}^n}.
\end{align*}
Let
\begin{align*}
		Q_{X_1X_1}&=\frac{1}{T}\sum_{i=1}^n(X_{1,t_{i}^n}-X_{1,t_{i-1}^n})(X_{1,t_{i}^n}-X_{1,t_{i-1}^n})^\top,\\
		Q_{X_1X_2}&=\frac{1}{T}\sum_{i=1}^n(X_{1,t_{i}^n}-X_{1,t_{i-1}^n})(X_{2,t_{i}^n}-X_{2,t_{i-1}^n})^\top,\\
		Q_{X_2X_2}&=\frac{1}{T}\sum_{i=1}^n(X_{2,t_{i}^n}-X_{2,t_{i-1}^n})(X_{2,t_{i}^n}-X_{2,t_{i-1}^n})^\top.
\end{align*}
Decompose $X_{1,t_{i}^n}-X_{1,t_{i-1}^n}$ as
\begin{align*}
	X_{1,t_{i}^n}-X_{1,t_{i-1}^n}=A_{i,m,n}+B_{i,m,n},
\end{align*}
where $A_{i,m,n}=\Lambda_{x_1,m}\Delta\xi_{m,i}$ and $B_{i,m,n}=\Delta\delta_{m,i}$. Noting that
\begin{align*}
    X_{2,t}=\Lambda_{x_2,m}\Psi_{m}^{-1}\Gamma_{m}\xi_{m,t}+\Lambda_{x_2,m}\Psi_{m}^{-1}\zeta_{m,t}+\varepsilon_{m,t},
\end{align*}
we decompose $X_{2,t_{i}^n}-X_{2,t_{i-1}^n}$ as
\begin{align*}
	X_{2,t_{i}^n}-X_{2,t_{i-1}^n}=C_{i,m,n}+D_{i,m,n}+E_{i,m,n},
\end{align*}
where $C_{i,m,n}=\Lambda_{x_2,m}\Psi_{m}^{-1}\Gamma_{m}\Delta\xi_{m,i}$, $D_{i,m,n}=\Lambda_{x_2,m}\Psi_{m}^{-1}\Delta\zeta_{m,i}$ and $E_{i,m,n}=\Delta\varepsilon_{m,i}$. Let
\begin{align*}
    \tilde{V}_{m,n}(\theta_{m,0})=\tilde{V}(Q_{XX},\Sigma_{m}(\theta_{m,0})).
\end{align*}
\begin{lemma}\label{Qlemma}
Under 
$\bf{[A1]}$-$\bf{[A2]}$, $\bf{[B1]}$-$\bf{[B3]}$, $\bf{[C1]}$-$\bf{[C3]}$ and $\bf{[D1]}$-$\bf{[D2]}$,
as $h_n\longrightarrow 0$ and $nh_n\longrightarrow \infty$,
\begin{align*}
	Q_{\xi\xi,m}\stackrel{P_{\theta_{m}\ } }{\longrightarrow}\Sigma_{\xi\xi,m}, \  
	Q_{\delta\delta,m}\stackrel{P_{\theta_{m}}}{\longrightarrow} \Sigma_{\delta\delta,m}, \ 
	Q_{\varepsilon\varepsilon,m}\stackrel{P_{\theta_{m}\ } }{\longrightarrow}
	\Sigma_{\varepsilon\varepsilon,m},\
	Q_{\zeta\zeta,m}\stackrel{P_{\theta_{m}\ } }{\longrightarrow} \Sigma_{\zeta\zeta,m},\qquad\quad\\
	Q_{\xi\delta,m}\stackrel{P_{\theta_{m}\ } }{\longrightarrow} 0,\
	Q_{\xi\varepsilon,m}\stackrel{P_{\theta_{m}\ } }{\longrightarrow} 0,\
	Q_{\xi\zeta,m}\stackrel{P_{\theta_{m}\ } }{\longrightarrow}
	0,\  Q_{\delta\varepsilon}\stackrel{P_{\theta_{m}\ } }{\longrightarrow} 0,\ Q_{\delta\zeta,m} \stackrel{P_{\theta_{m}\ } }{\longrightarrow}0, \  Q_{\varepsilon\zeta,m}\stackrel{P_{\theta_{m}\ } }{\longrightarrow} 0.
\end{align*}
\end{lemma}
\begin{proof}
The results can be shown in a similar way to Lemma 1 in Kusano and Uchida \cite{Kusano(2022)}.
\end{proof} 
\begin{lemma}\label{Xlemma}
Under 
$\bf{[A1]}$-$\bf{[A2]}$, $\bf{[B1]}$-$\bf{[B3]}$, $\bf{[C1]}$-$\bf{[C3]}$, $\bf{[D1]}$-$\bf{[D2]}$, 
$\bf{[E]}$ and $\bf{[F]}$,
as $h_n\longrightarrow 0$ and $nh_n\longrightarrow \infty$, 
\begin{align}
	\begin{split}
	\qquad\sum_{i=1}^n\E_{\theta_{m}}\left[\left\{\frac{1}{\sqrt{n}h_n}(X_{t_{i}^n}^{(j_1)}-X_{t_{i-1}^n}^{(j_1)})(X_{t_{i}^n}^{(j_2)}-X_{t_{i-1}^n}^{(j_2)})-\frac{1}{\sqrt{n}}(\Sigma_{m}(\theta_{m}))_{j_1j_2}\right\}|\mathscr{F}^{n}_{i-1}\right]\stackrel{P_{\theta_{m}\ } }{\longrightarrow}0\qquad\ \ \label{EXX}
	\end{split}
\end{align}
for $j_1,j_2=1,\cdots,p$,
\begin{align}
	\begin{split}
	&\sum_{i=1}^n\E_{\theta_{m}}\left[\frac{1}{\sqrt{n}h_n}(X_{t_{i}^n}^{(j_1)}-X_{t_{i-1}^n}^{(j_1)})(X_{t_{i}^n}^{(j_2)}-X_{t_{i-1}^n}^{(j_2)})-\frac{1}{\sqrt{n}}(\Sigma_{m}(\theta_{m}))_{j_1j_2}|\mathscr{F}^{n}_{i-1}\right]\\
	&\qquad\quad\times\E_{\theta_{m}}\left[\frac{1}{\sqrt{n}h_n}(X_{t_{i}^n}^{(j_3)}-X_{t_{i-1}^n}^{(j_3)})(X_{t_{i}^n}^{(j_4)}-X_{t_{i-1}^n}^{(j_4)})-\frac{1}{\sqrt{n}}(\Sigma_{m}(\theta_{m}))_{j_3j_4}|\mathscr{F}^{n}_{i-1}\right]\stackrel{P_{\theta_{m}\ } }{\longrightarrow} 0\label{EXXEXX}
	\end{split}
\end{align}
for $j_1,j_2,j_3,j_4=1,\cdots,p$,
\begin{align}
	\begin{split}
	&\sum_{i=1}^n\E_{\theta_{m}}\left[\left\{\frac{1}{\sqrt{n}h_n}(X_{t_{i}^n}^{(j_1)}-X_{t_{i-1}^n}^{(j_1)})(X_{t_{i}^n}^{(j_2)}-X_{t_{i-1}^n}^{(j_2)})-\frac{1}{\sqrt{n}}(\Sigma_{m}(\theta_{m}))_{j_1j_2}\right\}\right.\\
	&\qquad\qquad\qquad\quad\times\left.\left\{\frac{1}{\sqrt{n}h_n}(X_{t_{i}^n}^{(j_3)}-X_{t_{i-1}^n}^{(j_3)})(X_{t_{i}^n}^{(j_4)}-X_{t_{i-1}^n}^{(j_4)})-\frac{1}{\sqrt{n}}(\Sigma_{m}(\theta_{m}))_{j_3j_4}\right\}|\mathscr{F}^{n}_{i-1}\right]\\
	&\stackrel{P_{\theta_{m}\ } }{\longrightarrow} (\Sigma_{m}(\theta_{m}))_{j_1j_3}(\Sigma_{m}(\theta_{m}))_{j_2j_4}+(\Sigma_{m}(\theta_{m}))_{j_1j_4}(\Sigma_{m}(\theta_{m}))_{j_2j_3}\label{EXXXX}
\end{split}
\end{align}
for $j_1,j_2,j_3,j_4=1,\cdots,p$, and
\begin{align}
	\begin{split}
	&\sum_{i=1}^n\E_{\theta_{m}}\left[\left|\frac{1}{\sqrt{n}h_n}(X^{(j_1)}_{t_{i}^n}-X^{(j_1)}_{t_{i-1}^n})(X^{(j_2)}_{t_{i}^n}-X^{(j_2)}_{t_{i-1}^n})-\frac{1}{\sqrt{n}}(\Sigma_{m}(\theta_{m}))_{j_1j_2}\right|^4|\mathscr{F}^{n}_{i-1}\right]\stackrel{P_{\theta_{m}\ } }{\longrightarrow}0\label{EXX4}
	\end{split}
\end{align}
for $j_1,j_2=1,\cdots p$.
\end{lemma}
\begin{proof}
See Appendix \ref{ProofXlemma}.
\end{proof}
\begin{proof}[Proof of Theorem 1.]
We first show
\begin{align}
    Q_{XX}\stackrel{P_{\theta_{m}\ } }{\longrightarrow}\Sigma_{m}(\theta_{m}).\label{Qxxprob}
\end{align}
Recall that $Q_{X_2X_1}=Q_{X_1X_2}^{\top}$.
In order to show (\ref{Qxxprob}),
it is sufficient to show that
\begin{align}
	Q_{X_1X_1}&\stackrel{P_{\theta_{m}\ } }{\longrightarrow}\Lambda_{x_1,m}\Sigma_{\xi\xi,m}\Lambda_{x_1,m}^{\top}+\Sigma_{\delta\delta,m}\ \bigl(=\Sigma_{X_1X_1,m}(\theta_{m})\bigr),\label{Qx1x1cons}\\
	Q_{X_1X_2}&\stackrel{P_{\theta_{m}\ } }{\longrightarrow}
	\Lambda_{x_1,m}\Sigma_{\xi\xi,m}\Gamma_{m}^{\top}\Psi_{m}^{-1\top}\Lambda_{x_2,m}^{\top}\ \bigl(=\Sigma_{X_1X_2,m}(\theta_{m})\bigr),\label{Qx1x2cons}\\
	Q_{X_2X_2}&\stackrel{P_{\theta_{m}\ } }{\longrightarrow}\Lambda_{x_2,m}\Psi^{-1}_{m}(\Gamma_{m}\Sigma_{\xi\xi,m}\Gamma_{m}^{\top}+\Sigma_{\zeta\zeta,m})\Psi^{-1\top}_{m}\Lambda_{x_2,m}^{\top}+\Sigma_{\varepsilon\varepsilon,m}\ \bigl(=\Sigma_{X_2X_2,m}(\theta_{m})\bigr). \label{Qx2x2cons}
\end{align}
Using Lemma \ref{Qlemma} and Slutsky's theorem, one gets
\begin{align*}
	Q_{X_1X_1}
	&=\frac{1}{T}\sum_{i=1}^n\bigl\{\Lambda_{x_1,m}(\xi_{m,t_{i}^n}-\xi_{m,t_{i-1}^n})+\delta_{m,t_{i}^n}-\delta_{m,t_{i-1}^n}\bigr\}\bigl\{\Lambda_{x_1,m}(\xi_{m,t_{i}^n}-\xi_{m,t_{i-1}^n})+\delta_{m,t_{i}^n}-\delta_{m,t_{i-1}^n}\bigr\}^{\top}\\
	&=\Lambda_{x_1,m}Q_{\xi\xi,m}\Lambda_{x_1,m}^{\top}+\Lambda_{x_1,m}Q_{\xi\delta,m}+Q_{\xi\delta,m}^{\top}\Lambda_{x_1,m}^{\top}+Q_{\delta\delta,m}
	\stackrel{P_{\theta_{m}\ } }{\longrightarrow}\Lambda_{x_1,m}\Sigma_{\xi\xi,m}\Lambda_{x_1,m}^{\top}+\Sigma_{\delta\delta,m},\qquad
\end{align*}
which yields (\ref{Qx1x1cons}). In the same way, since
\begin{align*}
	Q_{X_1X_2}&=\Lambda_{x_1,m}Q_{\xi\xi,m}\Gamma_{m}^{\top}\Psi_{m}^{-1\top}\Lambda_{x_2,m}^{\top}+\Lambda_{x_1,m}Q_{\xi\zeta,m}\Psi_{m}^{-1\top}\Lambda_{x_2,m}^{\top}+\Lambda_{x_1,m}Q_{\xi\varepsilon,m}\\
	&\quad+Q_{\xi\delta,m}^{\top}\Gamma_{m}^{\top}\Psi_{m}^{-1\top}\Lambda_{x_2,m}^{\top}+Q_{\delta\zeta,m}\Psi_{m}^{-1\top}\Lambda_{x_2,m}^{\top}+Q_{\delta\varepsilon,m}\stackrel{P_{\theta_{m}\ } }{\longrightarrow}
	\Lambda_{x_1,m}\Sigma_{\xi\xi,m}\Gamma_{m}^{\top}\Psi_{m}^{-1\top}\Lambda_{x_2,m}^{\top},\qquad
\end{align*}
and
\begin{align*}
	Q_{X_2X_2}&=\Lambda_{x_2,m}\Psi_{m}^{-1}\Gamma_{m}Q_{\xi\xi,m}\Gamma_{m}^{\top}\Psi_{m}^{-1\top}\Lambda_{x_2,m}^{\top}+\Lambda_{x_2,m}\Psi_{m}^{-1}\Gamma_{m}Q_{\xi\zeta,m}\Psi_{m}^{-1\top}\Lambda_{x_2,m}^{\top}\\
	&\quad+\Lambda_{x_2,m}\Psi_{m}^{-1}\Gamma_{m}Q_{\xi\varepsilon,m}+\Lambda_{x_2,m}\Psi_{m}^{-1}Q_{\xi\zeta,m}^{\top}\Gamma_{m}^{\top}\Psi_{m}^{-1\top}\Lambda_{x_2,m}^{\top}\\
	&\quad+\Lambda_{x_2,m}\Psi_{m}^{-1}Q_{\zeta\zeta,m}\Psi_{m}^{-1\top}\Lambda_{x_2,m}^{\top}+\Lambda_{x_2,m}\Psi_{m}^{-1}Q_{\varepsilon\zeta,m}^{\top}+Q_{\varepsilon\xi,m}\Gamma_{m}^{\top}\Psi_{m}^{-1\top}\Lambda_{x_2,m}^{\top}\\
	&\quad+Q_{\varepsilon\zeta,m}\Psi_{m}^{-1\top}\Lambda_{x_2,m}^{\top}+Q_{\varepsilon\varepsilon,m}\stackrel{P_{\theta_{m}\ } }{\longrightarrow}\Lambda_{x_2,m}\Psi^{-1}_{m}(\Gamma_{m}\Sigma_{\xi\xi,m}\Gamma_{m}^{\top}+\Sigma_{\zeta\zeta,m})\Psi^{-1\top}_{m}\Lambda_{x_2,m}^{\top}+\Sigma_{\varepsilon\varepsilon,m},
\end{align*}
we obtain (\ref{Qx1x2cons}) and (\ref{Qx2x2cons}).
		
Next, we prove
\begin{align}
	\sqrt{n}(\vech{Q_{XX}}-\vech{\Sigma_{m}(\theta_{m})})\stackrel{d}{\longrightarrow} N_{\bar{p}}(0,W_{m}(\theta_{m})).\label{vechasym}
\end{align}
Consider the following convergence:
\begin{align}
	\sqrt{n}(\vec{Q_{XX}}-\vec{\Sigma_{m}(\theta_{m})})\stackrel{d}{\longrightarrow} N_{p^2}(0,\Gamma_{m}(\theta_{m})),\label{vecasym}
\end{align}
where 
\begin{align*}
	\Gamma_{m}(\theta_{m})_{p(j_1-1)+j_2,\ p(j_3-1)+j_4}=(\Sigma_{m}(\theta_{m}))_{j_1j_3}(\Sigma_{m}(\theta_{m}))_{j_2j_4}+(\Sigma_{m}(\theta_{m}))_{j_1j_4}(\Sigma_{m}(\theta_{m}))_{j_2j_3}
\end{align*}
for $j_1,j_2,j_3,j_4=1,\cdots,p$.
If (\ref{vecasym}) holds, then it follows from the continuous mapping theorem that
\begin{align}
	\begin{split}
	\sqrt{n}(\vech{Q_{XX}}-\vech{\Sigma_{m}(\theta_{m})})&=f(\sqrt{n}(\vec{Q_{XX}}-\vec{\Sigma_{m}(\theta_{m})}))\\
	&\stackrel{d}{\longrightarrow} f(N_{p^2}(0,\Gamma_{m}(\theta_{m})))\sim N_{\bar{p}}(0,\mathbb{D}_{p}^{+}\Gamma_{m}(\theta_{m})\mathbb{D}_{p}^{+\top}), \label{veccon}
	\end{split}
\end{align}
where $f(x)=\mathbb{D}_{p}^{+}x$ for $x\in\mathbb{R}^{\bar{p}}$. In an analogous manner to Lemma 6 in Kusano and Uchida \cite{Kusano(2022)},
\begin{align*}
	\mathbb{D}_{p}^{+}\Gamma_{m}(\theta_{m})\mathbb{D}_{p}^{+\top}=W_{m}(\theta_{m}),
\end{align*}
so that we obtain (\ref{vechasym}) from (\ref{veccon}). Consequently, it is sufficient to prove (\ref{vecasym}) in order to prove (\ref{vechasym}). Let
\begin{align*}
	L_{i,m,n}&=\frac{1}{\sqrt{n}h_n}\vec{(X_{t_{i}^n}-X_{t_{i-1}^n})(X_{t_{i}^n}-X_{t_{i-1}^n})^{\top}}-\frac{1}{\sqrt{n}}\vec{\Sigma_{m}(\theta_{m})}.
\end{align*}
The left side of (\ref{vecasym}) is expressed as
\begin{align*}
	\sqrt{n}(\vec Q_{XX}-\vec\Sigma_{m}(\theta_{m}))=\sum_{i=1}^n L_{i,m,n}.
\end{align*}
In a similar way to Lemma 5 in Kessler \cite{kessler(1997)}, if it holds that
\begin{align}
	\label{L}
	&\qquad\qquad\qquad\qquad\qquad\qquad\sum_{i=1}^n\E_{\theta_{m}}\left[L_{i,m,n}| \mathscr{F}^{n}_{i-1}\right]\stackrel{P_{\theta_{m}\ } }{\longrightarrow}0,\\
	\label{LL}
	\begin{split}
	&\sum_{i=1}^n\E_{\theta_{m}}\left[L_{i,m,n}L_{i,m,n}^\top| \mathscr{F}^{n}_{i-1}\right]-\sum_{i=1}^n\E_{\theta_{m}}\left[L_{i,m,n}| \mathscr{F}^{n}_{i-1}\right]
	\E_{\theta_{m}}\left[L_{i,m,n}| \mathscr{F}^{n}_{i-1}\right]^\top
	\stackrel{P_{\theta_{m}\ } }{\longrightarrow}\Gamma_{m}(\theta_{m}),
	\end{split}\\
	\label{L4}
	&\qquad\qquad\qquad\qquad\qquad\quad\ \ \sum_{i=1}^n\E_{\theta_{m}}\left[|L_{i,m,n}|^4|\mathscr{F}^{n}_{i-1}\right]\stackrel{P_{\theta_{m}\ } }{\longrightarrow}0,
\end{align}
then we can obtain (\ref{vecasym}) from Theorems 3.2 and 3.4 in Hall and Heyde \cite{Hall(1981)}. 
(\ref{EXX}) yields (\ref{L}). We see from 
(\ref{EXXXX}) that
\begin{align*}
	\sum_{i=1}^n\E_{\theta_{m}}\left[L_{i,m,n}L_{i,m,n}^\top| \mathscr{F}^{n}_{i-1}\right]\stackrel{P_{\theta_{m}\ }}{\longrightarrow}\Gamma_{m}(\theta_{m}),
\end{align*}
and it follows from 
(\ref{EXXEXX}) that
\begin{align*}
	\sum_{i=1}^n\E_{\theta_{m}}\left[L_{i,m,n}| \mathscr{F}^{n}_{i-1}\right]\E_{\theta_{m}}\left[L_{i,m,n}| \mathscr{F}^{n}_{i-1}\right]^{\top}\stackrel{P_{\theta_{m}\ } }{\longrightarrow} 0.
\end{align*}
Thus, Slutsky's theorem implies (\ref{LL}). Finally, we prove (\ref{L4}). 
Note that one has
\begin{align}
	\begin{split}
	0&\leq \sum_{i=1}^n\E_{\theta_{m}}\left[|L_{i,m,n}|^4| \mathscr{F}^{n}_{i-1}\right]\\
	&\qquad\qquad\leq\sum_{i=1}^n\E_{\theta_{m}}\left[\Bigl|\sum_{u=1}^{p^2}L_{i,m,n}^{(u)2}\Bigr|^2| \mathscr{F}^{n}_{i-1}\right]\leq C_{p}\sum_{u=1}^{p^2}\sum_{i=1}^n\E_{\theta_{m}}\left[|L_{i,m,n}^{(u)}|^4|\mathscr{F}^{n}_{i-1}\right]. 
	\label{L4ine}
	\end{split}
\end{align}
From 
(\ref{EXX4}), we have
\begin{align*}
	\sum_{i=1}^n\E_{\theta_{m}}\left[|L_{i,m,n}^{(u)}|^4|\mathscr{F}^{n}_{i-1}\right]\stackrel{P_{\theta_{m}\ } }{\longrightarrow} 0
\end{align*}
for $u=1,\cdots,p^2$, so that 
(\ref{L4}) holds from (\ref{L4ine}). 
Therefore, we obtain (\ref{vecasym}).
\end{proof}
\begin{lemma}\label{Sigmaposlemma}
Under 
$\bf{[B2]}$, $\bf{[C2]}$, $\bf{[E]}$ and $\bf{[F]}$, $\Sigma_{m}(\theta_{m})$ is a positive definite matrix.
\end{lemma}
\begin{proof}
See Appendix \ref{Sigmaposproof}.
\end{proof}
\begin{lemma}\label{Vposlemma}
Under 
$\bf{[B2]}$, $\bf{[C2]}$, $\bf{[E]}$ and $\bf{[F]}$, for positive definite matrices $X\in\mathbb{R}^{p\times p}$ and $Y\in\mathbb{R}^{p\times p}$, $V(X,Y)$ is a positive definite matrix.
\end{lemma}
\begin{proof}
See Appendix \ref{Vposproof}.
\end{proof}
\begin{lemma}\label{Fproblemma}
Under 
$\bf{[A1]}$-$\bf{[A2]}$, $\bf{[B1]}$-$\bf{[B3]}$, $\bf{[C1]}$-$\bf{[C3]}$, $\bf{[D1]}$-$\bf{[D2]}$, $\bf{[E]}$ and $\bf{[F]}$, 
as $h_n\longrightarrow 0$ and $nh_n\longrightarrow\infty$,
\begin{align} 
	\tilde{F}(Q_{XX},\Sigma_{m}(\theta_{m}))&\stackrel{P_{\theta_{m,0}\ } }{\longrightarrow}F(\Sigma_{m}(\theta_{m,0}),\Sigma_{m}(\theta_{m}))\quad\mbox{uniformly in $\theta_m$},\label{Fprob}\\
	\begin{split}
	\partial_{\theta_{m}}^2\tilde{F}(Q_{XX},\Sigma_{m}(\theta_{m}))&\stackrel{P_{\theta_{m,0}\ } }{\longrightarrow}\partial_{\theta_{m}}^2 F(\Sigma_{m}(\theta_{m,0}),\Sigma_{m}(\theta_{m}))\quad\mbox{uniformly in $\theta_{m}$}.\label{F2prob}
	\end{split}
\end{align}
\end{lemma}
\begin{proof}
See Appendix \ref{Fprobproof}.
\end{proof}
\begin{lemma}\label{Vproblemma}
Under 
$\bf{[A1]}$-$\bf{[A2]}$, $\bf{[B1]}$-$\bf{[B3]}$, $\bf{[C1]}$-$\bf{[C3]}$, $\bf{[D1]}$-$\bf{[D2]}$, $\bf{[E]}$ and $\bf{[F]}$,
as $h_n\longrightarrow 0$ and $nh_n\longrightarrow\infty$, 
\begin{align}
	\tilde{V}_{m,n}(\theta_{m,0})&\stackrel{P_{\theta_{m,0}\ } }{\longrightarrow}W_{m}(\theta_{m,0})^{-1},\label{Vprob}\\
	\partial_{\theta_{m}^{(i)}}\tilde{V}_{m,n}(\theta_{m,0})&\stackrel{P_{\theta_{m,0}\ } }{\longrightarrow}\partial_{\theta_{m}^{(i)}} W_{m}(\theta_{m,0})^{-1}\label{V1prob}
\end{align}
for $i=1,\cdots,q_m$.
\end{lemma}
\begin{proof}
The results can be shown in an analogous manner to Lemma \ref{Fproblemma}.
\end{proof}
\begin{lemma}\label{det}
Under 
$\bf{[B2]}$, $\bf{[C2]}$, $\bf{[E]}$, $\bf{[F]}$ and $\bf{[H]}$, 
\begin{align*}
	\det\bigl\{\Delta_{m}^{\top}W_m(\theta_{m,0})^{-1}\Delta_{m}\bigr\} \neq 0.
\end{align*}
\end{lemma}
\begin{proof}
The result can be shown in the same way as Lemma 6 in Kusano and Uchida \cite{Kusano(2022)}.
\end{proof}
\begin{proof}[Proof of Theorem 2.]
We first prove 
\begin{align}
	\hat{\theta}_{m,n}\stackrel{P_{\theta_{m,0}\ }}{\longrightarrow}\theta_{m,0}.\label{thetacons}
\end{align}
$\bf{[G]}$ and Lemmas \ref{Sigmaposlemma} and \ref{Vposlemma} yield
\begin{align*}
	\begin{split}
	F(\Sigma_{m}(\theta_{m,0}),\Sigma_{m}(\theta_{m}))=0&\Longleftrightarrow\vech\Sigma_{m}(\theta_{m,0})-\vech\Sigma_{m}(\theta_{m})=0\Longleftrightarrow \theta_{m,0}=\theta_{m}.
	\end{split}
\end{align*}
For any $\varepsilon>0$, there exists $\delta>0$ such that
\begin{align*}
	|\hat{\theta}_{m,n}-\theta_{m,0}|>\varepsilon\Longrightarrow F(\Sigma_{m}(\theta_{m,0}),\Sigma_{m}(\hat{\theta}_{m,n}))-F(\Sigma_{m}(\theta_{m,0}),\Sigma_{m}(\theta_{m,0}))>\delta.
\end{align*} 
From the definition of $\hat{\theta}_{m,n}$, 
\begin{align*}
	\tilde{F}(Q_{XX},\Sigma_{m}(\hat{\theta}_{m,n}))=\mathbb{F}_{m,n}(\hat{\theta}_{m,n})\leq\mathbb{F}_{m,n}(\theta_{m})= F(Q_{XX},\Sigma_{m}(\theta_{m})).
\end{align*}
It follows from
(\ref{Fprob}) that  
\begin{align*}
	0&\leq \PP_{\theta_{m,0}}\left(|\hat{\theta}_{m,n}-\theta_{m,0}|>\varepsilon\right)\\
	&\leq\PP_{\theta_{m,0}}\Bigl(F(\Sigma_{m}(\theta_{m,0}),\Sigma_{m}(\hat{\theta}_{m,n}))-F(\Sigma_{m}(\theta_{m,0}),\Sigma_{m}(\theta_{m,0}))>\delta\Bigr)\\
	&\leq\PP_{\theta_{m,0}}\left(F(\Sigma_{m}(\theta_{m,0}),\Sigma_{m}(\hat{\theta}_{m,n}))-\tilde{F}(Q_{XX},\Sigma_{m}(\hat{\theta}_{m,n}))>\frac{\delta}{3}\right)\\
	&\quad+\PP_{\theta_{m,0}}\left(\tilde{F}(Q_{XX},\Sigma_{m}(\hat{\theta}_{m,n}))-\tilde{F}(Q_{XX},\Sigma_{m}(\theta_{m,0}))>\frac{\delta}{3}\right)\\
	&\quad+\PP_{\theta_{m,0}}\left(\tilde{F}(Q_{XX},\Sigma_{m}(\theta_{m,0}))-F(\Sigma_{m}(\theta_{m,0}),\Sigma_{m}(\theta_{m,0}))>\frac{\delta}{3}\right)\\
	&\leq 2\PP_{\theta_{m,0}}\left(\sup_{\theta_{m}\in\Theta_{m}}\left|\tilde{F}(Q_{XX},\Sigma_{m}(\theta_{m}))-F(\Sigma_{m}(\theta_{m,0}),\Sigma_{m}(\theta_{m}))\right|>\frac{\delta}{3}\right)+0\stackrel{}{\longrightarrow}0
\end{align*}
as $n\longrightarrow\infty$, 
which yields (\ref{thetacons}).

Next, we prove
\begin{align*}
	\sqrt{n}(\hat{\theta}_{m,n}-\theta_{m,0})\stackrel{d}{\longrightarrow}N_{q_m}\bigl(0,(\Delta_{m}^{\top} W_{m}(\theta_{m,0})^{-1}\Delta_{m})^{-1}\bigr).
\end{align*}
The Taylor expansion of $\partial_{\theta_{m}}\mathbb{F}_{m,n}(\hat{\theta}_{m,n})$ around $\hat{\theta}_{m,n}=\theta_{m,0}$ is given by
\begin{align*}
	\begin{split}
	\partial_{\theta_{m}}\mathbb{F}_{m,n}(\hat{\theta}_{m,n})&=\partial_{\theta_{m}}\mathbb{F}_{m,n}(\theta_{m,0})+\int_{0}^{1}\partial^2_{\theta_{m}}\mathbb{F}_{m,n}(\tilde{\theta}_{m,n})d\lambda(\hat{\theta}_{m,n}-\theta_{m,0}),
	\end{split}
\end{align*}
where $\tilde{\theta}_{m,n}=\theta_{m,0}+\lambda(\hat{\theta}_{m,n}-\theta_{m,0})$. Since $\partial_{\theta_{m}}\mathbb{F}_{m,n}(\hat{\theta}_{m,n})=0$ from the definition of $\hat{\theta}_{m,n}$, one gets
\begin{align}
	-\sqrt{n}\partial_{\theta_{m}}\mathbb{F}_{m,n}(\theta_{m,0})=\int_{0}^{1}\partial^2_{\theta_{m}}\mathbb{F}_{m,n}(\tilde{\theta}_{m,n})d\lambda\sqrt{n}(\hat{\theta}_{m,n}-\theta_{m,0}).\label{thetataylor}
\end{align}
Theorem 1 and 
(\ref{V1prob}) imply that 
the left-hand side of (\ref{thetataylor}) is given by
\begin{align*}
	\begin{split}
	-\sqrt{n}\partial_{\theta_{m}^{(i)}}\mathbb{F}_{m,n}(\theta_{m,0})
	&=2\bigl\{\partial_{\theta_{m}^{(i)}}\vech\Sigma_{m}(\theta_{m,0})\bigr\}^{\top}\tilde{V}_{m,n}(\theta_{m,0})\sqrt{n}(\vech Q_{XX}-\vech\Sigma_{m}(\theta_{m,0}))\\
	&\quad-(\vech Q_{XX}-\vech\Sigma_{m}(\theta_{m,0}))^{\top}\partial_{\theta_{m}^{(i)}}\tilde{V}_{m,n}(\theta_{m,0}))\sqrt{n}(\vech Q_{XX}-\vech\Sigma_{m}(\theta_{m,0}))\\
	&=2\bigl\{\partial_{\theta_{m}^{(i)}}\vech\Sigma_{m}(\theta_{m,0})\bigr\}^{\top}\tilde{V}_{m,n}(\theta_{m,0})\sqrt{n}(\vech Q_{XX}-\vech\Sigma_{m}(\theta_{m,0}))+o_{p}(1)
	\end{split}
\end{align*}
for $i=1,\cdots,q_{m}$.
Thus, it follows from Theorem 1 and (\ref{Vprob}) that
\begin{align}
	\begin{split}
	-\sqrt{n}\partial_{\theta_{m}}\mathbb{F}_{m,n}(\theta_{m,0})&=2\Delta_{m}^{\top}\tilde{V}_{m,n}(\theta_{m,0})\sqrt{n}(\vech Q_{XX}-\vech\Sigma_{m}(\theta_{m,0}))+o_{p}(1)\\
	&\stackrel{d}{\longrightarrow}2\Delta_{m}^{\top}W_m(\theta_{m,0})^{-1}N_{\bar{p}}\bigl(0,W_m(\theta_{m,0})\bigr)\sim N_{q_m}\bigl(0,4\Delta_{m}^{\top}W_m(\theta_{m,0})^{-1}\Delta_{m}\bigr).\label{partialFprob}
	\end{split}
\end{align}
Set
\begin{align*}
	A_{m,n}=\Bigl\{|\hat{\theta}_{m,n}-\theta_{m,0}|\leq\rho_n\Bigr\},
\end{align*}
where $\{\rho_n\}_{n\in\mathbb{N}}$ is a positive sequence such that  $\rho_n\longrightarrow0$ as $n\longrightarrow\infty$. Note that $\partial_{\theta_{m}}^2F$ is uniform continuous in $\theta_{m}$ on $\Theta_{m}$ since $\partial_{\theta_{m}}^2F$ is continuous in $\theta_{m}$ and $\Theta_{m}$ is a compact set. As it holds that
\begin{align*}
	\partial_{\theta_{m}}^2F(\Sigma_{m}(\theta_{m,0}),\Sigma_{m}(\theta_{m,0}))=2\Delta_{m}^{\top}W_{m}(\theta_{m,0})^{-1}\Delta_{m},
\end{align*}
we see
\begin{align}
	\sup_{|\theta_{m}-\theta_{m,0}|\leq\rho_n}\Bigl\|\partial_{\theta_{m}}^2F(\Sigma_{m}(\theta_{m,0}),\Sigma_{m}(\theta_{m}))-2\Delta_{m}^{\top}W_{m}(\theta_{m,0})^{-1}\Delta_{m}\Bigr\|\longrightarrow 0\label{uniprob}
\end{align}
as $n\longrightarrow\infty$. Hence, 
we see from (\ref{F2prob}), (\ref{thetacons}) and (\ref{uniprob})
that for any $\varepsilon>0$,
\begin{align*}
	0&\leq \PP_{\theta_{m,0}}\left(\left\|\int_{0}^{1}\partial_{\theta_{m}}^2\mathbb{F}_{m,n}(\tilde{\theta}_{m,n})d\lambda-2\Delta_{m}^{\top}W_{m}(\theta_{m,0})^{-1}\Delta_{m}\right\|>\varepsilon\right)\\
	&\leq \PP_{\theta_{m,0}}\left(\left\{\left\|\int_{0}^{1}\Bigr\{\partial_{\theta_{m}}^2\tilde{F}(Q_{XX},\Sigma_{m}(\tilde{\theta}_{m,n}))-2\Delta_{m}^{\top}W_{m}(\theta_{m,0})^{-1}\Delta_{m}\Bigl\}d\lambda\right\|>\varepsilon\right\}\cap A_{m,n}\right)\\
	&\quad +\PP_{\theta_{m,0}}\left(\left\{\left\|\int_{0}^{1}\Bigr\{\partial_{\theta_{m}}^2\tilde{F}(Q_{XX},\Sigma_{m}(\tilde{\theta}_{m,n}))-2\Delta_{m}^{\top}W_{m}(\theta_{m,0})^{-1}\Delta_{m}\Bigl\}d\lambda\right\|>\varepsilon\right\}\cap A_{m,n}^{c}\right)\\
	&\leq \PP_{\theta_{m,0}}\left(\sup_{|\theta_{m}-\theta_{m,0}|\leq\rho_n}\Bigl\|\partial_{\theta_{m}}^2\tilde{F}(Q_{XX},\Sigma_{m}(\theta_{m}))-2\Delta_{m}^{\top}W_{m}(\theta_{m,0})^{-1}\Delta_{m}\Bigr\|>\varepsilon\right)+\PP_{\theta_{m,0}}\bigl(A_{m,n}^{c}\bigr)\\
	&\leq\PP_{\theta_{m,0}}\left(\sup_{|\theta_{m}-\theta_{m,0}|\leq\rho_n}\Bigl\||\partial_{\theta_{m}}^2\tilde{F}(Q_{XX},\Sigma_{m}(\theta_{m}))
	-\partial_{\theta_{m}}^2F(\Sigma_{m}(\theta_{m,0}),\Sigma_{m}(\theta_{m}))\Bigr\|>\frac{\varepsilon}{2}\right)\\
	&\quad +\PP_{\theta_{m,0}}\left(\sup_{|\theta_{m}-\theta_{m,0}|\leq\rho_n}\Bigl\|\partial_{\theta_{m}}^2F(\Sigma_{m}(\theta_{m,0}),\Sigma_{m}(\theta_{m}))-2\Delta_{m}^{\top}W_{m}(\theta_{m,0})^{-1}\Delta_{m}\Bigr\|>\frac{\varepsilon}{2}\right)+\PP_{\theta_{m,0}}\bigl(A_{m,n}^{c}\bigr)\\
	&\leq\PP_{\theta_{m,0}}\left(\sup_{\theta_{m}\in\Theta_{m}}\Bigl\|\partial_{\theta_{m}}^2\tilde{F}(Q_{XX},\Sigma_{m}(\theta_{m}))
	-\partial_{\theta_{m}}^2F(\Sigma_{m}(\theta_{m,0}),\Sigma_{m}(\theta_{m}))\Bigr\|>\frac{\varepsilon}{2}\right)\\
	&\quad +\PP_{\theta_{m,0}}\left(\sup_{|\theta_{m}-\theta_{m,0}|\leq\rho_n}\Bigl\|\partial_{\theta_{m}}^2F(\Sigma_{m}(\theta_{m,0}),\Sigma_{m}(\theta_{m}))-2\Delta_{m}^{\top}W_{m}(\theta_{m,0})^{-1}\Delta_{m}\Bigr\|>\frac{\varepsilon}{2}\right)+\PP_{\theta_{m,0}}\bigl(A_{m,n}^{c}\bigr)\\
	&\longrightarrow 0
\end{align*}
as $n\longrightarrow\infty$, 
which yields
\begin{align}
	\int_{0}^{1}\partial_{\theta_{m}}^2\mathbb{F}_{m,n}(\tilde{\theta}_{m,n})d\lambda\stackrel{P_{\theta_{m,0}}\ }{\longrightarrow}2\Delta_{m}^{\top}W_{m}(\theta_{m,0})^{-1}\Delta_{m}.\label{intprob}
\end{align}
Therefore, from (\ref{thetataylor}), (\ref{partialFprob}), (\ref{intprob}) and Lemma \ref{det},  we obtain
\begin{align*}
	\sqrt{n}(\hat{\theta}_{m,n}-\theta_{m,0})&\stackrel{d}{\longrightarrow}(2\Delta_{m}^{\top}W_{m}(\theta_{m,0})^{-1}\Delta_{m})^{-1}N_{\bar{p}}\bigl(0,4\Delta_{m}^{\top}W_{m}(\theta_{m,0})^{-1}\Delta_{m}\bigr)\\
	&\sim N_{q_m}\bigl(0,(\Delta_{m}^{\top} W_{m}(\theta_{m,0})^{-1}\Delta_{m})^{-1}\bigr).
\end{align*}
\end{proof}
\begin{proof}[Proof of Theorem 3.]
The Taylor expansion of $\mathbb{T}_{m^*,n}=n\mathbb{F}_{m^*,n}(\hat{\theta}_{m^*,n})$ around $\hat{\theta}_{m^*,n}=\theta_{m^*,0}$ is given by
\begin{align}
	\begin{split}
	\mathbb{T}_{m^*,n}
	&=n\mathbb{F}_{m^*,n}(\theta_{m^*,0})+n\partial_{\theta_{m^*}}\mathbb{F}_{m^*,n}(\theta_{m^*,0})^{\top}(\hat{\theta}_{m^*,n}-\theta_{m^*,0})\\
	&\qquad+n(\hat{\theta}_{m^*,n}-\theta_{m^*,0})^{\top}\left\{\int_{0}^{1}(1-\lambda)\partial^2_{\theta_{m^*}}\mathbb{F}_{m^*,n}(\tilde{\theta}_{m^*,n})d\lambda\right\}(\hat{\theta}_{m^*,n}-\theta_{m^*,0}),\label{testtaylor}
	\end{split}
\end{align}
where $\tilde{\theta}_{m^*,n}=\theta_{m^*,0}+\lambda(\hat{\theta}_{m^*,n}-\theta_{m^*,0})$. 
In a similar way to Theorem 2, we obtain
\begin{align}
    \sqrt{n}\partial_{\theta_{m^*}}\mathbb{F}_{m^*,n}(\theta_{m^*,0})=-2\Delta_{m^*}^{\top}\tilde{V}_{m^*,n}(\theta_{m^*,0})\sqrt{n}(\vech Q_{XX}-\vech\Sigma_{m^*}(\theta_{m^*,0}))+o_{p}(1)\label{nFstar}
\end{align}
under $H_0$ and
\begin{align}
	\begin{split}
	&\sqrt{n}(\hat{\theta}_{m^*,n}-\theta_{m^*,0})\\
	&\qquad=(\Delta_{m^*}^{\top}W_{m^*}(\theta_{m^*,0})^{-1}\Delta_{m^*})^{-1}\Delta_{m^*}^{\top}\tilde{V}_{m^*,n}(\theta_{m^*,0})\sqrt{n}(\vech Q_{XX}-\vech\Sigma_{m^*}(\theta_{m^*,0}))+o_{p}(1)\label{thetamstar} 
	\end{split}
\end{align}
under $H_0$. Let
\begin{align*}
	\tilde{H}_{m^*,n}(\theta_{m^*,0})=\tilde{V}_{m^*,n}(\theta_{m^*,0})^{\top}\Delta_{m^*}(\Delta_{m^*}^{\top}W_{m^*}(\theta_{m^*,0})^{-1}\Delta_{m^*})^{-1}\Delta_{m^*}^{\top}\tilde{V}_{m^*,n}(\theta_{m^*,0}).
\end{align*}
Theorem 1, (\ref{Vprob}), (\ref{nFstar}) and (\ref{thetamstar}) imply that 
the second term on the right-hand side of (\ref{testtaylor}) is expressed as
\begin{align}
	\begin{split}
	&\quad\ n\partial_{\theta_{m^*}}\mathbb{F}_{m^*,n}(\theta_{m^*,0})^{\top}(\hat{\theta}_{m^*,n}-\theta_{m^*,0})\\
	&=-2\sqrt{n}(\vech Q_{XX}-\vech\Sigma_{m^*}(\theta_{m^*,0}))^{\top}\tilde{H}_{m^*,n}(\theta_{m^*,0})\sqrt{n}(\vech Q_{XX}-\vech\Sigma_{m^*}(\theta_{m^*,0}))+o_p(1)\label{nFpar}
	\end{split}
\end{align}
under $H_0$. 
Recall that
\begin{align*}
	A_{m^*,n}= \Bigl\{|\hat{\theta}_{m^*,n}-\theta_{m^*,0}|\leq\rho_n\Bigr\},
\end{align*}
where a positive sequence $\{\rho_n\}_{n\in\mathbb{N}}$ satisfies $\rho_n\longrightarrow0$ as $n\longrightarrow\infty$. By an analogous manner to Theorem 2, it follows that for all $\varepsilon>0$,
\begin{align*}
	0&\leq\PP\left(\left\|\int_{0}^{1}(1-\lambda)\partial^2_{\theta_{m^*}}\mathbb{F}_{m^*,n}(\tilde{\theta}_{m^*,n})d\lambda-\Delta_{m^*}^{\top}W_{m^*}(\theta_{m^*,0})^{-1}\Delta_{m^*}\right\|>\varepsilon\right)\\
	&\leq\PP\left(\left\{\left\|\int_{0}^{1}(1-\lambda)\left\{\partial^2_{\theta_{m^*}}\tilde{F}(Q_{XX},\Sigma_{m^*}(\tilde{\theta}_{m^*,n}))-2\Delta_{m^*}^{\top}W_{m^*}(\theta_{m^*,0})^{-1}\Delta_{m^*}\right\}d\lambda\right\|>\varepsilon\right\}\cap A_{m^*,n}\right)\\
	&\quad+\PP\left(\left\{\left\|\int_{0}^{1}(1-\lambda)\left\{\partial^2_{\theta_{m^*}}\tilde{F}(Q_{XX},\Sigma_{m^*}(\tilde{\theta}_{m^*,n}))-2\Delta_{m^*}^{\top}W_{m^*}(\theta_{m^*,0})^{-1}\Delta_{m^*}\right\}d\lambda\right\|>\varepsilon\right\}\cap A_{m^*,n}^{c}\right)\\ 
	&\leq\PP\left(\sup_{|\theta_{m^*}-\theta_{m^*,0}|\leq\rho_{n}}\left\|\partial^2_{\theta_{m^*}}\tilde{F}(Q_{XX},\Sigma_{m^*}(\theta_{m^*}))-2\Delta_{m^*}^{\top}W_{m^*}(\theta_{m^*})^{-1}\Delta_{m^*}\right\|>\varepsilon\right)+\PP\bigl(A_{m^*,n}^{c}\bigr)\\
	&\leq\PP\left(\sup_{\theta_{m^*}\in\Theta_{m^*}}\left\|\partial^2_{\theta_{m^*}}\tilde{F}(Q_{XX},\Sigma_{m^*}(\theta_{m^*}))-\partial^2_{\theta_{m^*}}\tilde{F}(\Sigma_{m^*}(\theta_{m^*,0}),\Sigma_{m^*}(\theta_{m^*}))\right\|>\varepsilon\right)\\
	&+\PP\left(\sup_{|\theta_{m^*}-\theta_{m^*,0}|\leq\rho_{n}}\left\|\partial^2_{\theta_{m^*}}\tilde{F}(\Sigma_{m^*}(\theta_{m^*,0}),\Sigma_{m^*}(\theta_{m^*}))-2\Delta_{m^*}^{\top}W_{m^*}(\theta_{m^*})^{-1}\Delta_{m^*}\right\|>\varepsilon\right)+\PP\bigl(A_{m^*,n}^{c}\bigr)\\
	&\longrightarrow 0
\end{align*} 
as $n\longrightarrow\infty$ under $H_0$, so that
\begin{align}
	\int_{0}^{1}(1-\lambda)\partial_{\theta_{m^*}}^2\mathbb{F}_{m^*,n}(\tilde{\theta}_{m^*,n})d\lambda\stackrel{P}{\longrightarrow}\Delta_{m^*}^{\top}W_{m^*}(\theta_{m^*,0})^{-1}\Delta_{m^*}\label{intprob2}
\end{align}
under $H_{0}$. 
Thus, Theorem 1, (\ref{Vprob}) and (\ref{thetamstar}) imply that 
the third term on the right-hand side of (\ref{testtaylor}) is
\begin{align}
	\begin{split}
	&\quad\ n(\hat{\theta}_{m^*,n}-\theta_{m^*,0})^\top\left\{\int_{0}^{1}(1-\lambda)\partial^2_{\theta_{m^*}}\tilde{F}(Q_{XX},\Sigma_{m^*}(\tilde{\theta}_{m^*,n}))d\lambda\right\}(\hat{\theta}_{m^*,n}-\theta_{m^*,0})\\
	&=\sqrt{n}(\vech Q_{XX}-\vech\Sigma_{m^*}(\theta_{m^*,0}))^{\top}\tilde{H}_{m^*,n}(\theta_{m^*,0})\sqrt{n}(\vech Q_{XX}-\vech\Sigma_{m^*}(\theta_{m^*,0}))+o_{p}(1)\label{nFpar2}
	\end{split}
\end{align}
under $H_0$.
Therefore, it follows from (\ref{nFpar}) and (\ref{nFpar2}) that (\ref{testtaylor}) is given by
\begin{align}
	\begin{split}
	\mathbb{T}_{m^*,n}
	&=\sqrt{n}(\vech Q_{XX}-\vech\Sigma_{m^*}(\theta_{m^*,0}))^{\top}\bigl\{\tilde{V}_{m^*,n}(\theta_{m^*,0})-\tilde{H}_{m^*,n}(\theta_{m^*,0})\bigr\}\\
	&\qquad\qquad\qquad\qquad\qquad\qquad\qquad\qquad\times\sqrt{n}(\vech Q_{XX}-\vech\Sigma_{m^*}(\theta_{m^*,0}))+o_{p}(1)\label{testre1}
	\end{split}
\end{align}
under $H_0$.
Set
\begin{align*}
	\gamma_{m^*,n}=\tilde{V}_{m^*,n}(\theta_{m^*,0})
	^{\frac{1}{2}}\sqrt{n}(\vech{Q_{XX}}-\vech{\Sigma_{m^*}(\theta_{m^*,0})})
\end{align*}
and
\begin{align*}
    \tilde{P}_{m^*,n}=\tilde{V}_{m^*,n}(\theta_{m^*,0})^{-\frac{1}{2}}\bigl\{\tilde{V}_{m^*,n}(\theta_{m^*,0})-\tilde{H}_{m^*,n}\bigr\}\tilde{V}_{m^*,n}(\theta_{m^*,0})^{-\frac{1}{2}}.
\end{align*}
We can rewrite (\ref{testre1}) as
\begin{align}
	\mathbb{T}_{m^*,n}=\gamma_{m^*,n}^{\top}\tilde{P}_{m^*,n}(\theta_{m^*,0})\gamma_{m^*,n}.\label{Tgamma}
\end{align}
It follows from (\ref{Vprob}) and the continuous mapping theorem  that
under $H_0$,
\begin{align*}
	\begin{split}
	\tilde{V}_{m^*,n}(\theta_{m^*,0})^{\frac{1}{2}}&=f(\tilde{V}_{m^*,n}(\theta_{m^*,0}))\stackrel{P}{\longrightarrow}f(W_{m^*}(\theta_{m^*,0})^{-1})=W_{m^*}(\theta_{m^*,0})^{-\frac{1}{2}},
	\end{split}
\end{align*}
where $f(X)=X^{\frac{1}{2}}$ for $X\in\mathbb{R}^{\bar{p}\times\bar{p}}$. Theorem 1 and Slutsky's theorem show that
under $H_0$, 
\begin{align}
	\gamma_{m^*,n}\stackrel{d}{\longrightarrow}\gamma, \label{gammad}
\end{align}
where $\gamma\sim N_{\bar{p}}(0,\mathbb{I}_{\bar{p}})$.
Set
\begin{align*}
	H_{m^*}(\theta_{m,0})=W_{m^*}(\theta_{m^*,0})^{-1\top}\Delta_{m^*}(\Delta_{m^*}^{\top}W_{m^*}(\theta_{m^*,0})^{-1}\Delta_{m^*})^{-1}\Delta_{m^*}^{\top}W_{m^*}(\theta_{m^*,0})^{-1}.
\end{align*}
It follows from (\ref{Vprob}) and the continuous mapping theorem that under $H_0$,  
\begin{align}
	\tilde{H}_{m^*,n}(\theta_{m^*,0})=f(\tilde{V}_{m^*,n}(\theta_{m^*,0}))\stackrel{P}{\longrightarrow}f(W_{m^*}(\theta_{m^*,0})^{-1})=H_{m^*}(\theta_{m^*,0}), \label{Hprob}
\end{align}
where
\begin{align*}
	f(X)=X^{\top}\Delta_{m^*}(\Delta_{m^*}^{\top}W_{m^*}(\theta_{m^*,0})^{-1}\Delta_{m^*})^{-1}\Delta_{m^*}^{\top}X
\end{align*}
for $X\in\mathbb{R}^{\bar{p}\times\bar{p}}$.
Since (\ref{Vprob}) and the continuous mapping theorem imply that under $H_0$,
\begin{align*}
	\begin{split}
	\tilde{V}_{m^*,n}(\theta_{m^*,0})^{-\frac{1}{2}}=f(\tilde{V}_{m^*,n}(\theta_{m^*,0}))\stackrel{P}{\longrightarrow}f(W_{m^*}(\theta_{m^*,0})^{-1})=W_{m^*}(\theta_{m^*,0})^{-\frac{1}{2}},
	\end{split}
\end{align*}
where $f(X)=X^{-\frac{1}{2}}$ for $X\in\mathbb{R}^{\bar{p}\times\bar{p}}$, 
we obtain from (\ref{Vprob}), (\ref{Hprob}) and Slutsky's theorem that under $H_0$,
\begin{align}
	\tilde{P}_{m^*,n}(\theta_{m^*,0})\stackrel{P}{\longrightarrow}P_{m^*}(\theta_{m^*,0}), \label{Pprob}
\end{align}
where
\begin{align*}
	P_{m^*}(\theta_{m^*,0})=W_{m^*}(\theta_{m^*,0})^{\frac{1}{2}}\bigl\{W_{m^*}(\theta_{m^*,0})^{-1}-H_{m^*}(\theta_{m^*,0})\bigr\}W_{m^*}(\theta_{m^*,0})^{\frac{1}{2}}.
\end{align*}
Furthermore, it follows  from the continuous mapping theorem and (\ref{gammad}) that under $H_0$,
\begin{align}
	\gamma_{m^*,n}^{\top}P_{m^*}(\theta_{m^*,0})\gamma_{m^*,n}=f(\gamma_{m^*,n})\stackrel{d}{\longrightarrow}f(\gamma)=\gamma^{\top}P_{m^*}(\theta_{m^*,0})\gamma,   \label{gammaPd}
\end{align}
where 
\begin{align*}
    f(x)=x^{\top}P_{m^*}(\theta_{m^*,0})x
\end{align*}
for $x\in\mathbb{R}^{\bar{p}}$. 
We see from (\ref{gammad}) that $\gamma_{m^*,n}=O_{p}(1)$ under $H_0$, and it holds from (\ref{Pprob}) that
\begin{align}
	\gamma_{m^*,n}^{\top}\tilde{P}_{m^*,n}(\theta_{m^*,0})\gamma_{m^*,n}-\gamma_{m^*,n}^{\top}P_{m^*}(\theta_{m^*,0})\gamma_{m^*,n}\stackrel{P}{\longrightarrow} 0 \label{Pslu}
\end{align}
under $H_0$. Therefore, (\ref{Tgamma}), (\ref{gammaPd}), (\ref{Pslu}) and Slutsky's theorem yield 
\begin{align}
	\mathbb{T}_{m^*,n}\stackrel{d}{\longrightarrow}\gamma^{\top}P_{m^*}(\theta_{m^*,0})\gamma\label{Tprob}
\end{align}
under $H_0$.
Since one gets
\begin{align*}
    \gamma^{\top}P_{m^*}(\theta_{m^*,0})\gamma\sim\chi^2_{\bar{p}-q_{m^*}}
\end{align*}
in the same manner as Theorem 3 in Kusano and Uchida \cite{Kusano(2022)}, 
we obtain from (\ref{Tprob}) that
\begin{align*}
	\mathbb{T}_{m^*,n}\stackrel{d}{\longrightarrow}\chi^2_{\bar{p}-q_{m^*}}
\end{align*}
under $H_0$. 
\end{proof}
\begin{lemma}\label{starconslemma}
Under 
$\bf{[A1]}$-$\bf{[A2]}$, $\bf{[B1]}$-$\bf{[B3]}$, $\bf{[C1]}$-$\bf{[C3]}$, $\bf{[D1]}$-$\bf{[D2]}$, 
$\bf{[E]}$, $\bf{[F]}$ and $\bf{[I]}$,
as $h_n\longrightarrow0$ and $nh_n\longrightarrow\infty$,
\begin{align*}
	\hat{\theta}_{m^*,n}\stackrel{P}{\longrightarrow}\bar{\theta}_{m^*}
\end{align*}
under $H_1$.
\end{lemma}
\begin{proof}
See Appendix \ref{Proofstarcons}.
\end{proof}
\begin{lemma}\label{kitagawa}
Let $X_n\stackrel{P}{\longrightarrow} c>0$. For any $\epsilon>0$, 
\begin{align*}
	\PP\left(X_n\leq\frac{\epsilon}{n}\right)\stackrel{}{\longrightarrow} 0
\end{align*}
as $n\longrightarrow\infty$.
\end{lemma}
\begin{proof}
See Lemma 3 in Kitagawa and Uchida \cite{kitagawa(2014)}.
\end{proof}
\begin{proof}[Proof of Theorem 4.]
Since $\mathbb{U}_{m^*}(\theta_{m^*})$ is continuous in $\theta_{m^*}$, it holds 
from the continuous mapping theorem and Lemma \ref{starconslemma} that
\begin{align}
	\mathbb{U}_{m^*}(\hat{\theta}_{m^*,n})\stackrel{P}{\longrightarrow}\mathbb{U}_{m^*}(\bar{\theta}_{m^*})\label{Fu1prob}
\end{align}
under $H_1$.

It follows  from (\ref{Fprob}) and (\ref{Fu1prob}) that 
for all $\varepsilon>0$,
\begin{align*}
	0&\leq \PP\left(\Bigl|\frac{1}{n}\mathbb{T}_{m^*,n}-\mathbb{U}_{m^*}(\bar{\theta}_{m^*})\Bigr|>\epsilon\right)\\
	&\leq\PP\left(\bigl|\tilde{F}(Q_{XX},\Sigma_{m^*}(\hat{\theta}_{m^*,n}))-\mathbb{U}_{m^*}(\hat{\theta}_{m^*,n})
	\bigr|>\frac{\epsilon}{2}\right)+\PP\left(\bigl|\mathbb{U}_{m^*}(\hat{\theta}_{m^*,n})-\mathbb{U}_{m^*}(\bar{\theta}_{m^*})\bigr|>\frac{\epsilon}{2}\right)\\
	&\leq\PP\left(\sup_{\theta_{m^*}\in\Theta_{m^*}}\bigl|\tilde{F}(Q_{XX},\Sigma_{m^*}(\theta_{m^*}))-F(\Sigma_{m}(\theta_{m,0}),\Sigma_{m^*}(\theta_{m^*}))\bigr|>\frac{\epsilon}{2}\right)\\
	&\qquad+\PP\left(\bigl|\mathbb{U}_{m^*}(\hat{\theta}_{m^*,n})-\mathbb{U}_{m^*}(\bar{\theta}_{m^*})\bigr|>\frac{\epsilon}{2}\right)\stackrel{}{\longrightarrow}0
\end{align*}
under $H_1$ as $n\longrightarrow\infty$, which implies that
\begin{align}
	\frac{1}{n}\mathbb{T}_{m^*,n}\stackrel{P}{\longrightarrow} \mathbb{U}_{m^*}(\bar{\theta}_{m^*})\label{Tcons}
\end{align}
under $H_1$. Note that
\begin{align*}
    \vech{\Sigma_{m}(\theta_{m,0})}-\vech{\Sigma_{m^*}(\bar{\theta}_{m^*})}\neq 0
\end{align*}
under $H_1$. It follows  from 
Lemmas \ref{Sigmaposlemma}-\ref{Vposlemma} 
that $\mathbb{U}_{m^*}(\bar{\theta}_{m^*})>0$ under $H_1$.
Therefore, Lemma \ref{kitagawa} and (\ref{Tcons}) imply that under $H_1$
\begin{align*}
	\PP\Bigl(\mathbb{T}_{m^*,n}>\chi^2_{\bar{p}-q_{m^*}}(\alpha)\Bigr)
	&=1-\PP\left(\frac{1}{n}\mathbb{T}_{m^*,n}\leq \frac{1}{n}\chi^2_{\bar{p}-q_{m^*}}(\alpha)\right)\stackrel{}{\longrightarrow}1
\end{align*}
as $n\longrightarrow\infty$.
\end{proof}
\begin{proof}[Proof of Theorem 5.]
See, for example, Lemma 9 in Genon-Catalot and Jacod  \cite{Genon(1993)} for consistency and Theorem 3.2 in Jacod  \cite{Jacod(1997)} for asymptotic normality.
\end{proof}
\begin{proof}[Proofs of Theorems 6-8]
Since $T$ is fix and $nh_n^2=h_nT\longrightarrow 0$, the proofs of Theorems 6-8 are the same as those of Theorems 2-4, respectively.    
\end{proof}

\newpage
\section{Appendix}
\subsection{Details of simulation results}\label{simulation}
\ \\
\begin{figure}[h]
    \ \\
	\includegraphics[width=0.32\columnwidth]{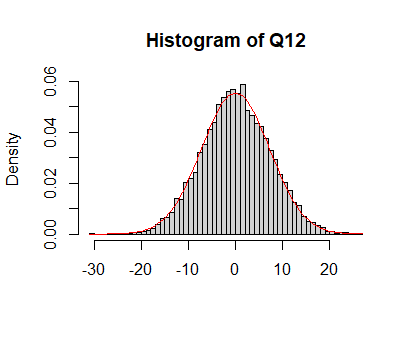}
	\includegraphics[width=0.32\columnwidth]{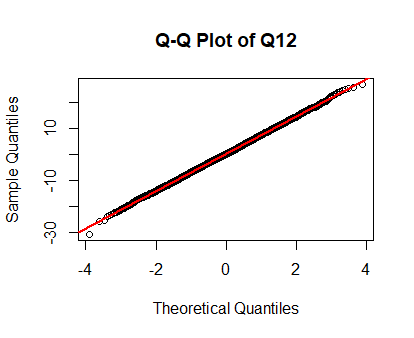}
	\includegraphics[width=0.32\columnwidth]{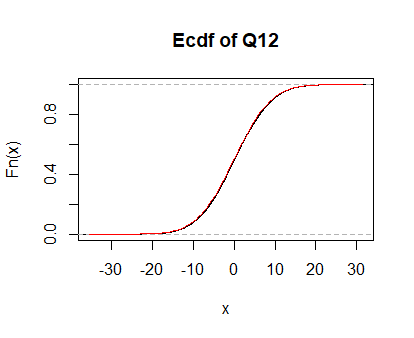}\\ \ \\
	\includegraphics[width=0.32\columnwidth]{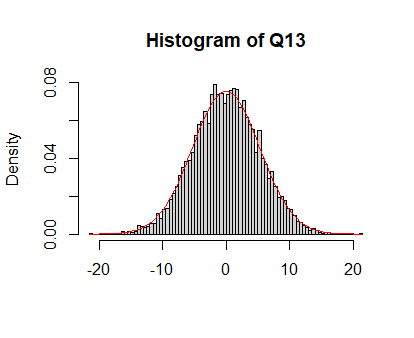}
	\includegraphics[width=0.32\columnwidth]{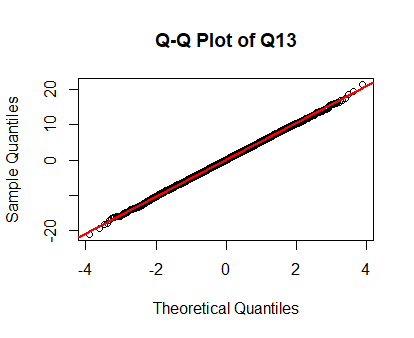}
	\includegraphics[width=0.32\columnwidth]{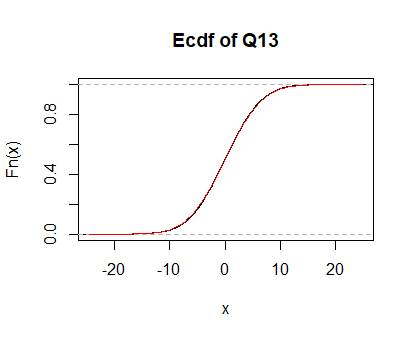}\\	\ \\	
	\includegraphics[width=0.32\columnwidth]{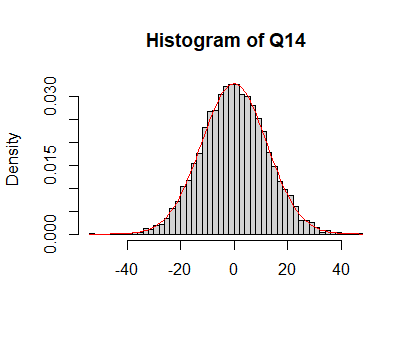}
	\includegraphics[width=0.32\columnwidth]{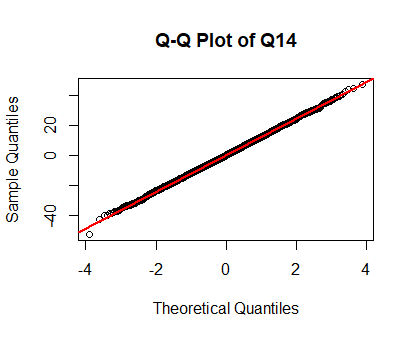}
	\includegraphics[width=0.32\columnwidth]{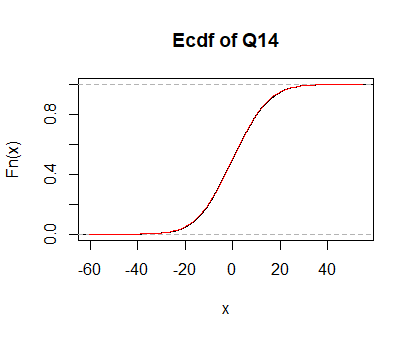}\\ \ \\
	\includegraphics[width=0.32\columnwidth]{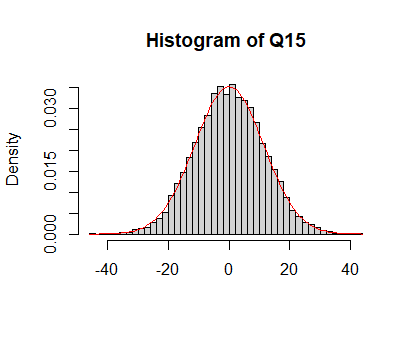}
	\includegraphics[width=0.32\columnwidth]{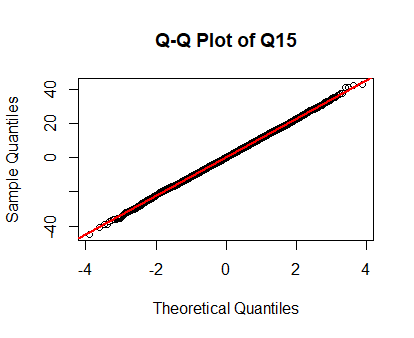}
	\includegraphics[width=0.32\columnwidth]{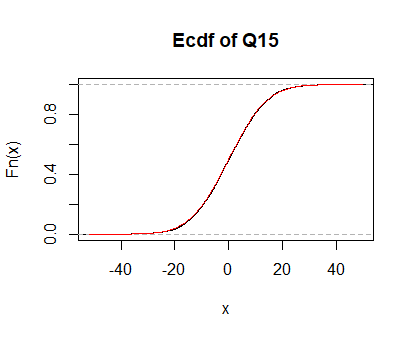}\\
\end{figure}
\newpage
\begin{figure}
    \ \\
	\includegraphics[width=0.32\columnwidth]{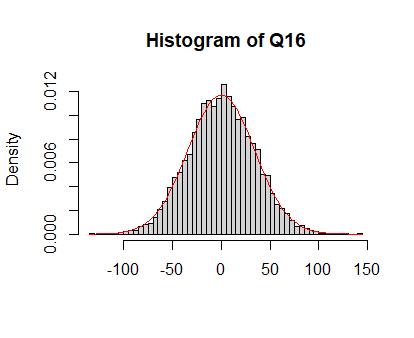}
	\includegraphics[width=0.32\columnwidth]{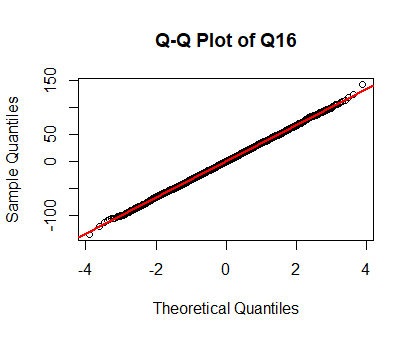}
	\includegraphics[width=0.32\columnwidth]{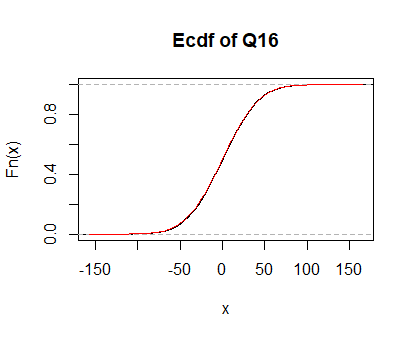}\\
	\includegraphics[width=0.32\columnwidth]{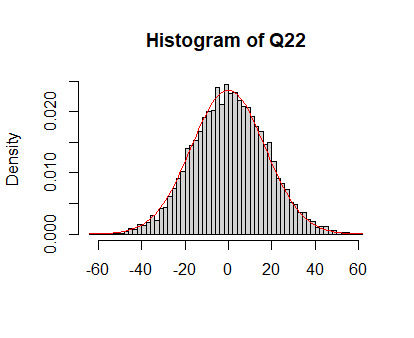}
	\includegraphics[width=0.32\columnwidth]{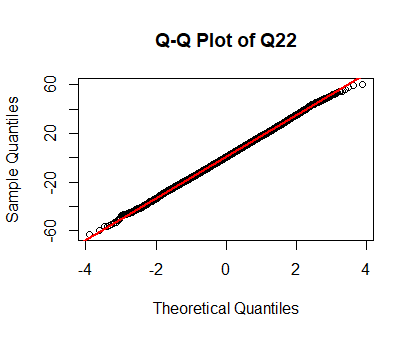}
	\includegraphics[width=0.32\columnwidth]{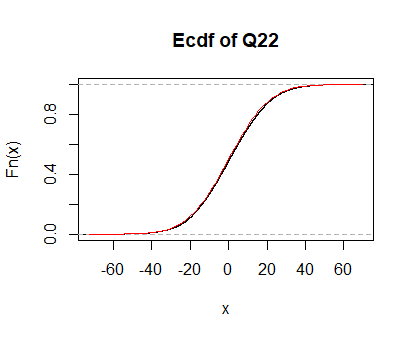}\\		
	\includegraphics[width=0.32\columnwidth]{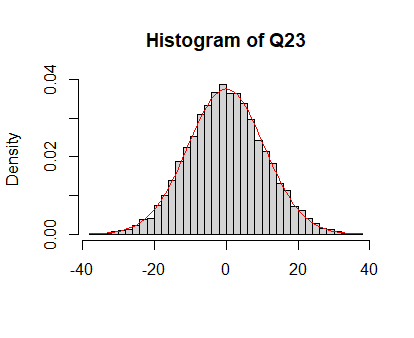}
	\includegraphics[width=0.32\columnwidth]{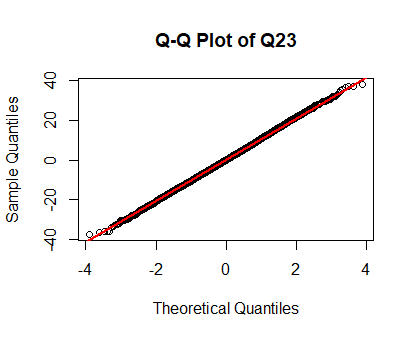}
	\includegraphics[width=0.32\columnwidth]{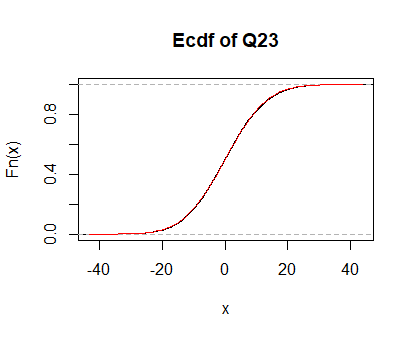}\\
	\includegraphics[width=0.32\columnwidth]{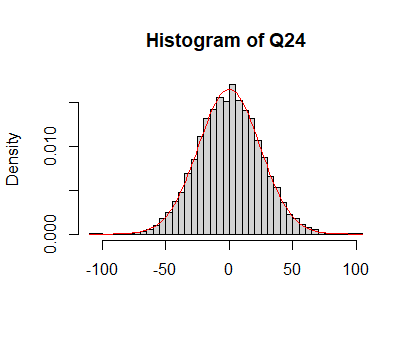}
	\includegraphics[width=0.32\columnwidth]{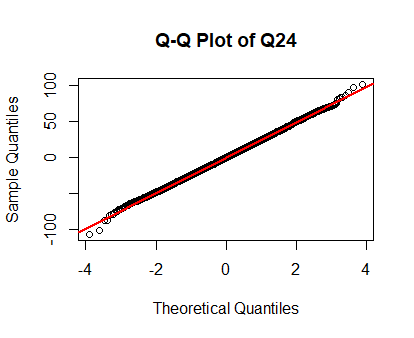}
	\includegraphics[width=0.32\columnwidth]{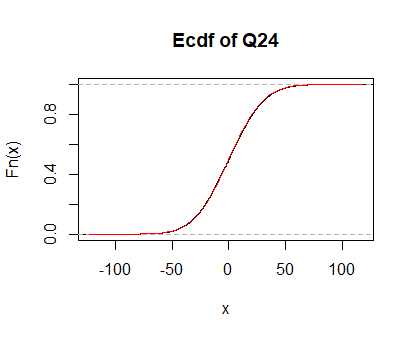}\\
	\includegraphics[width=0.32\columnwidth]{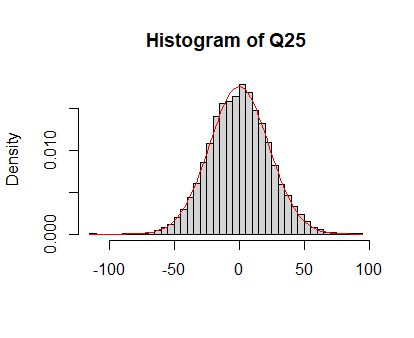}
	\includegraphics[width=0.32\columnwidth]{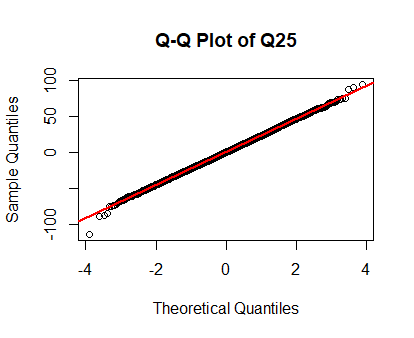}
	\includegraphics[width=0.32\columnwidth]{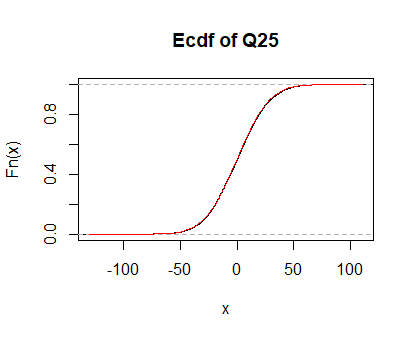}	
\end{figure}
\begin{figure}
    \ \\
	\includegraphics[width=0.32\columnwidth]{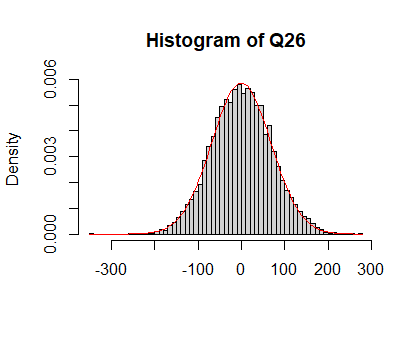}
	\includegraphics[width=0.32\columnwidth]{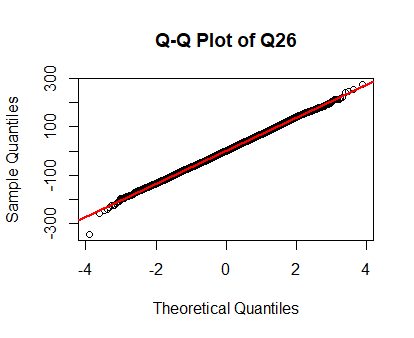}
	\includegraphics[width=0.32\columnwidth]{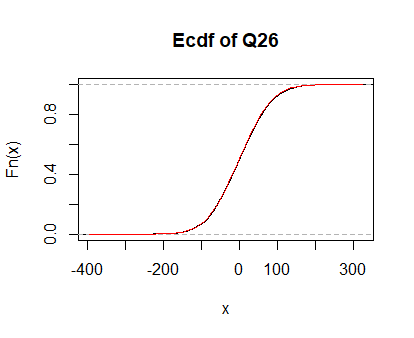}\\	
	\includegraphics[width=0.32\columnwidth]{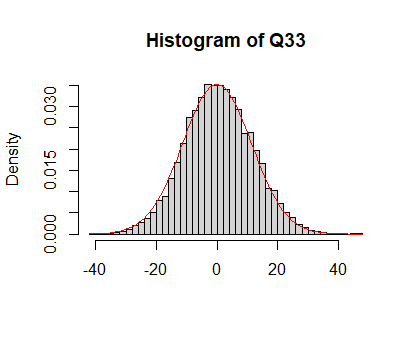}
	\includegraphics[width=0.32\columnwidth]{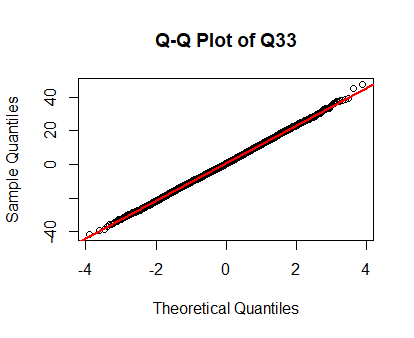}
	\includegraphics[width=0.32\columnwidth]{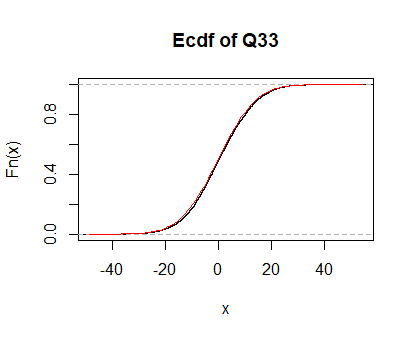}\\
	\includegraphics[width=0.32\columnwidth]{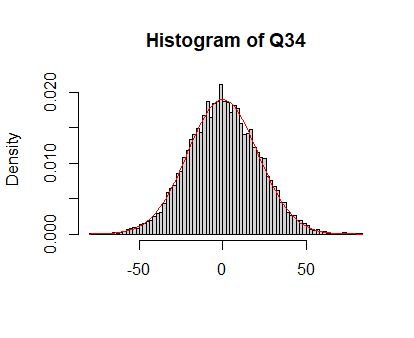}
	\includegraphics[width=0.32\columnwidth]{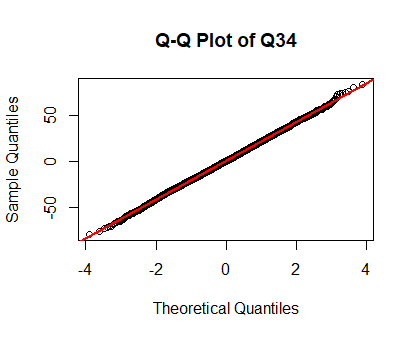}
	\includegraphics[width=0.32\columnwidth]{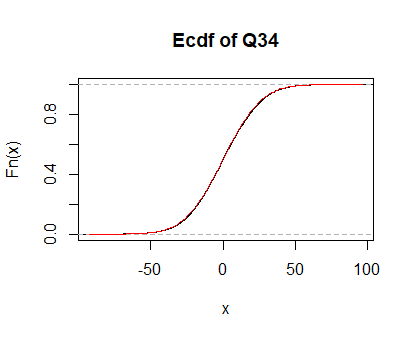}\\
	\includegraphics[width=0.32\columnwidth]{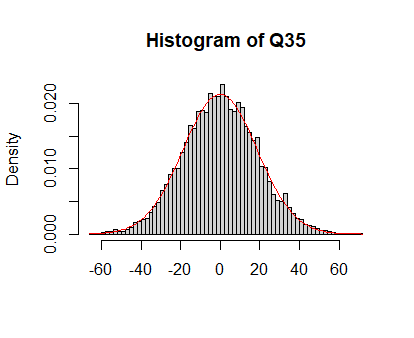}
	\includegraphics[width=0.32\columnwidth]{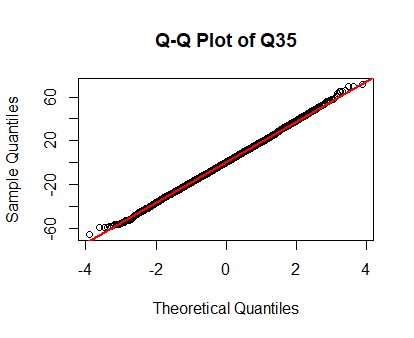}
	\includegraphics[width=0.32\columnwidth]{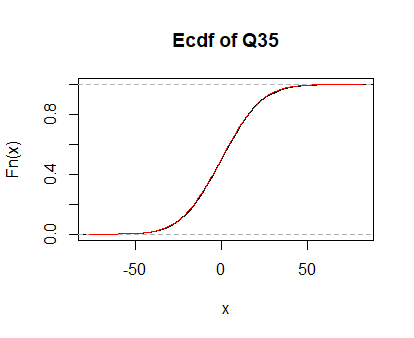}\\
	\includegraphics[width=0.32\columnwidth]{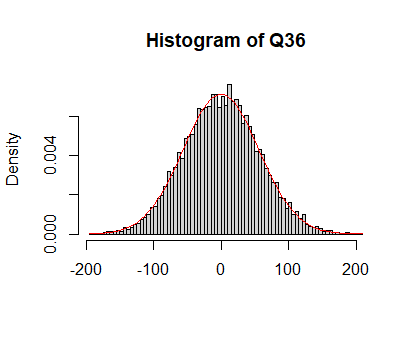}
	\includegraphics[width=0.32\columnwidth]{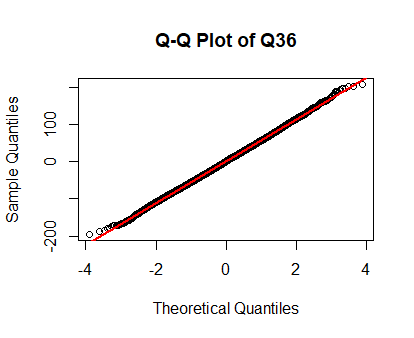}
	\includegraphics[width=0.32\columnwidth]{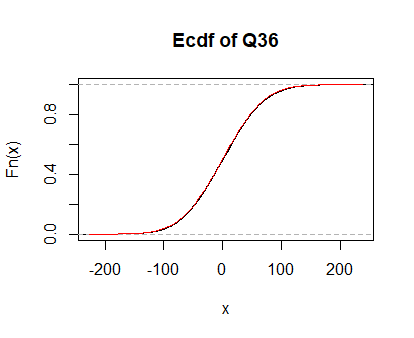}
\end{figure}
\begin{figure}
    \ \\
    \includegraphics[width=0.32\columnwidth]{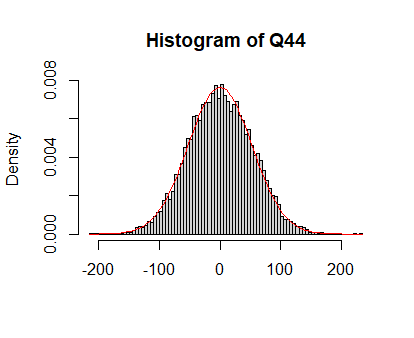}
	\includegraphics[width=0.32\columnwidth]{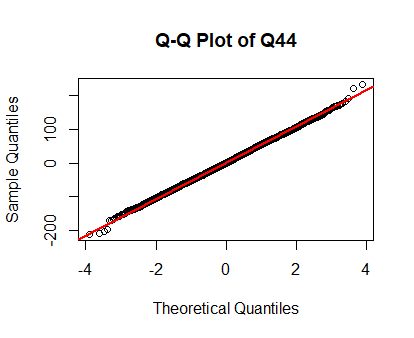}
	\includegraphics[width=0.32\columnwidth]{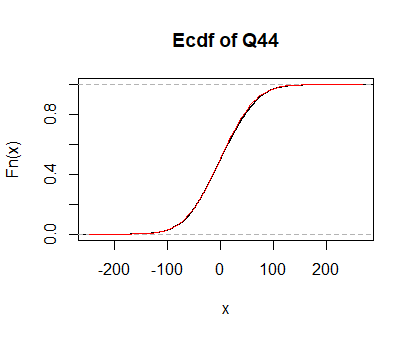}\\
	\includegraphics[width=0.32\columnwidth]{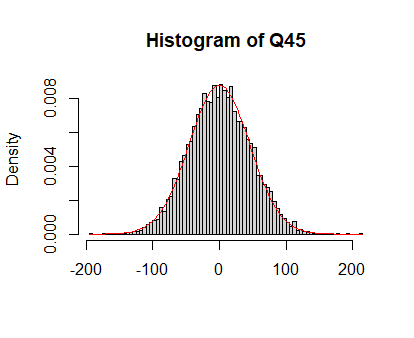}
	\includegraphics[width=0.32\columnwidth]{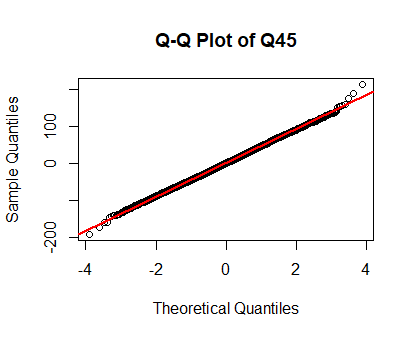}
	\includegraphics[width=0.32\columnwidth]{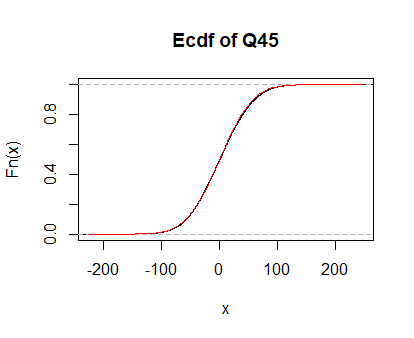}\\
	\includegraphics[width=0.32\columnwidth]{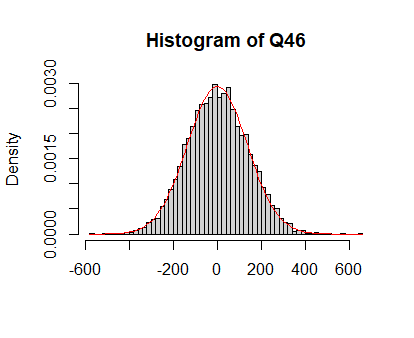}
	\includegraphics[width=0.32\columnwidth]{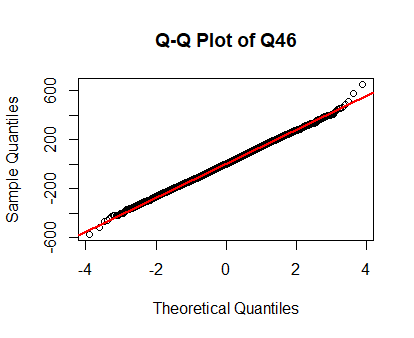}
	\includegraphics[width=0.32\columnwidth]{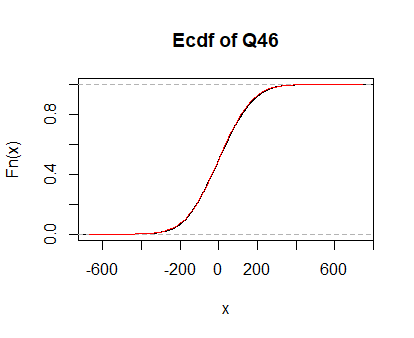}\\
	\includegraphics[width=0.32\columnwidth]{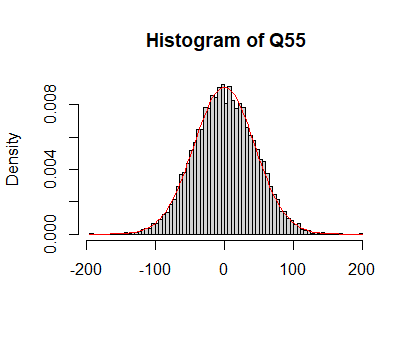}
	\includegraphics[width=0.32\columnwidth]{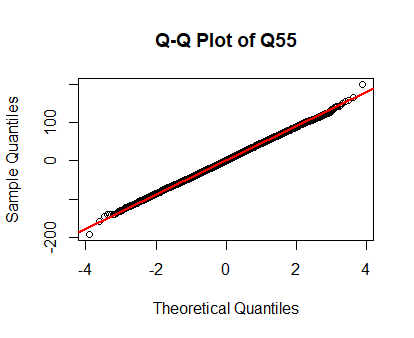}
	\includegraphics[width=0.32\columnwidth]{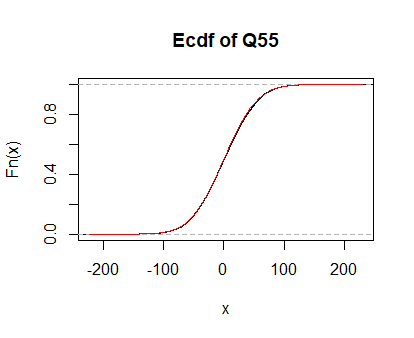}\\
	\includegraphics[width=0.32\columnwidth]{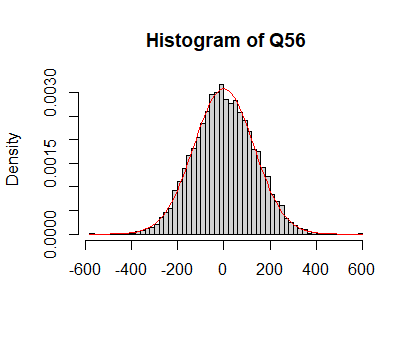}
	\includegraphics[width=0.32\columnwidth]{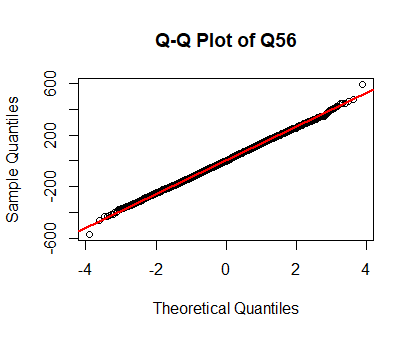}
	\includegraphics[width=0.32\columnwidth]{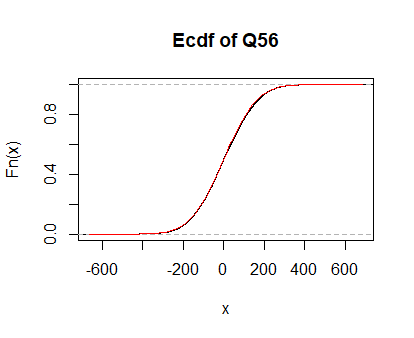}
\end{figure}
\begin{figure}
    \ \\
	\includegraphics[width=0.32\columnwidth]{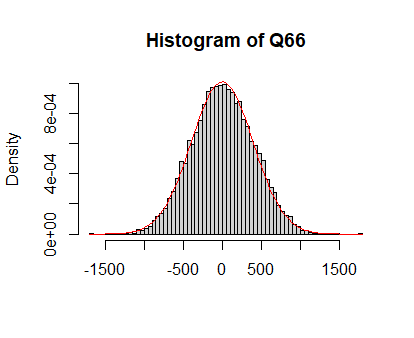}
	\includegraphics[width=0.32\columnwidth]{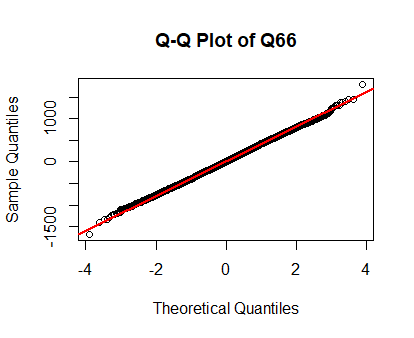}
	\includegraphics[width=0.32\columnwidth]{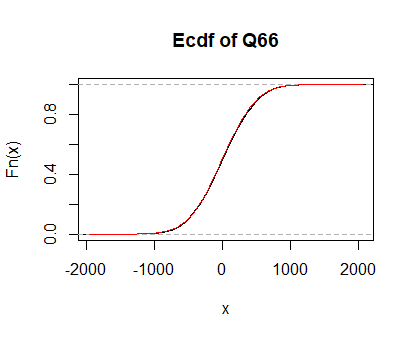}\\
	\caption{Histogram (left), Q-Q plot (middle) and empirical distribution (right) of $\sqrt{n}((Q_{XX})_{ij}$ - $(\Sigma_{m}(\theta_{m}))_{ij})\quad (i\leq j,\  i,j=1,\cdots,6)$.}
	\label{Qzzfigure2}
\end{figure}
\begin{figure}
    \ \\
    \ \\
	\includegraphics[width=0.32\columnwidth]{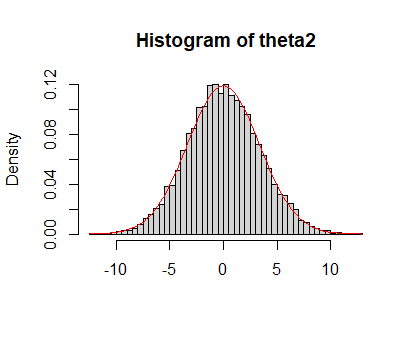}
	\includegraphics[width=0.32\columnwidth]{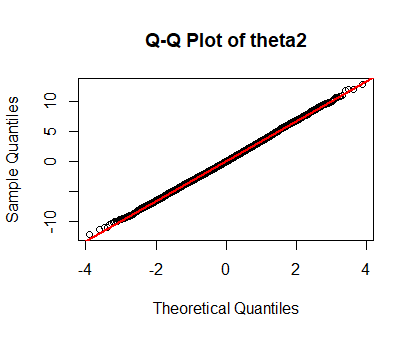}
	\includegraphics[width=0.32\columnwidth]{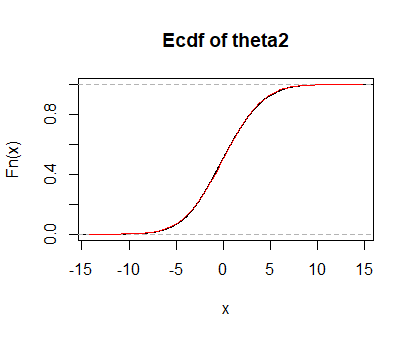}\\ \ \\
	\includegraphics[width=0.32\columnwidth]{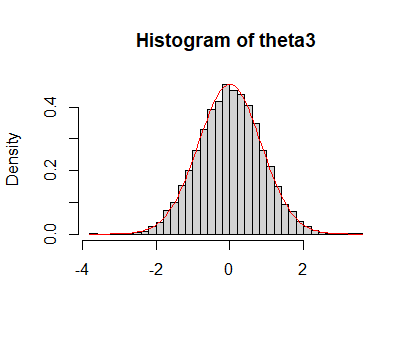}
	\includegraphics[width=0.32\columnwidth]{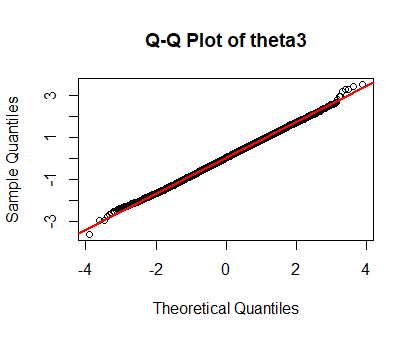}
	\includegraphics[width=0.32\columnwidth]{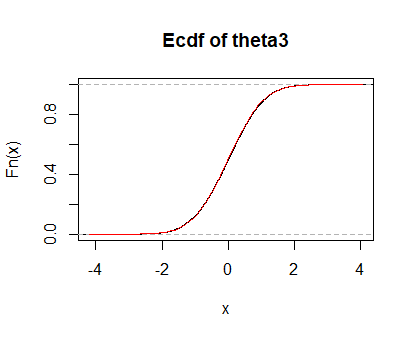}\\ \ \\
	\includegraphics[width=0.32\columnwidth]{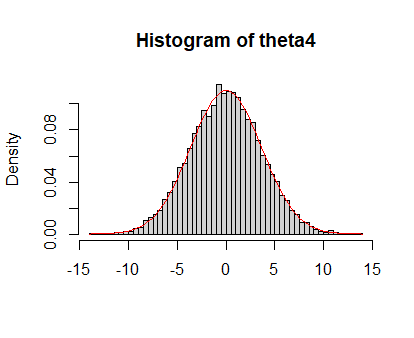}
	\includegraphics[width=0.32\columnwidth]{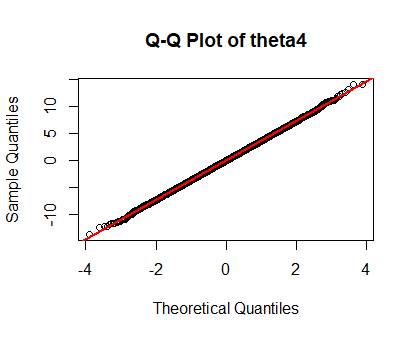}
	\includegraphics[width=0.32\columnwidth]{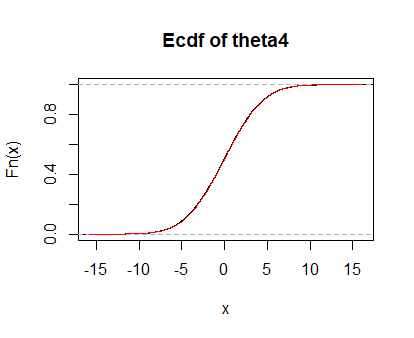}
\end{figure}
\begin{figure}
    \ \\
	\includegraphics[width=0.32\columnwidth]{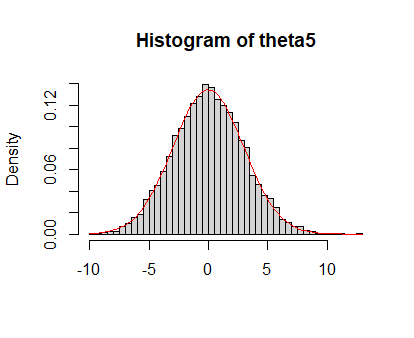}
	\includegraphics[width=0.32\columnwidth]{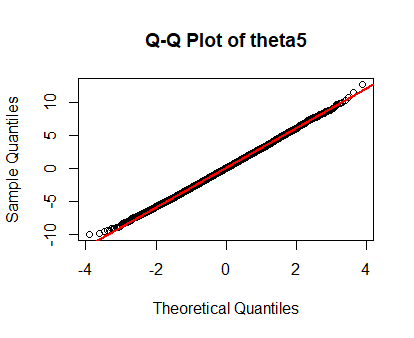}
	\includegraphics[width=0.32\columnwidth]{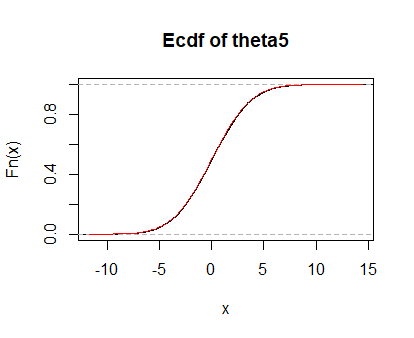}\\ 
	\includegraphics[width=0.32\columnwidth]{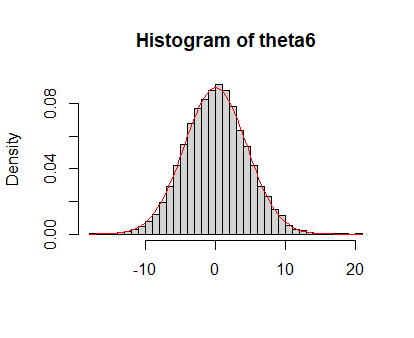}
	\includegraphics[width=0.32\columnwidth]{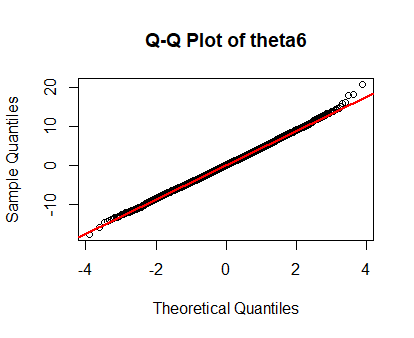}
	\includegraphics[width=0.32\columnwidth]{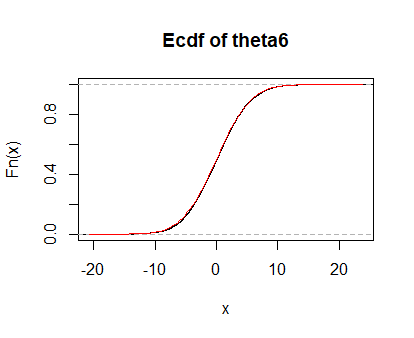}\\
	\includegraphics[width=0.32\columnwidth]{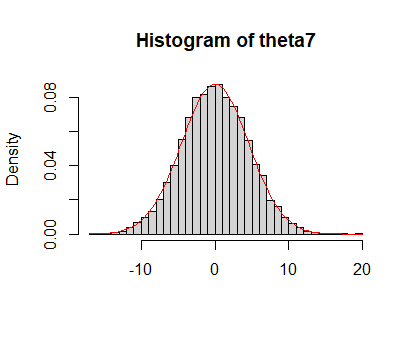}
	\includegraphics[width=0.32\columnwidth]{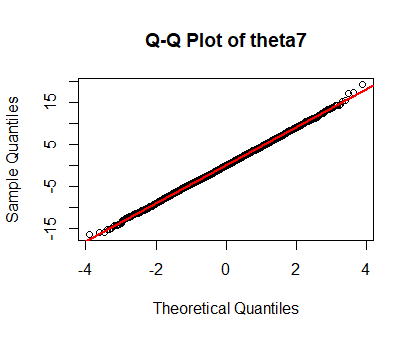}
	\includegraphics[width=0.32\columnwidth]{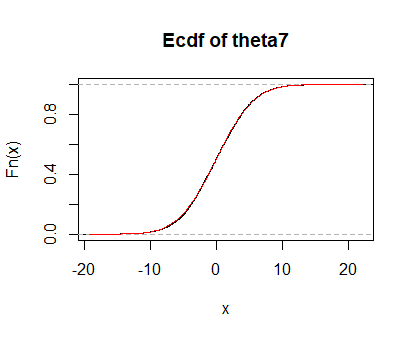}\\   		
	\includegraphics[width=0.32\columnwidth]{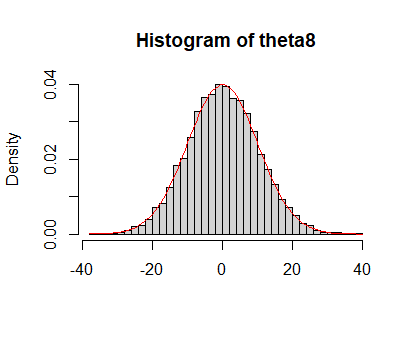}
	\includegraphics[width=0.32\columnwidth]{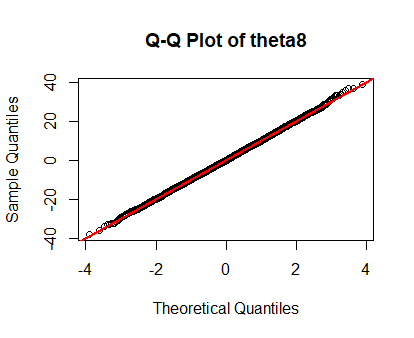}
	\includegraphics[width=0.32\columnwidth]{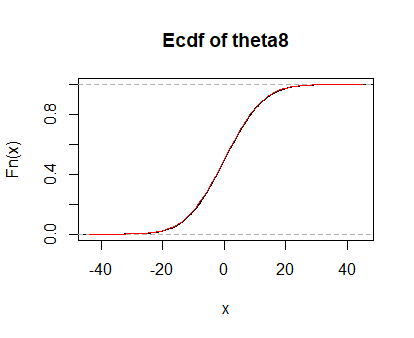}\\
	\includegraphics[width=0.32\columnwidth]{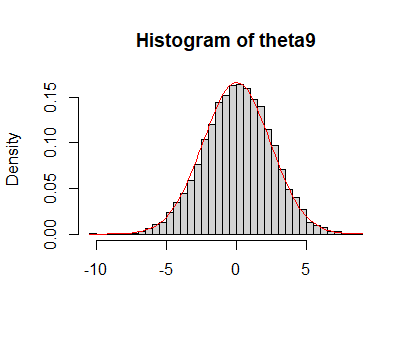}
	\includegraphics[width=0.32\columnwidth]{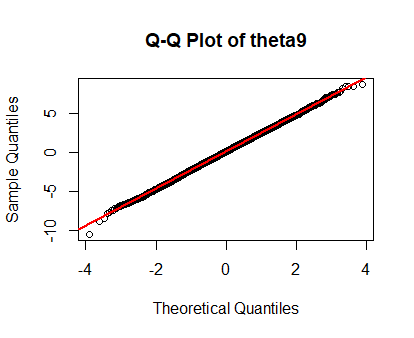}
	\includegraphics[width=0.32\columnwidth]{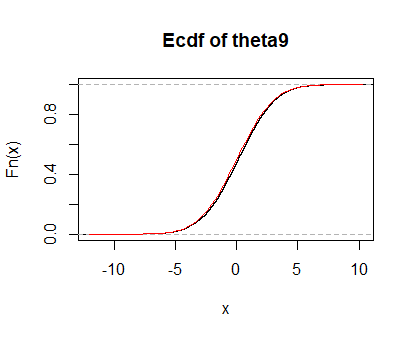}
\end{figure}
\begin{figure}
    \ \\
	\includegraphics[width=0.32\columnwidth]{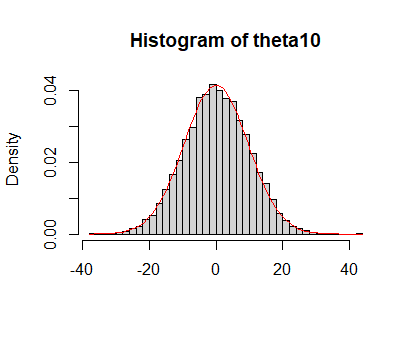}
	\includegraphics[width=0.32\columnwidth]{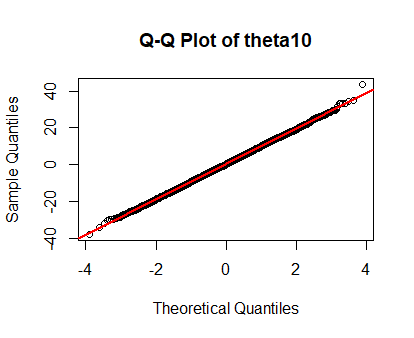}
	\includegraphics[width=0.32\columnwidth]{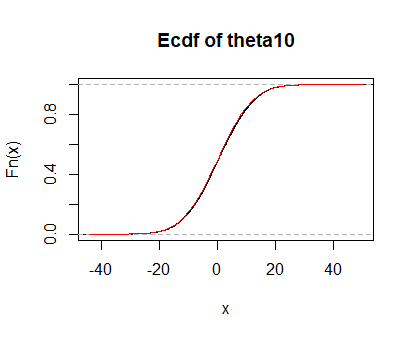}\\
	\includegraphics[width=0.32\columnwidth]{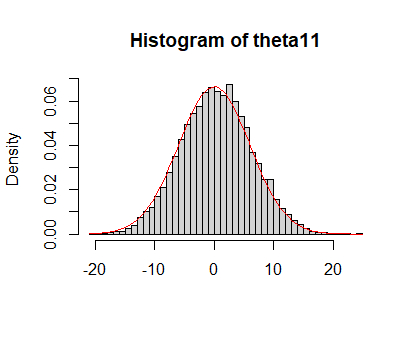}
	\includegraphics[width=0.32\columnwidth]{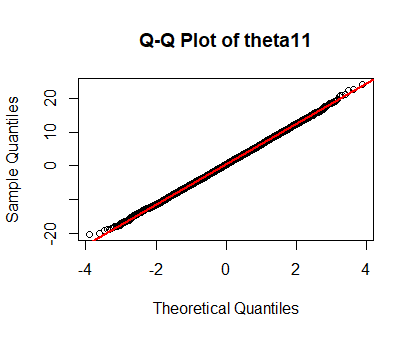}
	\includegraphics[width=0.32\columnwidth]{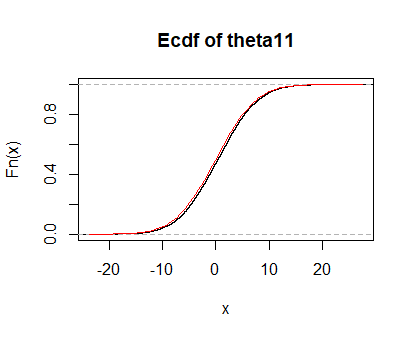}\\
	\includegraphics[width=0.32\columnwidth]{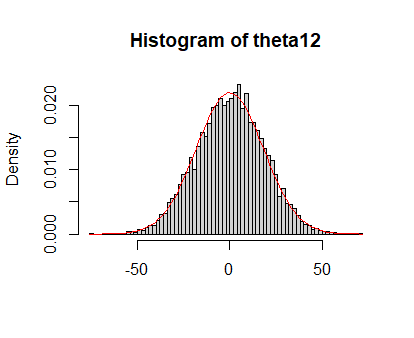}
	\includegraphics[width=0.32\columnwidth]{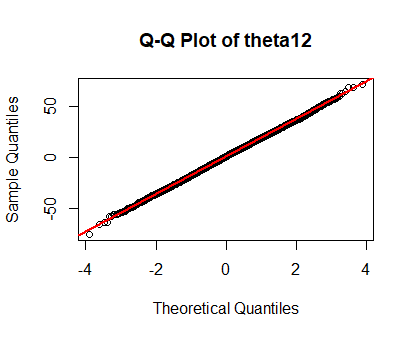}
	\includegraphics[width=0.32\columnwidth]{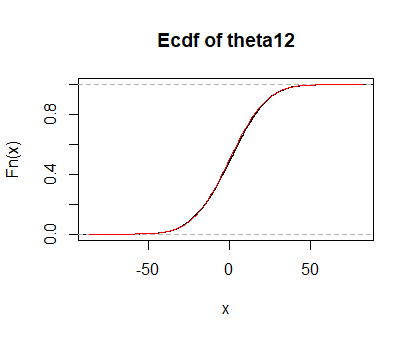}\\
	\includegraphics[width=0.32\columnwidth]{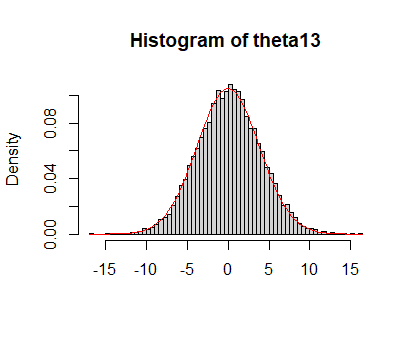}
	\includegraphics[width=0.32\columnwidth]{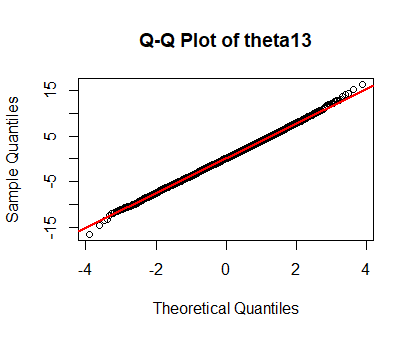}
	\includegraphics[width=0.32\columnwidth]{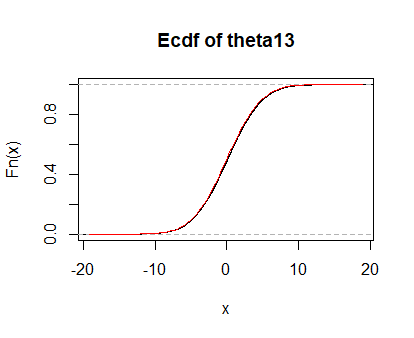}\\ 
	\includegraphics[width=0.32\columnwidth]{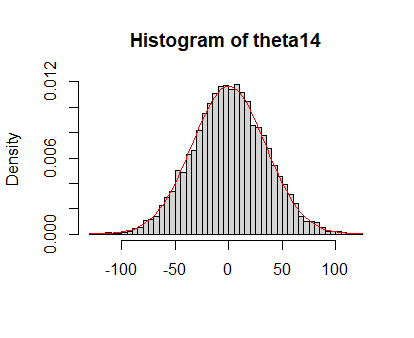}
	\includegraphics[width=0.32\columnwidth]{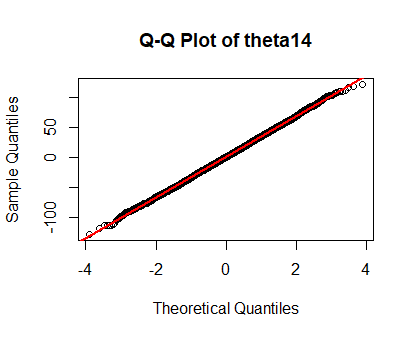}
	\includegraphics[width=0.32\columnwidth]{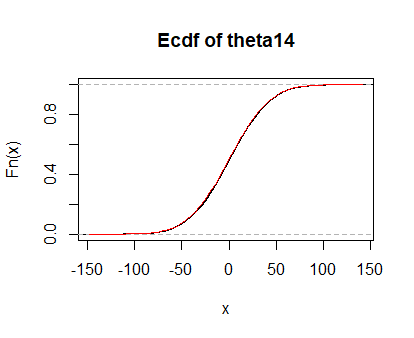}
\end{figure}
\begin{figure}
	\includegraphics[width=0.32\columnwidth]{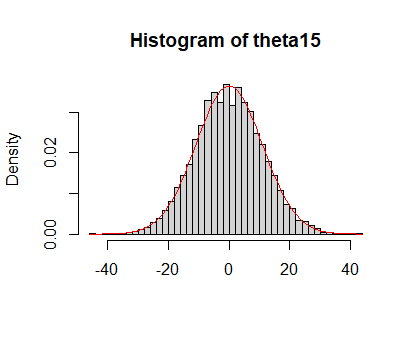}
	\includegraphics[width=0.32\columnwidth]{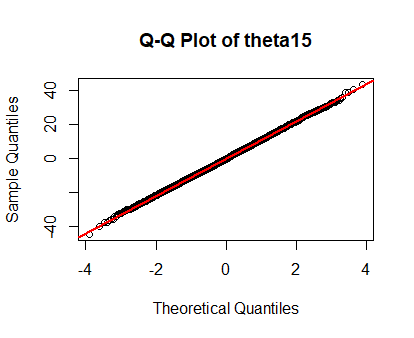}
	\includegraphics[width=0.32\columnwidth]{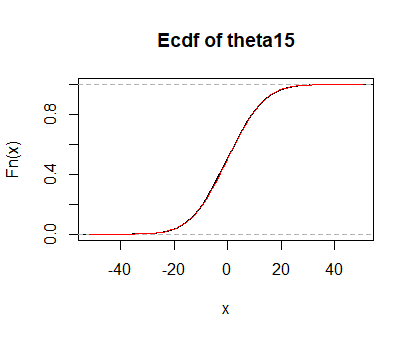}
	\caption{Histogram (left), Q-Q plot (middle) and empirical distribution (right) of $\sqrt{n}(\hat{\theta}_{m,n}^{(i)}-\theta_{m,n}^{(i)})\quad (i=1,\cdots,15)$.}   
	\label{thetafigure2}
\end{figure}
\clearpage
\subsection{Proof of Lemma \ref{Xlemma}}\label{ProofXlemma}
\begin{lemma}\label{Alemma}
Under $\bf{[A1]}$, 
\begin{align*}
	&\quad\E_{\theta_{m}}\left[A_{i,m,n}^{(j_1)}|\mathscr{F}^{n}_{i-1}\right]=R(h_n,\xi_{m,t_{i-1}^n}),\\
	&\quad\E_{\theta_{m}}\left[A_{i,m,n}^{(j_1)}A_{i,m,n}^{(j_2)}|\mathscr{F}^{n}_{i-1}\right]=h_n(\Lambda_{x_1,m}\Sigma_{\xi\xi,m}\Lambda_{x_1,m}^{\top})_{j_1j_2}+R(h_n^2,\xi_{m,t_{i-1}^n}),\\
	&\quad\E_{\theta_{m}}\left[A_{i,m,n}^{(j_1)}A_{i,m,n}^{(j_2)}A_{i,m,n}^{(j_3)}|\mathscr{F}^{n}_{i-1}\right]=R(h_n^2,\xi_{m,t_{i-1}^n}),\\
	&\quad\E_{\theta_{m}}\left[A_{i,m,n}^{(j_1)}A_{i,m,n}^{(j_2)}A_{i,m,n}^{(j_3)}A_{i,m,n}^{(j_4)}|\mathscr{F}^{n}_{i-1}\right]\\
	&=h_n^2\bigl\{(\Lambda_{x_1,m}\Sigma_{\xi\xi,m}\Lambda_{x_1,m}^{\top})_{j_1j_2}(\Lambda_{x_1,m}\Sigma_{\xi\xi,m}\Lambda_{x_1,m}^{\top})_{j_3j_4}+(\Lambda_{x_1,m}\Sigma_{\xi\xi,m}\Lambda_{x_1,m}^{\top})_{j_1j_3}(\Lambda_{x_1,m}\Sigma_{\xi\xi,m}\Lambda_{x_1,m}^{\top})_{j_2j_4}\qquad\\
	&\quad+(\Lambda_{x_1,m}\Sigma_{\xi\xi,m}\Lambda_{x_1,m}^{\top})_{j_1j_4}(\Lambda_{x_1,m}\Sigma_{\xi\xi,m}\Lambda_{x_1,m}^{\top})_{j_2j_3}\bigr\}+R(h_n^3,\xi_{m,t_{i-1}^n})
\end{align*}
for $j_1,j_2,j_3,j_4=1,\cdots,p_1$.
\end{lemma}
\begin{lemma}\label{Blemma}
Under $\bf{[B1]}$,
\begin{align*}
	&\quad\E_{\theta_{m}}\left[B_{i,m,n}^{(j_1)}|\mathscr{F}^{n}_{i-1}\right]=R(h_n,\delta_{m,t_{i-1}^n}),\\
	&\quad\E_{\theta_{m}}\left[B_{i,m,n}^{(j_1)}B_{i,m,n}^{(j_2)}|\mathscr{F}^{n}_{i-1}\right]=h_n(\Sigma_{\delta\delta,m})_{j_1j_2}+R(h_n^2,\delta_{m,t_{i-1}^n}),\\
	&\quad\E_{\theta_{m}}\left[B_{i,m,n}^{(j_1)}B_{i,m,n}^{(j_2)}B_{i,m,n}^{(j_3)}|\mathscr{F}^{n}_{i-1}\right]=R(h_n^2,\delta_{m,t_{i-1}^n}),\\
	&\quad\E_{\theta_{m}}\left[B_{i,m,n}^{(j_1)}B_{i,m,n}^{(j_2)}B_{i,m,n}^{(j_3)}B_{i,m,n}^{(j_4)}|\mathscr{F}^{n}_{i-1}\right]\\
	&=h_n^2\bigl\{(\Sigma_{\delta\delta,m})_{j_1j_2}(\Sigma_{\delta\delta,m})_{j_3j_4}+(\Sigma_{\delta\delta,m})_{j_1j_3}(\Sigma_{\delta\delta,m})_{j_2j_4}+(\Sigma_{\delta\delta,m})_{j_1j_4}(\Sigma_{\delta\delta,m})_{j_2j_3}\bigr\}+R(h_n^3,\delta_{m,t_{i-1}^n})\qquad\quad
\end{align*}
for $j_1,j_2,j_3,j_4=1,\cdots,p_1$.
\end{lemma}
\begin{lemma}\label{Clemma}
Under $\bf{[A1]}$,
\begin{align*}
	&\quad\E_{\theta_{m}}\left[C_{i,m,n}^{(j_1)}|\mathscr{F}^{n}_{i-1}\right]=R(h_n,\xi_{m,t_{i-1}^n}),\\
	&\quad\E_{\theta_{m}}\left[C_{i,m,n}^{(j_1)}C_{i,m,n}^{(j_2)}|\mathscr{F}^{n}_{i-1}\right]=h_n(\Lambda_{x_2,m}\Psi_{m}^{-1}\Gamma_{m}\Sigma_{\xi\xi,m}\Gamma_{m}^{\top}\Psi_{m}^{-1\top}\Lambda_{x_2,m}^{\top})_{j_1j_2}+R(h_n^2,\xi_{m,t_{i-1}^n}),\\
	&\quad\E_{\theta_{m}}\left[C_{i,m,n}^{(j_1)}C_{i,m,n}^{(j_2)}C_{i,m,n}^{(j_3)}|\mathscr{F}^{n}_{i-1}\right]=R(h_n^2,\xi_{m,t_{i-1}^n}),\\
	&\quad\E_{\theta_{m}}\left[C_{i,m,n}^{(j_1)}C_{i,m,n}^{(j_2)}C_{i,m,n}^{(j_3)}C_{i,m,n}^{(j_4)}|\mathscr{F}^{n}_{i-1}\right]\\
	&=h_n^2\bigl\{(\Lambda_{x_2,m}\Psi_{m}^{-1}\Gamma_{m}\Sigma_{\xi\xi,m}\Gamma_{m}^{\top}\Psi_{m}^{-1\top}\Lambda_{x_2,m}^{\top})_{j_1j_2}(\Lambda_{x_2,m}\Psi_{m}^{-1}\Gamma_{m}\Sigma_{\xi\xi,m}\Gamma_{m}^{\top}\Psi_{m}^{-1\top}\Lambda_{x_2,m}^{\top})_{j_3j_4}\\
	&\quad+(\Lambda_{x_2,m}\Psi_{m}^{-1}\Gamma_{m}\Sigma_{\xi\xi,m}\Gamma_{m}^{\top}\Psi_{m}^{-1\top}\Lambda_{x_2,m}^{\top})_{j_1j_3}(\Lambda_{x_2,m}\Psi_{m}^{-1}\Gamma_{m}\Sigma_{\xi\xi,m}\Gamma_{m}^{\top}\Psi_{m}^{-1\top}\Lambda_{x_2,m}^{\top})_{j_2j_4}\\
	&\quad+(\Lambda_{x_2,m}\Psi_{m}^{-1}\Gamma_{m}\Sigma_{\xi\xi,m}\Gamma_{m}^{\top}\Psi_{m}^{-1\top}\Lambda_{x_2,m}^{\top})_{j_1j_4}(\Lambda_{x_2,m}\Psi_{m}^{-1}\Gamma_{m}\Sigma_{\xi\xi,m}\Gamma_{m}^{\top}\Psi_{m}^{-1\top}\Lambda_{x_2,m}^{\top})_{j_2j_3}\bigr\}+R(h_n^3,\xi_{m,t_{i-1}^n})
\end{align*}
for $j_1,j_2,j_3,j_4=1,\cdots,p_2$.
\end{lemma}
\begin{lemma}\label{Dlemma}
Under $\bf{[D1]}$, 
\begin{align*}
	&\quad\E_{\theta_{m}}\left[D_{i,m,n}^{(j_1)}|\mathscr{F}^{n}_{i-1}\right]=R(h_n,\zeta_{m,t_{i-1}^n}),\\
	&\quad\E_{\theta_{m}}\left[D_{i,m,n}^{(j_1)}D_{i,m,n}^{(j_2)}|\mathscr{F}^{n}_{i-1}\right]
	=h_n(\Lambda_{x_2,m}\Psi_{m}^{-1}\Sigma_{\zeta\zeta,m}\Psi_{m}^{-1\top}\Lambda_{x_2,m}^{\top})_{j_1j_2}+R(h_n^2,\zeta_{m,t_{i-1}^n}),\\
	&\quad\E_{\theta_{m}}\left[D_{i,m,n}^{(j_1)}D_{i,m,n}^{(j_2)}D_{i,m,n}^{(j_3)}|\mathscr{F}^{n}_{i-1}\right]=R(h_n^2,\zeta_{m,t_{i-1}^n}),\\
	&\quad\E_{\theta_{m}}\left[D_{i,m,n}^{(j_1)}D_{i,m,n}^{(j_2)}D_{i,m,n}^{(j_3)}D_{i,m,n}^{(j_4)}|\mathscr{F}^{n}_{i-1}\right]\\
	&=h_n^2\bigl\{(\Lambda_{x_2,m}\Psi_{m}^{-1}\Sigma_{\zeta\zeta,m}\Psi_{m}^{-1\top}\Lambda_{x_2,m}^{\top})_{j_1j_2}(\Lambda_{x_2,m}\Psi_{m}^{-1}\Sigma_{\zeta\zeta,m}\Psi_{m}^{-1\top}\Lambda_{x_2,m}^{\top})_{j_3j_4}\\
	&\quad+(\Lambda_{x_2,m}\Psi_{m}^{-1}\Sigma_{\zeta\zeta,m}\Psi_{m}^{-1\top}\Lambda_{x_2,m}^{\top})_{j_1j_3}(\Lambda_{x_2,m}\Psi_{m}^{-1}\Sigma_{\zeta\zeta,m}\Psi_{m}^{-1\top}\Lambda_{x_2,m}^{\top})_{j_2j_4}\\
	&\quad+(\Lambda_{x_2,m}\Psi_{m}^{-1}\Sigma_{\zeta\zeta,m}\Psi_{m}^{-1\top}\Lambda_{x_2,m}^{\top})_{j_1j_4}(\Lambda_{x_2,m}\Psi_{m}^{-1}\Sigma_{\zeta\zeta,m}\Psi_{m}^{-1\top}\Lambda_{x_2,m}^{\top})_{j_2j_3}\bigr\}+R(h_n^3,\zeta_{m,t_{i-1}^n})\qquad\qquad\quad
\end{align*}
for $j_1,j_2,j_3,j_4=1,\cdots,p_2$.
\end{lemma}
\begin{lemma}\label{Elemma}
Under $\bf{[C1]}$, 
\begin{align*}
	&\quad\E_{\theta_{m}}\left[E_{i,m,n}^{(j_1)}|\mathscr{F}^{n}_{i-1}\right]=R(h_n,\varepsilon_{m,t_{i-1}^n}),\\
	&\quad\E_{\theta_{m}}\left[E_{i,m,n}^{(j_1)}E_{i,m,n}^{(j_2)}|\mathscr{F}^{n}_{i-1}\right]=h_n(\Sigma_{\varepsilon\varepsilon,m})_{j_1j_2}+R(h_n^2,\varepsilon_{m,t_{i-1}^n}),\\
	&\quad\E_{\theta_{m}}\left[E_{i,m,n}^{(j_1)}E_{i,m,n}^{(j_2)}E_{i,m,n}^{(j_3)}|\mathscr{F}^{n}_{i-1}\right]=R(h_n^2,\varepsilon_{m,t_{i-1}^n}),\\
	&\quad\E_{\theta_{m}}\left[E_{i,m,n}^{(j_1)}E_{i,m,n}^{(j_2)}E_{i,m,n}^{(j_3)}E_{i,m,n}^{(j_4)}|\mathscr{F}^{n}_{i-1}\right]\\
	&=h_n^2\bigl\{(\Sigma_{\varepsilon\varepsilon,m})_{j_1j_2}(\Sigma_{\varepsilon\varepsilon,m})_{j_3j_4}+(\Sigma_{\varepsilon\varepsilon,m})_{j_1j_3}(\Sigma_{\varepsilon\varepsilon,m})_{j_2j_4}+(\Sigma_{\varepsilon\varepsilon,m})_{j_1j_4}(\Sigma_{\varepsilon\varepsilon,m})_{j_2j_3}\bigr\}+R(h_n^3,\varepsilon_{m,t_{i-1}^n})\qquad\quad
\end{align*}
for $j_1,j_2,j_3,j_4=1,\cdots,p_2$.
\end{lemma}
\begin{lemma}\label{AClemma}
Under $\bf{[A1]}$, 
\begin{align*}
	\E_{\theta_{m}}\left[A_{i,m,n}^{(j_1)}C_{i,m,n}^{(j_2)}|\mathscr{F}^{n}_{i-1}\right]
	=h_n(\Lambda_{x_1,m}\Sigma_{\xi\xi,m}\Gamma_{m}^{\top}\Psi^{-1\top}_{m}\Lambda_{x_2,m}^{\top})_{j_1j_2}+R(h_n^2,\xi_{m,t_{i-1}^n}),\qquad\qquad\qquad\qquad\ 
\end{align*}
for $j_1=1,\cdots,p_1$, $j_2=1,\cdots,p_2$,
\begin{align*}
	\E_{\theta_{m}}\left[A_{i,m,n}^{(j_1)}A_{i,m,n}^{(j_2)}C_{i,m,n}^{(j_3)}|\mathscr{F}^{n}_{i-1}\right]=R(h_n^2,\xi_{m,t_{i-1}^n})\qquad\qquad\qquad\qquad\qquad\qquad\qquad\qquad\qquad\qquad\qquad
\end{align*}
for $j_1,j_2=1,\cdots,p_1$, $j_3=1,\cdots,p_2$,
\begin{align*}
	\E_{\theta_{m}}\left[A_{i,m,n}^{(j_1)}C_{i,m,n}^{(j_2)}C_{i,m,n}^{(j_3)}|\mathscr{F}^{n}_{i-1}\right]=R(h_n^2,\xi_{m,t_{i-1}^n})\qquad\qquad\qquad\qquad\qquad\qquad\qquad\qquad\qquad\qquad\qquad
\end{align*}
for $j_1=1,\cdots,p_1$, $j_2,j_3=1,\cdots,p_2$,
\begin{align*}
	&\quad\E_{\theta_{m}}\left[A_{i,m,n}^{(j_1)}A_{i,m,n}^{(j_2)}A_{i,m,n}^{(j_3)}C_{i,m,n}^{(j_4)}|\mathscr{F}^{n}_{i-1}\right]\\
	&=h_n^2\bigl\{(\Lambda_{x_1,m}\Sigma_{\xi\xi,m}\Lambda_{x_1,m}^{\top})_{j_1j_2}(\Lambda_{x_1,m}\Sigma_{\xi\xi,m}\Gamma_{m}^{\top}\Psi^{-1\top}_{m}\Lambda_{x_2,m}^{\top})_{j_3j_4}\\
	&\quad+(\Lambda_{x_1,m}\Sigma_{\xi\xi,m}\Lambda_{x_1,m}^{\top})_{j_1j_3}(\Lambda_{x_1,m}\Sigma_{\xi\xi,m}\Gamma_{m}^{\top}\Psi^{-1\top}_{m}\Lambda_{x_2,m}^{\top})_{j_2j_4}\\
	&\quad+(\Lambda_{x_1,m}\Sigma_{\xi\xi,m}\Gamma_{m}^{\top}\Psi^{-1\top}_{m}\Lambda_{x_2,m}^{\top})_{j_1j_4}(\Lambda_{x_1,m}\Sigma_{\xi\xi,m}\Lambda_{x_1,m}^{\top})_{j_2j_3}\bigr\}+R(h_n^3,\xi_{m,t_{i-1}^n})\qquad\qquad\qquad\qquad\qquad\quad\ 
\end{align*}
for $j_1,j_2,j_3=1,\cdots,p_1$, $j_4=1,\cdots,p_2$,
\begin{align*}
	&\quad\E_{\theta_{m}}\left[A_{i,m,n}^{(j_1)}A_{i,m,n}^{(j_2)}C_{i,m,n}^{(j_3)}C_{i,m,n}^{(j_4)}|\mathscr{F}^{n}_{i-1}\right]\\
	&=h_n^2\bigl\{(\Lambda_{x_1,m}\Sigma_{\xi\xi,m}\Lambda_{x_1,m}^{\top})_{j_1j_2}(\Lambda_{x_2,m}\Psi^{-1}_{m}\Gamma_{m}\Sigma_{\xi\xi,m}\Gamma_{m}^{\top}\Psi^{-1\top}_{m}\Lambda_{x_2,m}^{\top})_{j_3j_4}\\
	&\quad+(\Lambda_{x_1,m}\Sigma_{\xi\xi,m}\Gamma_{m}^{\top}\Psi^{-1\top}_{m}\Lambda_{x_2,m}^{\top})_{j_1j_3}(\Lambda_{x_1,m}\Sigma_{\xi\xi,m}\Gamma_{m}^{\top}\Psi^{-1\top}_{m}\Lambda_{x_2,m}^{\top})_{j_2j_4}\\
	&\quad+(\Lambda_{x_1,m}\Sigma_{\xi\xi,m}\Gamma_{m}^{\top}\Psi^{-1\top}_{m}\Lambda_{x_2,m}^{\top})_{j_1j_4}(\Lambda_{x_1,m}\Sigma_{\xi\xi,m}\Gamma_{m}^{\top}\Psi^{-1\top}_{m}\Lambda_{x_2,m}^{\top})_{j_2j_3}\bigr\}+R(h_n^3,\xi_{m,t_{i-1}^n})\qquad\qquad\qquad\qquad
\end{align*}
for $j_1,j_2=1,\cdots,p_1$, $j_3,j_4=1,\cdots,p_2$, and
\begin{align*}
	&\quad\E_{\theta_{m}}\left[A_{i,m,n}^{(j_1)}C_{i,m,n}^{(j_2)}C_{i,m,n}^{(j_3)}C_{i,m,n}^{(j_4)}|\mathscr{F}^{n}_{i-1}\right]\\
	&=h_n^2\bigl\{(\Lambda_{x_1,m}\Sigma_{\xi\xi,m}\Gamma_{m}^{\top}\Psi^{-1\top}_{m}\Lambda_{x_2,m}^{\top})_{j_1j_2}(\Lambda_{x_2,m}\Psi_{m}^{-1}\Gamma_m\Sigma_{\xi\xi,m}\Gamma_{m}^{\top}\Psi^{-1\top}_{m}\Lambda_{x_2,m}^{\top})_{j_3j_4}\\
	&\quad+(\Lambda_{x_1,m}\Sigma_{\xi\xi,m}\Gamma_{m}^{\top}\Psi^{-1\top}_{m}\Lambda_{x_2,m}^{\top})_{j_1j_3}(\Lambda_{x_2,m}\Psi_{m}^{-1}\Gamma_m\Sigma_{\xi\xi,m}\Gamma_{m}^{\top}\Psi^{-1\top}_{m}\Lambda_{x_2,m}^{\top})_{j_2j_4}\\
	&\quad+(\Lambda_{x_1,m}\Sigma_{\xi\xi,m}\Gamma_{m}^{\top}\Psi^{-1\top}_{m}\Lambda_{x_2,m}^{\top})_{j_1j_4}(\Lambda_{x_2,m}\Psi_{m}^{-1}\Gamma_m\Sigma_{\xi\xi,m}\Gamma_{m}^{\top}\Psi^{-1\top}_{m}\Lambda_{x_2,m}^{\top})_{j_2j_3}\bigr\}+R(h_n^3,\xi_{m,t_{i-1}^n})\qquad\qquad
\end{align*}
for $j_1,j_2,j_3,j_4=1,\cdots,p_2$.
\end{lemma}
\begin{proof}[Proofs of Lemmas \ref{Alemma}-\ref{AClemma}.]
The results can be shown in a similar way to Lemmas 2-3 in Kusano and Uchida \cite{Kusano(2022)}.
\end{proof}
\begin{lemma}\label{EX1X1lemma}
Under $\bf{[A1]}$ and $\bf{[B1]}$, 
\begin{align}
	\begin{split}
	&\quad\E_{\theta_{m}}\left[(X_{1,t_{i}^n}^{(j_1)}-X_{1,t_{i-1}^n}^{(j_1)})(X_{1,t_{i}^n}^{(j_2)}-X_{1,t_{i-1}^n}^{(j_2)})|\mathscr{F}^{n}_{i-1}\right]\\
	&=h_n(\Sigma_{X_1X_1,m}(\theta_{m}))_{j_1j_2}+h_n^2\bigl\{R(1,\xi_{m,t_{i-1}^n})+R(1,\delta_{m,t_{i-1}^n})+R(1,\xi_{m,t_{i-1}^n})R(1,\delta_{m,t_{i-1}^n})\bigr\},\qquad\label{EX1X1}
	\end{split}\\
	\begin{split}
	&\quad\E_{\theta_{m}}\left[(X_{1,t_{i}^n}^{(j_1)}-X_{1,t_{i-1}^n}^{(j_1)})(X_{1,t_{i}^n}^{(j_2)}-X_{1,t_{i-1}^n}^{(j_2)})(X_{1,t_{i}^n}^{(j_3)}-X_{1,t_{i-1}^n}^{(j_3)})(X_{1,t_{i}^n}^{(j_4)}-X_{1,t_{i-1}^n}^{(j_4)})|\mathscr{F}^{n}_{i-1}\right]\\
	&=h_n^2\bigl\{(\Sigma_{X_1X_1,m}(\theta_{m}))_{j_1j_2}(\Sigma_{X_1X_1,m}(\theta_{m}))_{j_3j_4}+(\Sigma_{X_1X_1,m}(\theta_{m}))_{j_1j_3}(\Sigma_{X_1X_1,m}(\theta_{m}))_{j_2j_4}\\
	&\quad+(\Sigma_{X_1X_1,m}(\theta_{m}))_{j_1j_4}(\Sigma_{X_1X_1,m}(\theta_{m}))_{j_2j_3}\bigr\}+h_n^3\bigl\{R(1,\xi_{m,t_{i-1}^n})+R(1,\delta_{m,t_{i-1}^n})\bigr\}\qquad\qquad\qquad\qquad\\
	&\quad+h_n^4R(1,\xi_{m,t_{i-1}^n})R(1,\delta_{m,t_{i-1}^n})\label{EX1X1X1X1}
	\end{split}
\end{align}
for $j_1,j_2,j_3,j_4=1,\cdots,p_1$.
\end{lemma}
\begin{proof}
From Lemma \ref{Alemma} and Lemma \ref{Blemma}, 
the results can be shown in a similar way to Lemma 4 in Kusano and Uchida \cite{Kusano(2022)}.
\end{proof} 
\begin{lemma}\label{EX2X2lemma}
Under $\bf{[A1]}$, $\bf{[C1]}$ and $\bf{[D1]}$,
\begin{align}
	\begin{split}
	&\quad\ \E_{\theta_{m}}\left[(X_{2,t_{i}^n}^{(j_1)}-X_{2,t_{i-1}^n}^{(j_1)})(X_{2,t_{i}^n}^{(j_2)}-X_{2,t_{i-1}^n}^{(j_2)})|\mathscr{F}^{n}_{i-1}\right]\\
	&=h_n(\Sigma_{X_2X_2,m}(\theta_{m}))_{j_1j_2}+h_n^2\bigl\{R(1,\xi_{m,t_{i-1}^n})+R(1,\varepsilon_{m,t_{i-1}^n})+R(1,\zeta_{m,t_{i-1}^n})+R(1,\xi_{m,t_{i-1}^n})\qquad\quad \\
	&\quad\times R(1,\varepsilon_{m,t_{i-1}^n})+R(1,\xi_{m,t_{i-1}^n})R(1,\zeta_{m,t_{i-1}^n})+R(1,\varepsilon_{m,t_{i-1}^n})R(1,\zeta_{m,t_{i-1}^n})\bigr\},\label{EX2X2}
	\end{split}\\
	\begin{split}
	&\quad\ \E_{\theta_{m}}\left[(X_{2,t_{i}^n}^{(j_1)}-X_{2,t_{i-1}^n}^{(j_1)})(X_{2,t_{i}^n}^{(j_2)}-X_{2,t_{i-1}^n}^{(j_2)})(X_{2,t_{i}^n}^{(j_3)}-X_{2,t_{i-1}^n}^{(j_3)})(X_{2,t_{i}^n}^{(j_4)}-X_{2,t_{i-1}^n}^{(j_4)})|\mathscr{F}^{n}_{i-1}\right]\\
	&=h_n^2\bigl\{(\Sigma_{X_2X_2,m}(\theta_{m}))_{j_1j_2}(\Sigma_{X_2X_2,m}(\theta_{m}))_{j_3j_4}+(\Sigma_{X_2X_2,m}(\theta_{m}))_{j_1j_3}(\Sigma_{X_2X_2,m}(\theta_{m}))_{j_2j_4}\\
	&\quad+(\Sigma_{X_2X_2,m}(\theta_{m}))_{j_1j_4}(\Sigma_{X_2X_2,m}(\theta_{m}))_{j_2j_3}\bigr\}+h_n^3\bigl\{R(1,\xi_{m,t_{i-1}^n})+R(1,\varepsilon_{m,t_{i-1}^n})+R(1,\zeta_{m,t_{i-1}^n})\qquad\\
	&\quad+R(1,\xi_{m,t_{i-1}^n})R(1,\varepsilon_{m,t_{i-1}^n})+R(1,\xi_{m,t_{i-1}^n})R(1,\zeta_{m,t_{i-1}^n})+R(1,\varepsilon_{m,t_{i-1}^n})R(1,\zeta_{m,t_{i-1}^n})\\
	&\quad+R(1,\xi_{m,t_{i-1}^n})R(1,\varepsilon_{m,t_{i-1}^n})R(1,\zeta_{m,t_{i-1}^n})\bigr\}\label{EX2X2X2X2}
    \end{split}
\end{align}
for $j_1,j_2,j_3,j_4=1,\cdots,p_2$.
\end{lemma}
\begin{proof}
First, consider (\ref{EX2X2}). We see from 
Lemmas \ref{Clemma}-\ref{Elemma} 
that
\begin{align*}
	&\quad\E_{\theta_{m}}\left[C_{i,m,n}^{(j_1)}C_{i,m,n}^{(j_2)}+D_{i,m,n}^{(j_1)}D_{i,m,n}^{(j_2)}+E_{i,m,n}^{(j_1)}E_{i,m,n}^{(j_2)}|\mathscr{F}^{n}_{i-1}\right]\\
	&=h_n\bigl\{(\Lambda_{x_2,m}\Psi_{m}^{-1}\Gamma_{m}\Sigma_{\xi\xi,m}\Gamma_{m}^{\top}\Psi_{m}^{-1\top}\Lambda_{x_2,m}^{\top})_{j_1j_2}+(\Lambda_{x_2,m}\Psi_{m}^{-1}\Sigma_{\zeta\zeta,m}\Psi_{m}^{-1\top}\Lambda_{x_2,m}^{\top})_{j_1j_2}+(\Sigma_{\varepsilon\varepsilon,m})_{j_1j_2}\bigr\}\\
	&\quad+R(h_n^2,\xi_{m,t_{i-1}^n})+R(h_n^2,\varepsilon_{m,t_{i-1}^n})+R(h_n^2,\zeta_{m,t_{i-1}^n})\\
	&=h_n(\Sigma_{X_2X_2,m}(\theta_{m}))_{j_1j_2}+h_n^2\bigl\{R(1,\xi_{m,t_{i-1}^n})+R(1,\varepsilon_{m,t_{i-1}^n})+R(1,\zeta_{m,t_{i-1}^n})\bigr\}
\end{align*}
for $j_1,j_2=1,\cdots,p_2$. In addition, 
it follows from 
Lemmas \ref{Clemma}-\ref{Elemma}
and the independence of $\xi_{m,t}$, $\varepsilon_{m,t}$, and $\zeta_{m,t}$ that
\begin{align*}
	\E_{\theta_{m}}\left[C_{i,m,n}^{(j_1)}D_{i,m,n}^{(j_2)}|\mathscr{F}^{n}_{i-1}\right]&=\E_{\theta_{m}}\left[C_{i,m,n}^{(j_1)}|\mathscr{F}^{n}_{i-1}\right]\E_{\theta_{m}}\left[D_{i,m,n}^{(j_2)}|\mathscr{F}^{n}_{i-1}\right]=R(h_n,\xi_{m,t_{i-1}^n})R(h_n,\zeta_{m,t_{i-1}^n}),\\
	\E_{\theta_{m}}\left[C_{i,m,n}^{(j_1)}E_{i,m,n}^{(j_2)}|\mathscr{F}^{n}_{i-1}\right]&=\E_{\theta_{m}}\left[C_{i,m,n}^{(j_1)}|\mathscr{F}^{n}_{i-1}\right]\E_{\theta_{m}}\left[E_{i,m,n}^{(j_2)}|\mathscr{F}^{n}_{i-1}\right]=R(h_n,\xi_{m,t_{i-1}^n})R(h_n,\varepsilon_{m,t_{i-1}^n}),\\
	\E_{\theta_{m}}\left[D_{i,m,n}^{(j_1)}E_{i,m,n}^{(j_2)}|\mathscr{F}^{n}_{i-1}\right]&=\E_{\theta_{m}}\left[D_{i,m,n}^{(j_1)}|\mathscr{F}^{n}_{i-1}\right]\E_{\theta_{m}}\left[E_{i,m,n}^{(j_2)}|\mathscr{F}^{n}_{i-1}\right]=R(h_n,\zeta_{m,t_{i-1}^n})R(h_n,\varepsilon_{m,t_{i-1}^n})
\end{align*}
for $j_1,j_2=1,\cdots,p_2$.  
Therefore, we obtain
\begin{align*}
	&\quad\ \E_{\theta_{m}}\left[ (X^{(j_1)}_{2,t_{i}^n}-X^{(j_1)}_{2,t_{i-1}^n})(X^{(j_2)}_{2,t_{i}^n}-X^{(j_2)}_{2,t_{i-1}^n})|\mathscr{F}^{n}_{i-1}\right]\\
	&=\E_{\theta_{m}}\left[(C_{i,m,n}^{(j_1)}+D_{i,m,n}^{(j_1)}+E_{i,m,n}^{(j_1)})(C_{i,m,n}^{(j_2)}+D_{i,m,n}^{(j_2)}+E_{i,m,n}^{(j_2)})|\mathscr{F}^{n}_{i-1}\right]\\
	&=h_n(\Sigma_{X_2X_2,m}(\theta_{m}))_{j_1j_2}+h_n^2\bigl\{R(1,\xi_{m,t_{i-1}^n})+R(1,\varepsilon_{m,t_{i-1}^n})+R(1,\zeta_{m,t_{i-1}^n})+R(1,\xi_{m,t_{i-1}^n})R(1,\varepsilon_{m,t_{i-1}^n})\\
	&\quad+R(1,\xi_{m,t_{i-1}^n})R(1,\zeta_{m,t_{i-1}^n})+R(1,\varepsilon_{m,t_{i-1}^n})R(1,\zeta_{m,t_{i-1}^n})\bigr\}
\end{align*}
for $j_1,j_2=1,\cdots,p_2$. Furthermore, from 
Lemmas \ref{Clemma}-\ref{Elemma},
(\ref{EX2X2X2X2}) 
can be shown in the same way.
\end{proof}
\begin{lemma}\label{EX1X2lemma}
Under $\bf{[A1]}$, $\bf{[B1]}$, $\bf{[C1]}$ and $\bf{[D1]}$, 
\begin{align*}
	&\quad\ \E_{\theta_{m}}\left[ (X^{(j_1)}_{1,t_{i}^n}-X^{(j_1)}_{1,t_{i-1}^n})(X^{(j_2)}_{2,t_{i}^n}-X^{(j_2)}_{2,t_{i-1}^n})| \mathscr{F}^{n}_{i-1}\right]\\
	&=h_n(\Sigma_{X_1X_2,m}(\theta_{m}))_{j_1j_2}+h_n^2\bigl\{R(1,\xi_{m,t_{i-1}^n})+R(1,\xi_{m,t_{i-1}^n})R(1,\delta_{m,t_{i-1}^n})+R(1,\xi_{m,t_{i-1}^n})R(1,\varepsilon_{m,t_{i-1}^n})\qquad\quad\\
	&\quad+R(1,\xi_{m,t_{i-1}^n})R(1,\zeta_{m,t_{i-1}^n})+R(1,\delta_{m,t_{i-1}^n})R(1,\zeta_{m,t_{i-1}^n})+R(1,\delta_{m,t_{i-1}^n})R(1,\varepsilon_{m,t_{i-1}^n})\bigr\}
\end{align*}
for $j_1=1,\cdots,p_1,\ j_2=1,\cdots,p_2$,
\begin{align*}
	&\quad\ \E_{\theta_{m}}\left[(X_{1,t_{i}^n}^{(j_1)}-X_{1,t_{i-1}^n}^{(j_1)})(X_{1,t_{i}^n}^{(j_2)}-X_{1,t_{i-1}^n}^{(j_2)})(X_{1,t_{i}^n}^{(j_3)}-X_{1,t_{i-1}^n}^{(j_3)})(X_{2,t_{i}^n}^{(j_4)}-X_{2,t_{i-1}^n}^{(j_4)})|\mathscr{F}^{n}_{i-1}\right]\\
	&=h_n^2\bigl\{(\Sigma_{X_1X_1,m}(\theta_{m}))_{j_1j_2}(\Sigma_{X_1X_2,m}(\theta_{m}))_{j_3j_4}+(\Sigma_{X_1X_1,m}(\theta_{m}))_{j_1j_3}(\Sigma_{X_1X_2,m}(\theta_{m}))_{j_2j_4}\\
	&\quad+(\Sigma_{X_1X_2,m}(\theta_{m}))_{j_1j_4}(\Sigma_{X_1X_1,m}(\theta_{m}))_{j_2j_3}\bigr\}+h_n^3\bigl\{R(1,\xi_{m,t_{i-1}^n})+R(1,\delta_{m,t_{i-1}^n})+R(1,\xi_{m,t_{i-1}^n})R(1,\delta_{m,t_{i-1}^n})\\
	&\quad+R(1,\xi_{m,t_{i-1}^n})R(1,\varepsilon_{m,t_{i-1}^n})+R(1,\xi_{m,t_{i-1}^n})R(1,\zeta_{m,t_{i-1}^n})+R(1,\delta_{m,t_{i-1}^n})R(1,\varepsilon_{m,t_{i-1}^n})+R(1,\delta_{m,t_{i-1}^n})\\
	&\quad\times R(1,\zeta_{m,t_{i-1}^n})+R(1,\xi_{m,t_{i-1}^n})R(1,\delta_{m,t_{i-1}^n})R(1,\varepsilon_{m,t_{i-1}^n})+R(1,\xi_{m,t_{i-1}^n})R(1,\delta_{m,t_{i-1}^n})R(1,\zeta_{m,t_{i-1}^n})\bigr\}.
\end{align*}
for $j_1,j_2,j_3=1,\cdots,p_1,\ j_4=1,\cdots,p_2$,
\begin{align*}
	&\quad\ \E_{\theta_{m}}\left[(X_{1,t_{i}^n}^{(j_1)}-X_{1,t_{i-1}^n}^{(j_1)})(X_{1,t_{i}^n}^{(j_2)}-X_{1,t_{i-1}^n}^{(j_2)})(X_{2,t_{i}^n}^{(j_3)}-X_{2,t_{i-1}^n}^{(j_3)})(X_{2,t_{i}^n}^{(j_4)}-X_{2,t_{i-1}^n}^{(j_4)})|\mathscr{F}^{n}_{i-1}\right]\\
	&=h_n^2\bigl\{(\Sigma_{X_1X_1,m}(\theta_{m}))_{j_1j_2}(\Sigma_{X_2X_2,m}(\theta_{m}))_{j_3j_4}+(\Sigma_{X_1X_2,m}(\theta_{m}))_{j_1j_3}(\Sigma_{X_1X_2,m}(\theta_{m}))_{j_2j_4}\\
	&\quad+(\Sigma_{X_1X_2,m}(\theta_{m}))_{j_1j_4}(\Sigma_{X_1X_2,m}(\theta_{m}))_{j_2j_3}\bigr\}+h_n^3\bigl\{R(1,\xi_{m,t_{i-1}^n})+R(1,\delta_{m,t_{i-1}^n})+R(1,\varepsilon_{m,t_{i-1}^n})\\
	&\quad+R(1,\zeta_{m,t_{i-1}^n})+R(1,\xi_{m,t_{i-1}^n})R(1,\delta_{m,t_{i-1}^n})+R(1,\xi_{m,t_{i-1}^n})R(1,\varepsilon_{m,t_{i-1}^n})+R(1,\xi_{m,t_{i-1}^n})R(1,\zeta_{m,t_{i-1}^n})\\
	&\quad+R(1,\delta_{m,t_{i-1}^n})R(1,\varepsilon_{m,t_{i-1}^n}) +R(1,\delta_{m,t_{i-1}^n})R(1,\zeta_{m,t_{i-1}^n})+R(1,\xi_{m,t_{i-1}^n})R(1,\delta_{m,t_{i-1}^n})R(1,\varepsilon_{m,t_{i-1}^n})\\
	&\quad+R(1,\xi_{m,t_{i-1}^n})R(1,\delta_{m,t_{i-1}^n})R(1,\zeta_{m,t_{i-1}^n})+R(1,\xi_{m,t_{i-1}^n})R(1,\varepsilon_{m,t_{i-1}^n})R(1,\zeta_{m,t_{i-1}^n})+R(1,\delta_{m,t_{i-1}^n})\qquad\quad\\
	&\quad\times R(1,\varepsilon_{m,t_{i-1}^n})R(1,\zeta_{m,t_{i-1}^n})\bigr\}+h_n^4R(1,\xi_{m,t_{i-1}^n})R(1,\delta_{m,t_{i-1}^n})R(1,\zeta_{m,t_{i-1}^n})R(1,\varepsilon_{m,t_{i-1}^n})
\end{align*}
for $j_1,j_2=1,\cdots,p_1, \ j_3,j_4=1,\cdots,p_2$, and
\begin{align*}
	&\quad\ \E_{\theta_{m}}\left[(X_{1,t_{i}^n}^{(j_1)}-X_{1,t_{i-1}^n}^{(j_1)})(X_{2,t_{i}^n}^{(j_2)}-X_{2,t_{i-1}^n}^{(j_2)})(X_{2,t_{i}^n}^{(j_3)}-X_{2,t_{i-1}^n}^{(j_3)})(X_{2,t_{i}^n}^{(j_4)}-X_{2,t_{i-1}^n}^{(j_4)})|\mathscr{F}^{n}_{i-1}\right]\\
	&=h_n^2\bigl\{(\Sigma_{X_1X_2,m}(\theta_{m}))_{j_1j_2}(\Sigma_{X_2X_2,m}(\theta_{m}))_{j_3j_4}+(\Sigma_{X_1X_2,m}(\theta_{m}))_{j_1j_3}(\Sigma_{X_2X_2,m}(\theta_{m}))_{j_2j_4}\\
	&\quad+(\Sigma_{X_1X_2,m}(\theta_{m}))_{j_1j_4}(\Sigma_{X_2X_2,m}(\theta_{m}))_{j_2j_3}\bigr\}+h_n^3\bigl\{R(1,\xi_{m,t_{i-1}^n})+R(1,\varepsilon_{m,t_{i-1}^n})+R(1,\zeta_{m,t_{i-1}^n})\\
	&\quad+R(1,\xi_{m,t_{i-1}^n})R(1,\delta_{m,t_{i-1}^n})+R(1,\xi_{m,t_{i-1}^n})R(1,\varepsilon_{m,t_{i-1}^n})+R(1,\xi_{m,t_{i-1}^n})R(1,\zeta_{m,t_{i-1}^n})\\
	&\quad+R(1,\delta_{m,t_{i-1}^n})R(1,\varepsilon_{m,t_{i-1}^n})+R(1,\delta_{m,t_{i-1}^n})R(1,\zeta_{m,t_{i-1}^n})+R(1,\xi_{m,t_{i-1}^n})R(1,\delta_{m,t_{i-1}^n})R(1,\varepsilon_{m,t_{i-1}^n})\\
	&\quad+R(1,\xi_{m,t_{i-1}^n})R(1,\delta_{m,t_{i-1}^n})R(1,\zeta_{m,t_{i-1}^n})+R(1,\xi_{m,t_{i-1}^n})R(1,\varepsilon_{m,t_{i-1}^n})R(1,\zeta_{m,t_{i-1}^n})+R(1,\delta_{m,t_{i-1}^n})\qquad\\
	&\quad\times R(1,\varepsilon_{m,t_{i-1}^n})R(1,\zeta_{m,t_{i-1}^n})\bigr\}+h_n^4 R(1,\xi_{m,t_{i-1}^n})R(1,\delta_{m,t_{i-1}^n})R(1,\varepsilon_{m,t_{i-1}^n})R(1,\zeta_{m,t_{i-1}^n})
\end{align*}
for $j_1=1,\cdots,p_1, \ j_2,j_3,j_4=1,\cdots,p_2$.
\end{lemma}
\begin{proof}
From Lemmas \ref{Alemma}-\ref{AClemma},
the results can be shown in an analogous manner to Lemma \ref{EX2X2lemma}.
\end{proof}
\begin{lemma}\label{klemma}
Under $\bf{[A1]}$, $\bf{[B1]}$, $\bf{[C1]}$, $\bf{[D1]}$ and $\bf{[E1]}$,  
\begin{align}
	\E_{\theta_{m}}\left[\bigl|A_{i,m,n}^{(j)}\bigr|^k|\mathscr{F}^{n}_{i-1}\right]&=R(h_n^\frac{k}{2},\xi_{m,t_{i-1}^n})\quad (j=1,\cdots p_1),\label{Ak}\\
	\E_{\theta_{m}}\left[\bigl|B_{i,m,n}^{(j)}\bigr|^k|\mathscr{F}^{n}_{i-1}\right]&=R(h_n^\frac{k}{2},\delta_{m,t_{i-1}^n})\quad (j=1,\cdots p_1),\label{Bk}\\
	\E_{\theta_{m}}\left[\bigl|C_{i,m,n}^{(j)}\bigr|^k|\mathscr{F}^{n}_{i-1}\right]&=R(h_n^\frac{k}{2},\xi_{m,t_{i-1}^n})\quad (j=1,\cdots p_2),\label{Ck}\\   
	\E_{\theta_{m}}\left[\bigl|D_{i,m,n}^{(j)}\bigr|^k|\mathscr{F}^{n}_{i-1}\right]&=R(h_n^\frac{k}{2}\zeta_{m,t_{i-1}^n}),\quad (j=1,\cdots p_2),\label{Dk}\\    
	\E_{\theta_{m}}\left[\bigl|E_{i,m,n}^{(j)}\bigr|^k|\mathscr{F}^{n}_{i-1}\right]&=R(h_n^\frac{k}{2},\varepsilon_{m,t_{i-1}^n})\quad (j=1,\cdots p_2)\label{Ek}
\end{align}
for $k\geq 2$.
\end{lemma}
\begin{proof}
In an analogous manner to Lemma 6 in Kessler \cite{kessler(1997)}, 
the results can be shown.
\end{proof}
\begin{proof}[Proof of Lemma \ref{Xlemma}.]
First, we prove (\ref{L}). 
Note that $\Sigma_{X_2X_1,m}(\theta_{m})=\Sigma_{X_1X_2,m}(\theta_{m})^{\top}$. 
It is sufficient to show that
\begin{align}
	\begin{split}
	&\ \sum_{i=1}^n\E_{\theta_{m}}\left[\left\{\frac{1}{\sqrt{n}h_n}(X_{1,t_{i}^n}^{(j_1)}-X_{1,t_{i-1}^n}^{(j_1)})(X_{1,t_{i}^n}^{(j_2)}-X_{1,t_{i-1}^n}^{(j_2)})-\frac{1}{\sqrt{n}}(\Sigma_{X_1X_1,m}(\theta_{m}))_{j_1j_2}\right\}|\mathscr{F}^{n}_{i-1}\right]\stackrel{P_{\theta_{m}}\ }{\longrightarrow}0\label{EX1X1prob}
	\end{split}
\end{align}
for all $j_1,j_2=1,\cdots p_1$,
\begin{align}
	\begin{split}
	&\sum_{i=1}^n\E_{\theta_{m}}\left[\left\{\frac{1}{\sqrt{n}h_n}(X_{1,t_{i}^n}^{(j_1)}-X_{1,t_{i-1}^n}^{(j_1)})(X_{2,t_{i}^n}^{(j_2)}-X_{2,t_{i-1}^n}^{(j_2)})-\frac{1}{\sqrt{n}}(\Sigma_{X_1X_2,m}(\theta_{m}))_{j_1j_2}\right\}|\mathscr{F}^{n}_{i-1}\right]\stackrel{P_{\theta_{m}}\ }{\longrightarrow}0\label{EX1X2prob}
	\end{split}
\end{align}
for all $j_1=1,\cdots,p_1,\ j_2=1,\cdots p_2$, and
\begin{align}
	\begin{split}
	&\sum_{i=1}^n\E_{\theta_{m}}\left[\left\{\frac{1}{\sqrt{n}h_n}(X_{2,t_{i}^n}^{(j_1)}-X_{2,t_{i-1}^n}^{(j_1)})(X_{2,t_{i}^n}^{(j_2)}-X_{2,t_{i-1}^n}^{(j_2)})-\frac{1}{\sqrt{n}}(\Sigma_{X_2X_2,m}(\theta_{m}))_{j_1j_2}\right\}|\mathscr{F}^{n}_{i-1}\right]\stackrel{P_{\theta_{m}}\ }{\longrightarrow} 0\label{EX2X2prob}
	\end{split}
\end{align}
for all $j_1,j_2=1,\cdots p_2$. 
From (\ref{EX1X1}), we have
\begin{align*}
	&\quad\sum_{i=1}^n\E_{\theta_{m}}\left[\left\{\frac{1}{\sqrt{n}h_n}(X_{1,t_{i}^n}^{(j_1)}-X_{1,t_{i-1}^n}^{(j_1)})(X_{1,t_{i}^n}^{(j_2)}-X_{1,t_{i-1}^n}^{(j_2)})-\frac{1}{\sqrt{n}}(\Sigma_{X_1X_1,m}(\theta_{m}))_{j_1j_2}\right\}|\mathscr{F}^{n}_{i-1}\right]\\
	&=\frac{1}{\sqrt{n}}\sum_{i=1}^n\left\{\frac{1}{h_n}\E_{\theta_{m}}\left[(X_{1,t_{i}^n}^{(j_1)}-X_{1,t_{i-1}^n}^{(j_1)})(X_{1,t_{i}^n}^{(j_2)}-X_{1,t_{i-1}^n}^{(j_2)})|\mathscr{F}^{n}_{i-1}\right]-(\Sigma_{X_1X_1,m}(\theta_{m}))_{j_1j_2}\right\}\\
	&=\frac{h_n}{\sqrt{n}}\sum_{i=1}^n\bigl\{R(1,\xi_{m,t_{i-1}^n})+R(1,\delta_{m,t_{i-1}^n})+R(1,\xi_{m,t_{i-1}^n})R(1,\delta_{m,t_{i-1}^n})\bigr\}\\
	&=\sqrt{nh_n^2}\ \frac{1}{n}\sum_{i=1}^n\bigl\{R(1,\xi_{m,t_{i-1}^n})+R(1,\delta_{m,t_{i-1}^n})+R(1,\xi_{m,t_{i-1}^n})R(1,\delta_{m,t_{i-1}^n})\bigr\}\stackrel{P_{\theta_{m}}\ }{\longrightarrow} 0\qquad\qquad\qquad\quad
\end{align*}
for $j_1,j_2=1,\cdots,p_1$,  which deduces (\ref{EX1X1prob}). In the same way, we obtain (\ref{EX1X2prob}) and (\ref{EX2X2prob}) from Lemma \ref{EX2X2lemma} and Lemma \ref{EX1X2lemma}, respectively.

Next, we show (\ref{LL}). It is sufficient to prove that
\begin{align}
	\begin{split}
	&\sum_{i=1}^n\E_{\theta_{m}}\left[\left\{\frac{1}{\sqrt{n}h_n}(X_{1,t_{i}^n}^{(j_1)}-X_{1,t_{i-1}^n}^{(j_1)})(X_{1,t_{i}^n}^{(j_2)}-X_{1,t_{i-1}^n}^{(j_2)})-\frac{1}{\sqrt{n}}(\Sigma_{X_1X_1,m}(\theta_{m}))_{j_1j_2}\right\}\right.\\
	&\qquad\qquad\qquad\times\left.\left\{\frac{1}{\sqrt{n}h_n}(X_{1,t_{i}^n}^{(j_3)}-X_{1,t_{i-1}^n}^{(j_3)})(X_{1,t_{i}^n}^{(j_4)}-X_{1,t_{i-1}^n}^{(j_4)})-\frac{1}{\sqrt{n}}(\Sigma_{X_1X_1,m}(\theta_{m}))_{j_3j_4}\right\}|\mathscr{F}^{n}_{i-1}\right]\\
	&\stackrel{P_{\theta_{m}}\ }{\longrightarrow} (\Sigma_{X_1X_1,m}(\theta_{m}))_{j_{1}j_{3}}(\Sigma_{X_1X_1,m}(\theta_{m}))_{j_{2}j_{4}}+(\Sigma_{X_1X_1,m}(\theta_{m}))_{j_{1}j_{4}}(\Sigma_{X_1X_1,m}(\theta_{m}))_{j_{2}j_{3}}\label{X1X1X1X1}
	\end{split}
\end{align}
for $j_1,j_2,j_3,j_4=1,\cdots,p_1$,
\begin{align}
	\begin{split}
	&\sum_{i=1}^n\E_{\theta_{m}}\left[\left\{\frac{1}{\sqrt{n}h_n}(X_{1,t_{i}^n}^{(j_1)}-X_{1,t_{i-1}^n}^{(j_1)})(X_{1,t_{i}^n}^{(j_2)}-X_{1,t_{i-1}^n}^{(j_2)})-\frac{1}{\sqrt{n}}(\Sigma_{X_1X_1,m}(\theta_{m}))_{j_1j_2}\right\}\right.\\
	&\qquad\qquad\qquad\times\left.\left\{\frac{1}{\sqrt{n}h_n}(X_{1,t_{i}^n}^{(j_3)}-X_{1,t_{i-1}^n}^{(j_3)})(X_{2,t_{i}^n}^{(j_4)}-X_{2,t_{i-1}^n}^{(j_4)})-\frac{1}{\sqrt{n}}(\Sigma_{X_1X_2,m}(\theta_{m}))_{j_3j_4}\right\}|\mathscr{F}^{n}_{i-1}\right]\\
	&\stackrel{P_{\theta_{m}}\ }{\longrightarrow} (\Sigma_{X_1X_1,m}(\theta_{m}))_{j_{1}j_{3}}(\Sigma_{X_1X_2,m}(\theta_{m}))_{j_{2}j_{4}}+(\Sigma_{X_1X_2,m}(\theta_{m}))_{j_{1}j_{4}}(\Sigma_{X_1X_2,m}(\theta_{m}))_{j_{2}j_{3}}\label{X1X1X1X2}
	\end{split}
\end{align}
for $j_1,j_2,j_3=1,\cdots,p_1,\ j_4=1,\cdots,p_2$,
\begin{align}
	\begin{split}
	&\sum_{i=1}^n\E_{\theta_{m}}\left[\left\{\frac{1}{\sqrt{n}h_n}(X_{1,t_{i}^n}^{(j_1)}-X_{1,t_{i-1}^n}^{(j_1)})(X_{1,t_{i}^n}^{(j_2)}-X_{1,t_{i-1}^n}^{(j_2)})-\frac{1}{\sqrt{n}}(\Sigma_{X_1X_1,m}(\theta_{m}))_{j_1j_2}\right\}\right.\\
	&\qquad\qquad\qquad\times\left.\left\{\frac{1}{\sqrt{n}h_n}(X_{2,t_{i}^n}^{(j_3)}-X_{2,t_{i-1}^n}^{(j_3)})(X_{2,t_{i}^n}^{(j_4)}-X_{2,t_{i-1}^n}^{(j_4)})-\frac{1}{\sqrt{n}}(\Sigma_{X_2X_2,m}(\theta_{m}))_{j_3j_4}\right\}|\mathscr{F}^{n}_{i-1}\right]\\
	&\stackrel{P_{\theta_{m}}\ }{\longrightarrow} (\Sigma_{X_1X_2,m}(\theta_{m}))_{j_{1}j_{3}}(\Sigma_{X_1X_2,m}(\theta_{m}))_{j_{2}j_{4}}+(\Sigma_{X_1X_2,m}(\theta_{m}))_{j_{1}j_{4}}(\Sigma_{X_1X_2,m}(\theta_{m}))_{j_{2}j_{3}}\label{X1X1X2X2}
	\end{split}
\end{align}
for $j_1,j_2=1,\cdots,p_1,\ j_3,j_4=1,\cdots,p_2$,
\begin{align}
	\begin{split}
	&\sum_{i=1}^n\E_{\theta_{m}}\left[\left\{\frac{1}{\sqrt{n}h_n}(X_{1,t_{i}^n}^{(j_1)}-X_{1,t_{i-1}^n}^{(j_1)})(X_{2,t_{i}^n}^{(j_2)}-X_{2,t_{i-1}^n}^{(j_2)})-\frac{1}{\sqrt{n}}(\Sigma_{X_1X_2,m}(\theta_{m}))_{j_1j_2}\right\}\right.\\
	&\qquad\qquad\qquad\times\left.\left\{\frac{1}{\sqrt{n}h_n}(X_{1,t_{i}^n}^{(j_3)}-X_{1,t_{i-1}^n}^{(j_3)})(X_{2,t_{i}^n}^{(j_4)}-X_{2,t_{i-1}^n}^{(j_4)})-\frac{1}{\sqrt{n}}(\Sigma_{X_1X_2,m}(\theta_{m}))_{j_3j_4}\right\}|\mathscr{F}^{n}_{i-1}\right]\\
	&\stackrel{P_{\theta_{m}}\ }{\longrightarrow} (\Sigma_{X_1X_1,m}(\theta_{m}))_{j_{1}j_{3}}(\Sigma_{X_2X_2,m}(\theta_{m}))_{j_{2}j_{4}}+(\Sigma_{X_1X_2,m}(\theta_{m}))_{j_{1}j_{4}}(\Sigma_{X_2X_1,m}(\theta_{m}))_{j_{2}j_{3}}\label{X1X2X1X2}
	\end{split}
\end{align}
for $j_1,j_3=1,\cdots,p_1,\ j_2,j_4=1,\cdots,p_2$,
\begin{align}
	\begin{split}
	&\sum_{i=1}^n\E_{\theta_{m}}\left[\left\{\frac{1}{\sqrt{n}h_n}(X_{1,t_{i}^n}^{(j_1)}-X_{1,t_{i-1}^n}^{(j_1)})(X_{2,t_{i}^n}^{(j_2)}-X_{2,t_{i-1}^n}^{(j_2)})-\frac{1}{\sqrt{n}}(\Sigma_{X_1X_2,m}(\theta_{m}))_{j_1j_2}\right\}\right.\\
	&\qquad\qquad\qquad\times\left.\left\{\frac{1}{\sqrt{n}h_n}(X_{2,t_{i}^n}^{(j_3)}-X_{2,t_{i-1}^n}^{(j_3)})(X_{2,t_{i}^n}^{(j_4)}-X_{2,t_{i-1}^n}^{(j_4)})-\frac{1}{\sqrt{n}}(\Sigma_{X_2X_2,m}(\theta_{m}))_{j_3j_4}\right\}|\mathscr{F}^{n}_{i-1}\right]\\
	&\stackrel{P_{\theta_{m}}\ }{\longrightarrow} (\Sigma_{X_1X_2,m}(\theta_{m}))_{j_{1}j_{3}}(\Sigma_{X_2X_2,m}(\theta_{m}))_{j_{2}j_{4}}+(\Sigma_{X_1X_2,m}(\theta_{m}))_{j_{1}j_{4}}(\Sigma_{X_2X_2,m}(\theta_{m}))_{j_{2}j_{3}}\label{X1X2X2X2}
	\end{split}
\end{align}
for $j_1=1,\cdots,p_1,\ j_2,j_3,j_4=1,\cdots,p_2$, and
\begin{align}
	\begin{split}
	&\sum_{i=1}^n\E_{\theta_{m}}\left[\left\{\frac{1}{\sqrt{n}h_n}(X_{2,t_{i}^n}^{(j_1)}-X_{2,t_{i-1}^n}^{(j_1)})(X_{2,t_{i}^n}^{(j_2)}-X_{2,t_{i-1}^n}^{(j_2)})-\frac{1}{\sqrt{n}}(\Sigma_{X_2X_2,m}(\theta_{m}))_{j_1j_2}\right\}\right.\\
	&\qquad\qquad\qquad\times\left.\left\{\frac{1}{\sqrt{n}h_n}(X_{2,t_{i}^n}^{(j_3)}-X_{2,t_{i-1}^n}^{(j_3)})(X_{2,t_{i}^n}^{(j_4)}-X_{2,t_{i-1}^n}^{(j_4)})-\frac{1}{\sqrt{n}}(\Sigma_{X_2X_2,m}(\theta_{m}))_{j_3j_4}\right\}|\mathscr{F}^{n}_{i-1}\right]\\
	&\stackrel{P_{\theta_{m}}\ }{\longrightarrow} (\Sigma_{X_2X_2,m}(\theta_{m}))_{j_{1}j_{3}}(\Sigma_{X_2X_2,m}(\theta_{m}))_{j_{2}j_{4}}+(\Sigma_{X_2X_2,m}(\theta_{m}))_{j_{1}j_{4}}(\Sigma_{X_2X_2,m}(\theta_{m}))_{j_{2}j_{3}}\label{X2X2X2X2}
	\end{split}
\end{align}
for $j_1,j_2,j_3,j_4=1,\cdots,p_2$. It holds from Lemma \ref{EX1X1lemma} that
\begin{align*}
	&\quad\ \sum_{i=1}^n\E_{\theta_{m}}\left[\left\{\frac{1}{\sqrt{n}h_n}(X_{1,t_{i}^n}^{(j_1)}-X_{1,t_{i-1}^n}^{(j_1)})(X_{1,t_{i}^n}^{(j_2)}-X_{1,t_{i-1}^n}^{(j_2)})-\frac{1}{\sqrt{n}}(\Sigma_{X_1X_1,m}(\theta_{m}))_{j_1j_2}\right\}\right.\\
	&\left.\qquad\qquad\qquad\qquad\quad\times\left\{\frac{1}{\sqrt{n}h_n}(X_{1,t_{i}^n}^{(j_3)}-X_{1,t_{i-1}^n}^{(j_3)})(X_{1,t_{i}^n}^{(j_4)}-X_{1,t_{i-1}^n}^{(j_4)})-\frac{1}{\sqrt{n}}(\Sigma_{X_1X_1,m}(\theta_{m}))_{j_3j_4}\right\}|\mathscr{F}^{n}_{i-1}\right]\qquad\\
	&=\frac{1}{nh_n^2}\sum_{i=1}^n\E_{\theta_{m}}\left[(X_{1,t_{i}^n}^{(j_1)}-X_{1,t_{i-1}^n}^{(j_1)})(X_{1,t_{i}^n}^{(j_2)}-X_{1,t_{i-1}^n}^{(j_2)})(X_{1,t_{i}^n}^{(j_3)}-X_{1,t_{i-1}^n}^{(j_3)})(X_{1,t_{i}^n}^{(j_4)}-X_{1,t_{i-1}^n}^{(j_4)})|\mathscr{F}^{n}_{i-1}\right]\\
	&\quad -\frac{1}{nh_n}\sum_{i=1}^n\E_{\theta_{m}}\left[(X_{1,t_{i}^n}^{(j_1)}-X_{1,t_{i-1}^n}^{(j_1)})(X_{1,t_{i}^n}^{(j_2)}-X_{1,t_{i-1}^n}^{(j_2)})|\mathscr{F}^{n}_{i-1}\right](\Sigma_{X_1X_1,m}(\theta_{m}))_{j_3j_4}\\
	&\quad -\frac{1}{nh_n}\sum_{i=1}^n\E_{\theta_{m}}\left[(X_{1,t_{i}^n}^{(j_3)}-X_{1,t_{i-1}^n}^{(j_3)})(X_{1,t_{i}^n}^{(j_4)}-X_{1,t_{i-1}^n}^{(j_4)})|\mathscr{F}^{n}_{i-1}\right](\Sigma_{X_1X_1,m}(\theta_{m}))_{j_1j_2}\\
	&\quad+\frac{1}{n}\sum_{i=1}^n(\Sigma_{X_1X_1,m}(\theta_{m}))_{j_1j_2}(\Sigma_{X_1X_1,m}(\theta_{m}))_{j_3j_4}\\
	&=(\Sigma_{X_1X_1,m}(\theta_{m}))_{j_{1}j_{3}}(\Sigma_{X_1X_1,m}(\theta_{m}))_{j_{2}j_{4}}+(\Sigma_{X_1X_1,m}(\theta_{m}))_{j_{1}j_{4}}(\Sigma_{X_1X_1,m}(\theta_{m}))_{j_{2}j_{3}}\\
	&\quad+h_n\times\frac{1}{n}\sum_{i=1}^n\bigl\{R(1,\xi_{m,t_{i-1}^n})+R(1,\delta_{m,t_{i-1}^n})+R(1,\xi_{m,t_{i-1}^n})R(1,\delta_{m,t_{i-1}^n})\bigr\}\\
	&\stackrel{P_{\theta_{m}}\ }{\longrightarrow} (\Sigma_{X_1X_1,m}(\theta_{m}))_{j_{1}j_{3}}(\Sigma_{X_1X_1,m}(\theta_{m}))_{j_{2}j_{4}}+(\Sigma_{X_1X_1,m}(\theta_{m}))_{j_{1}j_{4}}(\Sigma_{X_1X_1,m}(\theta_{m}))_{j_{2}j_{3}}
\end{align*}
for $j_1,j_2,j_3,j_4=1,\cdots,p_1$, which yields (\ref{X1X1X1X1}). 
In an analogous manner, 
Lemmas \ref{EX1X1lemma}-\ref{EX1X2lemma} imply  (\ref{X1X1X1X2})-(\ref{X2X2X2X2}).

Finally, we prove (\ref{L4}). It is sufficient to show that
\begin{align}
	\begin{split}
	&\sum_{i=1}^n\E_{\theta_{m}}\left[\left|\frac{1}{\sqrt{n}h_n}(X^{(j_1)}_{1,t_{i}^n}-X^{(j_1)}_{1,t_{i-1}^n})(X^{(j_2)}_{1,t_{i}^n}-X^{(j_2)}_{1,t_{i-1}^n})-\frac{1}{\sqrt{n}}(\Sigma_{X_1X_1,m}(\theta_{m}))_{j_1j_2}
	\right|^4|\mathscr{F}^{n}_{i-1}\right]\stackrel{P_{\theta_{m}}\ }{\longrightarrow}0\label{EX1X14}
	\end{split}
\end{align}
for $j_1,j_2=1,\cdots p_1$,
\begin{align}
	\begin{split}
	&\sum_{i=1}^n\E_{\theta_{m}}\left[\left|\frac{1}{\sqrt{n}h_n}(X^{(j_1)}_{1,t_{i}^n}-X^{(j_1)}_{1,t_{i-1}^n})(X^{(j_2)}_{2,t_{i}^n}-X^{(j_2)}_{2,t_{i-1}^n})-\frac{1}{\sqrt{n}}(\Sigma_{X_1X_2,m}(\theta_{m}))_{j_1j_2}\right|^4|\mathscr{F}^{n}_{i-1}\right]\stackrel{P_{\theta_{m}}\ }{\longrightarrow}0 \label{EX1X24}
	\end{split}
\end{align}
for $j_1=1,\cdots p_1,\ j_2=1,\cdots p_2$, and
\begin{align}
	\begin{split}
	&\sum_{i=1}^n\E_{\theta_{m}}\left[\left|\frac{1}{\sqrt{n}h_n}(X^{(j_1)}_{2,t_{i}^n}-X^{(j_1)}_{2,t_{i-1}^n})(X^{(j_2)}_{2,t_{i}^n}-X^{(j_2)}_{2,t_{i-1}^n})-\frac{1}{\sqrt{n}}(\Sigma_{X_2X_2,m}(\theta_{m}))_{j_1j_2}
	\right|^4|\mathscr{F}^{n}_{i-1}\right]\stackrel{P_{\theta_{m}}\ }{\longrightarrow}0\label{EX2X24}
	\end{split}
\end{align}
for $j_1,j_2=1,\cdots p_2$. We can evaluate as follows:
\begin{align}
	\begin{split}
	0&\leq\sum_{i=1}^n\E_{\theta_{m}}\left[\left|\frac{1}{\sqrt{n}h_n}(X^{(j_1)}_{1,t_{i}^n}-X^{(j_1)}_{1,t_{i-1}^n})(X^{(j_2)}_{1,t_{i}^n}-X^{(j_2)}_{1,t_{i-1}^n})-\frac{1}{\sqrt{n}}(\Sigma_{X_1X_1,m}(\theta_{m}))_{j_1j_2}
	\right|^4|\mathscr{F}^{n}_{i-1}\right]\\
	&\qquad\qquad\leq\frac{C_{1}}{n^2h_n^4}\sum_{i=1}^n\E_{\theta_{m}}\left[\left|(X^{(j_1)}_{1,t_{i}^n}-X^{(j_1)}_{1,t_{i-1}^n})(X^{(j_2)}_{1,t_{i}^n}-X^{(j_2)}_{1,t_{i-1}^n})\right|^4|\mathscr{F}^{n}_{i-1}\right]+\frac{C_{1}}{n}(\Sigma_{X_1,X_1,m}(\theta_{m}))_{j_1j_2}^4\label{X1X14ine}
	\end{split}
\end{align}
for $j_1,j_2=1,\cdots,p_1$. 
Using Cauchy-Schwartz's inequality and  (\ref{Ak}), 
we have
\begin{align*}
	\E_{\theta_{m}}\left[\left|A^{(j_1)}_{i,n,m}A^{(j_2)}_{i,n,m}\right|^4|\mathscr{F}^{n}_{i-1}\right]&\leq\E_{\theta_{m}}\left[\left|A^{(j_1)}_{i,n,m}\right|^8|\mathscr{F}^{n}_{i-1}\right]^{\frac{1}{2}}\E_{\theta_{m}}\left[\left|A^{(j_2)}_{i,n,m})\right|^8|\mathscr{F}^{n}_{i-1}\right]^{\frac{1}{2}}\\
	&\qquad\qquad\leq R(h_n^4,\xi_{m,t_{i-1}^n})^{\frac{1}{2}}R(h_n^4,\xi_{m,t_{i-1}^n})^{\frac{1}{2}}\leq R(h_n^4,\xi_{m,t_{i-1}^n})
\end{align*}
for $j_1,j_2=1,\cdots,p_1$.
In the same way, from (\ref{Bk}), we obtain  
\begin{align*}
	\E_{\theta_{m}}\left[\left|B^{(j_1)}_{i,n,m}B^{(j_2)}_{i,n,m}\right|^4|\mathscr{F}^{n}_{i-1}\right]\leq R(h_n^4,\delta_{m,t_{i-1}^n})
\end{align*}
for $j_1,j_2=1,\cdots,p_1$.
Furthermore, 
it follows from the independence of $\xi_{m,t}$ and $\delta_{m,t}$,
 (\ref{Ak}) and (\ref{Bk}) that
\begin{align*}
	\E_{\theta_{m}}\left[\left|A^{(j_1)}_{i,n,m}B^{(j_2)}_{i,n,m}\right|^4|\mathscr{F}^{n}_{i-1}\right]&=\E_{\theta_{m}}\left[\left|A^{(j_1)}_{i,n,m}\right|^4|\mathscr{F}^{n}_{i-1}\right]\E_{\theta_{m}}\left[\left|B^{(j_2)}_{i,n,m}\right|^4|\mathscr{F}^{n}_{i-1}\right]\\
	&\leq R(h_n^2,\xi_{m,t_{i-1}^n})R(h_n^2,\delta_{m,t_{i-1}^n})
\end{align*}
for $j_1,j_2=1,\cdots,p_1$.
Thus, 
for $j_1,j_2=1,\cdots,p_1$, one has
\begin{align*}
	0&\leq\frac{C_{1}}{n^2h_n^4}\sum_{i=1}^n\E_{\theta_{m}}\left[\left|(X^{(j_1)}_{1,t_{i}^n}-X^{(j_1)}_{1,t_{i-1}^n})(X^{(j_2)}_{1,t_{i}^n}-X^{(j_2)}_{1,t_{i-1}^n})\right|^4|\mathscr{F}^{n}_{i-1}\right]\\
	&\leq\frac{C_{1}}{n^2h_n^4}\sum_{i=1}^n\E_{\theta_{m}}\left[\left|(A^{(j_1)}_{i,n,m}+B^{(j_1)}_{i,n,m})(A^{(j_2)}_{i,n,m}+B^{(j_2)}_{i,n,m})
	\right|^4|\mathscr{F}^{n}_{i-1}\right]\\
	&\leq\frac{C_{2}}{n^2h_n^4}\sum_{i=1}^n\E_{\theta_{m}}\left[\left|A^{(j_1)}_{i,n,m}A^{(j_2)}_{i,n,m})\right|^4|\mathscr{F}^{n}_{i-1}\right]+\frac{C_{2}}{n^2h_n^4}\sum_{i=1}^n\E_{\theta_{m}}\left[\left|A^{(j_1)}_{i,n,m}B^{(j_2)}_{i,n,m})\right|^4|\mathscr{F}^{n}_{i-1}\right]\\
	&\quad +\frac{C_{2}}{n^2h_n^4}\sum_{i=1}^n\E_{\theta_{m}}\left[\left|B^{(j_1)}_{i,n,m}A^{(j_2)}_{i,n,m})\right|^4|\mathscr{F}^{n}_{i-1}\right]+\frac{C_{2}}{n^2h_n^4}\sum_{i=1}^n\E_{\theta_{m}}\left[\left|B^{(j_1)}_{i,n,m}B^{(j_2)}_{i,n,m})\right|^4|\mathscr{F}^{n}_{i-1}\right]\\
	&\leq\frac{C_{2}}{n}\frac{1}{n}\sum_{i=1}^n R(1,\xi_{m,t_{i-1}^n})+\frac{C_{2}}{n}\frac{1}{n}\sum_{i=1}^n R(1,\xi_{m,t_{i-1}^n})R(1,\delta_{m,t_{i-1}^n})+\frac{C_{2}}{n}\frac{1}{n}\sum_{i=1}^n R(1,\delta_{m,t_{i-1}^n})\stackrel{P_{\theta_{m}}\ }{\longrightarrow}0,
\end{align*}
which yields
\begin{align}
	\frac{C_{1}}{n^2h_n^4}\sum_{i=1}^n\E_{\theta_{m}}\left[\left|(X^{(j_1)}_{1,t_{i}^n}-X^{(j_1)}_{1,t_{i-1}^n})(X^{(j_2)}_{1,t_{i}^n}-X^{(j_2)}_{1,t_{i-1}^n})\right|^4|\mathscr{F}^{n}_{i-1}\right]\stackrel{P_{\theta_{m}}\ }{\longrightarrow}0\label{CX1X14}
\end{align}
for $j_1,j_2=1,\cdots,p_1$. Hence, we obtain (\ref{EX1X14}) from (\ref{X1X14ine}) and (\ref{CX1X14}). In the same way, we can show (\ref{EX1X24}) and (\ref{EX2X24}) from Lemma \ref{klemma}.
\end{proof}
\subsection{Proof of Lemma \ref{Sigmaposlemma}}\label{Sigmaposproof}
\ \\
For any matrix $A$, $\mathbb{C}(A)$ denotes the column space of $A$.
\begin{lemma}\label{positive}
Set $A\in\mathbb{R}^{(p_1+p_2)\times(p_1+p_2)}$ as
\begin{align*}
    A=\begin{pmatrix}
	T & U\\
	U^{\top} & W
\end{pmatrix},
\end{align*}
where $T\in\mathbb{R}^{p_1\times p_1}$, $U\in\mathbb{R}^{p_1\times p_2}$ and $W\in\mathbb{R}^{p_2\times p_2}$. 
\begin{enumerate}[$(i)$]
	\item If $T$ is a positive definite matrix, and $W-U^{\top}T^{-1}U$ is a positive definite matrix, then $A$ is a positive definite matrix.
	\item If $T$ is a semi-positive definite matrix, $W-U^{\top}T^{-}U$ is a semi-positive definite matrix, and $\mathbb{C}(U)\subset \mathbb{C}(T)$, then $A$ is a semi-positive definite matrix.
\end{enumerate}
\end{lemma}
\begin{proof}
See Theorem 14.8.5 in Harville \cite{Harville(1998)}.
\end{proof}
\begin{proof}[Proof of Lemma \ref{Sigmaposlemma}]
We decompose $\Sigma_{m}(\theta_{m})$ as
\begin{align*}
    \Sigma_m(\theta_m)&=\begin{pmatrix}
	T_m & U_{m}\\
	U_{m}^{\top} & W_m
	\end{pmatrix}
	+\begin{pmatrix}
	\Sigma_{\delta\delta,m} & O_{p_1\times p_2}\\
	O_{p_2\times p_1} & \Sigma_{\varepsilon\varepsilon,m}
	\end{pmatrix},
\end{align*}
where 
\begin{align*}
    T_m&=\Lambda_{x_1,m}\Sigma_{\xi\xi,m}\Lambda_{x_1,m}^{\top},\\
    U_m&=\Lambda_{x_1,m}\Sigma_{\xi\xi,m}\Gamma_{m}^{\top}\Psi^{-1\top}_{m}\Lambda_{x_2,m}^{\top},\\
	W_m&=\Lambda_{x_2,m}\Psi^{-1}_{m}\left(\Gamma_m\Sigma_{\xi\xi,m}\Gamma^{\top}_m+\Sigma_{\zeta\zeta,m}\right)\Psi^{-1\top}_{m}\Lambda_{x_2,m}^{\top}.
\end{align*}
Recalling that $\Sigma_{\xi\xi,m}=S_{1,m}S_{1,m}^{\top}$ is a semi-positive definite matrix, one has
\begin{align}
	T_{m}=\bigl(\Lambda_{x_1,m}\Sigma_{\xi\xi,m}^{\frac{1}{2}}\bigr)\bigl(\Lambda_{x_1,m}\Sigma_{\xi\xi,m}^{\frac{1}{2}}\bigr)^{\top}\geq 0.\label{l1-1}
\end{align}
Since it holds from $\bf{[F]}$ that
\begin{align*}
    \Lambda_{x_1,m}^{-}\Lambda_{x_1,m}=\mathbb{I}_{k_{1,m}},\ \ 
    \Lambda_{x_1,m}^{\top}(\Lambda_{x_1,m}^{\top})^{-}=\mathbb{I}_{k_{1,m}},
\end{align*}
we obtain
\begin{align*}
	U_m^{\top} T_m^{-} U_m
	&=\Lambda_{x_2,m}\Psi^{-1}_{m}\Gamma_{m}\Sigma_{\xi\xi,m}\Lambda_{x_1,m}^{\top}(\Lambda_{x_1,m}\Sigma_{\xi\xi,m}\Lambda_{x_1,m}^{\top})^{-}\Lambda_{x_1,m}\Sigma_{\xi\xi,m}\Gamma_{m}^{\top}\Psi^{-1\top}_{m}\Lambda_{x_2,m}^{\top}\\
	&=\Lambda_{x_2,m}\Psi^{-1}_{m}\Gamma_{m}\Lambda_{x_1,m}^{-}\Lambda_{x_1,m}\Sigma_{\xi\xi,m}\Lambda_{x_1,m}^{\top}(\Lambda_{x_1,m}\Sigma_{\xi\xi,m}\Lambda_{x_1,m}^{\top})^{-}\\
	&\qquad\qquad\qquad\qquad\qquad\qquad\qquad\qquad\times\Lambda_{x_1,m}\Sigma_{\xi\xi,m}\Lambda_{x_1,m}^{\top}(\Lambda_{x_1,m}^{\top})^{-}\Gamma_{m}^{\top}\Psi^{-1\top}_{m}\Lambda_{x_2,m}^{\top}\\
	&=\Lambda_{x_2,m}\Psi^{-1}_{m}\Gamma_{m}\Lambda_{x_1,m}^{-}\Lambda_{x,m}\Sigma_{\xi\xi,m}\Lambda_{x_1,m}^{\top}(\Lambda_{x_1,m}^{\top})^{-}\Gamma_{m}^{\top}\Psi^{-1\top}_{m}\Lambda_{x_2,m}^{\top}\\
	&=\Lambda_{x_2,m}\Psi^{-1}_{m}\Gamma_{m}\Sigma_{\xi\xi,m}\Gamma_{m}^{\top}\Psi^{-1\top}_{m}\Lambda_{x_2,m}^{\top}.
\end{align*}
Noting that $\Sigma_{\zeta\zeta,m}=S_{4,m}S_{4,m}^{\top}$ is a semi-positive definite matrix, we have
\begin{align}
	\begin{split}
	W_{m}-U_m^{\top}T_m^{-}U_m
	&=\Lambda_{x_2,m}\Psi^{-1}_{m}\Sigma_{\zeta\zeta,m}\Psi^{-1\top}_{m}\Lambda_{x_2,m}^{\top}\\
	&=\bigl(\Lambda_{x_2,m}\Psi^{-1}_{m}\Sigma_{\zeta\zeta,m}^{\frac{1}{2}})(\Lambda_{x_2,m}\Psi^{-1}_{m}\Sigma_{\zeta\zeta,m}^{\frac{1}{2}}\bigr)^{\top}\geq 0.
	\end{split}\label{l1-2}
\end{align}
Furthermore, we set 
\begin{align*}
    F_{m}=(\Lambda_{x_1,m}^{\top})^{-}\Gamma_{m}^{\top}\Psi^{-1\top}_{m}\Lambda_{x_2,m}^{\top},
\end{align*}
which yields
\begin{align*}
    U_m=T_{m}F_{m}.
\end{align*}
Thus, it follows from Lemma 4.2.2 in Harville \cite{Harville(1998)} that
\begin{align}
	\mathbb{C}(U_m)\subset \mathbb{C}(T_m).\label{l1-3}
\end{align}
Hence, 
Lemma \ref{positive} (ii), (\ref{l1-1}), (\ref{l1-2}) and (\ref{l1-3}) imply that
\begin{align}
	\begin{pmatrix}
	T_m & U_{m}\\
	U_{m}^{\top} & W_m
	\end{pmatrix}\geq 0 .\label{l1-4}
\end{align}
Since it follows from $\bf{[C2]}$ that
\begin{align*}
	\Sigma_{\varepsilon\varepsilon,m}-O_{p_1\times p_2}^{\top}\Sigma_{\delta\delta,m}^{-1}O_{p_1\times p_2}=\Sigma_{\varepsilon\varepsilon,m}>0,
\end{align*}
we see  from Lemma \ref{positive} (i) and $\bf{[B2]}$ that
\begin{align}
	\begin{pmatrix}
	\Sigma_{\delta\delta,m} & O_{p_1\times p_2}\\
	O_{p_2\times p_1} & \Sigma_{\varepsilon\varepsilon,m}
	\end{pmatrix}>0. \label{l1-5}
\end{align}
Therefore, from (\ref{l1-4}) and (\ref{l1-5}), we obtain 
\begin{align*}
	\Sigma_{m}(\theta_{m})=\begin{pmatrix}
	T_m & U_{m}\\
    U_{m}^{\top} & W_m
	\end{pmatrix}+     
	\begin{pmatrix}
	\Sigma_{\delta\delta,m} & O_{p_1\times p_2}\\
	O_{p_2\times p_1} & \Sigma_{\varepsilon\varepsilon,m}
	\end{pmatrix}>0.
\end{align*}
\end{proof}
\subsection{Proof of Lemma \ref{Vposlemma}}\label{Vposproof}
\begin{proof}[Proof of Lemma \ref{Vposlemma}]
Since $X$ and $Y$ are positive definite matrices and one has 
\begin{align*}
    Y+\lambda_{1}\lambda_{2}(X-Y)= \left\{
    \begin{array}{ll}
    Y & (\lambda_{1}\lambda_{2}=0), \\
    \lambda_{1}\lambda_{2}X+(1-\lambda_{1}\lambda_{2})Y & (0< \lambda_{1}\lambda_{2}< 1), \\
    X & (\lambda_{1}\lambda_{2}=1),
    \end{array}
\right.
\end{align*}
it holds that
\begin{align*}
	Y+\lambda_{1}\lambda_{2}(X-Y)>0
\end{align*}
for $\lambda_{1},\lambda_{2}\in[0,1]$.
Noting that
\begin{align*}
	\bigl\{Y+\lambda_{1}\lambda_{2}(X-Y)\bigr\}^{-1}\otimes\bigl\{Y+\lambda_{1}\lambda_{2}(X-Y)\bigr\}^{-1}>0
\end{align*}
for $\lambda_{1},\lambda_{2}\in[0,1]$ and
\begin{align*}
	\mathbb{D}_{p}^{+}x=0\Longleftrightarrow x=0
\end{align*}
for $x\in\mathbb{R}^{\bar{p}}(\neq 0)$, 
one has
\begin{align*}
	x^{\top}\mathbb{D}_{p}^{+\top}\bigl\{Y+\lambda_{1}\lambda_{2}(X-Y)\bigr\}^{-1}\otimes\bigl\{Y+\lambda_{1}\lambda_{2}(X-Y)\bigr\}^{-1}\mathbb{D}_{p}^{+}x>0
\end{align*}
for $\lambda_{1},\lambda_{2}\in[0,1]$ and $x\in\mathbb{R}^{\bar{p}}(\neq 0)$. If $\lambda_{2}$ is not zero, we see
\begin{align*}
	\lambda_{2}x^{\top}\mathbb{D}_{p}^{+\top}\bigl\{Y+\lambda_{1}\lambda_{2}(X-Y)\bigr\}^{-1}\otimes\bigl\{Y+\lambda_{1}\lambda_{2}(X-Y)\bigr\}^{-1}\mathbb{D}_{p}^{+}x>0
\end{align*}
for $\lambda_{1},\lambda_{2}\in[0,1]$ and $x\in\mathbb{R}^{\bar{p}}(\neq 0)$. 
Therefore, we obtain
\begin{align*}
x^{\top}V(X,Y)x=\int_{0}^{1}\int_{0}^{1}\lambda_{2}x^{\top}\mathbb{D}_{p}^{+\top}\bigl\{Y+\lambda_{1}\lambda_{2}(X-Y)\bigr\}^{-1}\otimes\bigl\{Y+\lambda_{1}\lambda_{2}(X-Y)\bigr\}^{-1}\mathbb{D}_{p}^{+}x d\lambda_{1}d\lambda_{2}>0
\end{align*}
for $x\in\mathbb{R}^{\bar{p}}(\neq 0)$. \end{proof}
\subsection{Proof of Lemma \ref{Fproblemma}}\label{Fprobproof}
\begin{lemma}\label{suplemma}
Let $f:\mathbb{R}^{p\times q}\times \mathbb{R}^{r}\rightarrow \mathbb{R}$ denote a continuous function
and $A$ be a compact subset set of $\mathbb{R}^{r}$. Then,
\begin{align*}
	\sup_{\alpha\in A}\left|f(Y,\alpha)-f(Y_0,\alpha)\right|\longrightarrow 0
\end{align*}
as $Y\longrightarrow Y_0$.
\end{lemma}
\begin{proof}
For all $\varepsilon>0$, there exists $\alpha_0\in\mathbb{R}^{r}$ such that
\begin{align}
	\sup_{\alpha\in A}\left|f(Y,\alpha)-f(Y_0,\alpha)\right|-\varepsilon<   \left|f(Y,\alpha_0)-f(Y_0,\alpha_0)\right|\leq\sup_{\alpha\in A}\left|f(Y,\alpha)-f(Y_0,\alpha)\right|.\label{supproof1}
\end{align}
By the continuity of $f$,
there exists $\delta>0$ such that
\begin{align}
	\|Y-Y_0\|<\delta\Longrightarrow \left|f(Y,\alpha_0)-f(Y_0,\alpha_0)\right|<\varepsilon. \label{supproof2}
\end{align}
Therefore, we see 
from (\ref{supproof1}) and (\ref{supproof2})
that
\begin{align*}
	\|Y-Y_0\|<\delta\Longrightarrow \sup_{\alpha\in A}\left|f(Y,\alpha)-f(Y_0,\alpha)\right|<2\varepsilon,
\end{align*}
which implies
\begin{align*}
	\sup_{\alpha\in A}\left|f(Y,\alpha)-f(Y_0,\alpha)\right|\longrightarrow 0 
\end{align*}
as $Y\longrightarrow Y_0$.
\end{proof}
\begin{proof}[Proof of Lemma \ref{Fproblemma}]
Set 
\begin{align*}
    J_n=\Bigl\{Q_{ZZ}\ \mbox{is non-singular}\Bigr\}.
\end{align*}
Since $F$ is continuous in $\theta_{m}$, 
from Lemma \ref{suplemma}, 
for any $\varepsilon>0$, there exists $\delta>0$ such that
\begin{align*}
	&
	\bigl\|Q_{XX}-\Sigma_{m}(\theta_{m,0})\bigr\|<\delta  
	\Longrightarrow \sup_{\theta_{m}\in\Theta_{m}}\bigl|F(Q_{XX},\Sigma_{m}(\theta_{m}))-F(\Sigma_{m}(\theta_{m,0}),\Sigma_{m}(\theta_{m}))\bigr|<\varepsilon
\end{align*}
on $J_{n}$.
Therefore, one has
\begin{align}
	\begin{split}
	0&\leq\PP_{\theta_{m,0}}\left(\Bigl\{\bigl\|Q_{XX}-\Sigma_{m}(\theta_{m,0})\bigr\|<\delta\Bigr\}\cap J_n\right)\\
	&\leq\PP_{\theta_{m,0}}\left(\left\{\sup_{\theta_{m}\in\Theta_{m}}\bigl|F(Q_{XX},\Sigma_{m}(\theta_{m}))-F(\Sigma_{m}(\theta_{m,0}),\Sigma_{m}(\theta_{m}))\bigr|<\varepsilon\right\}\cap J_n\right)\\
	&\leq\PP_{\theta_{m,0}}\left(\left\{\sup_{\theta_{m}\in\Theta_{m}}\bigl|\tilde{F}(Q_{XX},\Sigma_{m}(\theta_{m}))-F(\Sigma_{m}(\theta_{m,0}),\Sigma_{m}(\theta_{m}))\bigr|<\varepsilon\right\}\cap J_n\right)\\
	&\leq\PP_{\theta_{m,0}}\left(\sup_{\theta_{m}\in\Theta_{m}}\bigl|\tilde{F}(Q_{XX},\Sigma_{m}(\theta_{m}))-F(\Sigma_{m}(\theta_{m,0}),\Sigma_{m}(\theta_{m}))\bigr|<\varepsilon\right).\label{Qine}
	\end{split}
\end{align}
Since we see from  Lemma \ref{Sigmaposlemma} 
that $\Sigma_{m}(\theta_{m,0})$ is non-singular,
it holds from Theorem \ref{Qztheorem} that 
\begin{align*}
	\PP_{\theta_{m,0}}\bigl(J_{n}\bigr)\stackrel{}{\longrightarrow}1
\end{align*}
as $n\longrightarrow\infty$. Thus, 
from Theorem \ref{Qztheorem}, we obtain 
\begin{align*}
	0&\leq\PP_{\theta_{m,0}}\left(\left\{\Bigl\{
	\bigl\|Q_{XX}-\Sigma_{m}(\theta_{m,0})\bigr\|<\delta\Bigr\}\cap J_n\right\}^{c}\right)\\
	&=\PP_{\theta_{m,0}}\left(\Bigl\{
	\bigl\|Q_{XX}-\Sigma_{m}(\theta_{m,0})\bigr\|\geq\delta\Bigr\}\cup J_n^{c}\right)\\
	&\leq\PP_{\theta_{m,0}}\Bigl(
	\bigl\|Q_{XX}-\Sigma_{m}(\theta_{m,0})\bigr\|\geq\delta\Bigr)+\PP_{\theta_{m,0}}\bigl(J_n^c\bigr)\stackrel{}{\longrightarrow}0
\end{align*}
as $n\longrightarrow\infty$, which yields
\begin{align*}
	\PP_{\theta_{m,0}}\left(\Bigl\{\bigl\|Q_{XX}-\Sigma_{m}(\theta_{m,0})\bigr\|<\delta\Bigr\}\cap J_n\right)\stackrel{}{\longrightarrow}1. 
\end{align*}
Hence, it follows from (\ref{Qine}) that for all $\varepsilon>0$,
\begin{align*}
	\PP_{\theta_{m,0}}\left(\sup_{\theta_{m}\in\Theta_{m}}\bigl|\tilde{F}(Q_{XX},\Sigma_{m}(\theta_{m}))-F(\Sigma_{m}(\theta_{m,0}),\Sigma_{m}(\theta_{m}))\bigr|<\varepsilon\right)\stackrel{}{\longrightarrow}1,
\end{align*}
which implies 
\begin{align*}
	\sup_{\theta_{m}\in\Theta_{m}}\bigl|\tilde{F}(Q_{XX},\Sigma_{m}(\theta_{m}))-F(\Sigma_{m}(\theta_{m,0}),\Sigma_{m}(\theta_{m}))\bigr|\stackrel{P_{\theta_{m,0}}\ }{\longrightarrow}0.
\end{align*}
Therefore, we obtain (\ref{Fprob}). Furthermore, we can show (\ref{F2prob}) in the same way.
\end{proof}
\subsection{Proof of Lemma \ref{starconslemma}}\label{Proofstarcons}
\begin{proof}[Proof of Lemma \ref{starconslemma}]
From $\bf{[I]}$,
for any $\varepsilon>0$, there exists $\delta>0$ such that
\begin{align}
	\bigl|\hat{\theta}_{m^*,n}-\bar{\theta}_{m^*}\bigr|>\varepsilon\Longrightarrow \mathbb{U}_{m^*}(\hat{\theta}_{m^*,n})-\mathbb{U}_{m^*}(\bar{\theta}_{m^*})>\delta.\label{assumptionI}
\end{align}
It holds from the definition of $\hat{\theta}_{m^*,n}$ that
\begin{align*}
    \tilde{F}(Q_{XX},\Sigma_{m^*}(\hat{\theta}_{m^*,n}))\leq\tilde{F}(Q_{XX},\Sigma_{m^*}(\bar{\theta}_{m^*})).
\end{align*}
We obtain  from Lemma \ref{Fproblemma} and (\ref{assumptionI}) that  
for all $\varepsilon>0$,
\begin{align*}
	0&\leq \PP\left(\bigl|\hat{\theta}_{m^*,n}-\bar{\theta}_{m^*}\bigr|>\varepsilon\right)\\
	&\leq\PP\Bigl(\mathbb{U}_{m^*}(\hat{\theta}_{m^*,n})-\mathbb{U}_{m^*}(\bar{\theta}_{m^*})>\delta\Bigr)\\
	&\leq\PP\left(F(\Sigma_{m}(\theta_{m,0}),\Sigma_{m^*}(\hat{\theta}_{m^*,n}))-\tilde{F}(Q_{XX},\Sigma_{m^*}(\hat{\theta}_{m^*,n}))>\frac{\delta}{3}\right)\\
	&\quad+\PP\left(\tilde{F}(Q_{XX},\Sigma_{m^*}(\hat{\theta}_{m^*,n}))-\tilde{F}(Q_{XX},\Sigma_{m^*}(\bar{\theta}_{m^*}))>\frac{\delta}{3}\right)\\
	&\quad+\PP\left(\tilde{F}(Q_{XX},\Sigma_{m^*}(\bar{\theta}_{m^*}))-F(\Sigma_{m}(\theta_{m,0}),\Sigma_{m^*}(\bar{\theta}_{m^*}))>\frac{\delta}{3}\right)\\
	&\leq 2\PP\left(\sup_{\theta_{m^*}\in\Theta_{m^*}}\left|\tilde{F}(Q_{XX},\Sigma_{m^*}(\theta_{m^*}))-F(\Sigma_{m}(\theta_{m,0}),\Sigma_{m^*}(\theta_{m^*}))\right|>\frac{\delta}{3}\right)+0\stackrel{}{\longrightarrow}0
\end{align*}
under $H_1$ as $n\longrightarrow\infty$,
which implies 
\begin{align*}
	\hat{\theta}_{m^*,n}\stackrel{P}{\longrightarrow}\bar{\theta}_{m^*}
\end{align*}
under $H_1$.
\end{proof}
\end{document}